%% file: paper.tex
\documentclass[12pt]{article}

\RequirePackage{amsthm,amsmath,amsfonts,amssymb}
\RequirePackage{graphicx,subcaption}
\usepackage{latexsym,amssymb,mathrsfs,comment}
\usepackage{bm}
\usepackage{caption}
\usepackage{enumitem}
\usepackage{url, placeins}
\usepackage{algorithmic, float}
\usepackage{booktabs}

\usepackage[T1]{fontenc}
\usepackage[utf8]{inputenc}
\usepackage[english]{babel}
\usepackage[font=small, labelfont = bf]{caption}
\usepackage{graphicx}
\usepackage{algorithm}
\usepackage{algorithmic}
\usepackage {graphicx,color}
\usepackage{capt-of}
\usepackage{booktabs}
\usepackage{varwidth}
\usepackage{comment}
\usepackage{lineno}
\usepackage{amssymb}
\usepackage{amsmath}
\usepackage{amsbsy}
\usepackage{setspace}
\usepackage[authoryear]{natbib}
\usepackage{bm}
\usepackage{textcomp}
\usepackage{amsthm}
\usepackage{hyperref}
\usepackage{blindtext}

\usepackage[T1]{fontenc}
	\usepackage[utf8]{inputenc}
	\usepackage[english]{babel}
	\usepackage[font=small, labelfont = bf]{caption}
	\usepackage{graphicx}
	\usepackage{algorithm}
	\usepackage{algorithmic}
	\usepackage {graphicx,color}
	\usepackage{capt-of}
	\usepackage{booktabs}
	\usepackage{varwidth}
	\usepackage{comment}
	\usepackage{lineno}
	\usepackage{amssymb}
	\usepackage{amsmath}
	\usepackage{amsbsy}
	\usepackage{setspace}
\newcommand{\Var}{\mbox{\rm var}}
\newcommand{\Cov}{\mbox{\rm cov}}
\newcommand{\cum}{\mbox{\rm cum}}

\def \R{\mathbb{R}}
\def \IR{\mathbb{R}}
\def \E{\mathbb{E}}

\def \IN{\mathbb{N}}

\def \IZ{\mathbb{Z}}

\newcommand{\bea}{\begin{eqnarray}}
\newcommand{\ena}{\end{eqnarray}}
\newcommand{\beq}{\begin{equation}}
\newcommand{\enq}{\end{equation}}
\newcommand{\beas}{\begin{eqnarray*}}
	\newcommand{\enas}{\end{eqnarray*}}

\def\numberlikeadb{\global\def\theequation{\thesection.\arabic{equation}}}
\numberlikeadb

\theoremstyle{plain}

\newtheorem{theorem}{Theorem}[section]
\newtheorem{lemma}[theorem]{Lemma}

\newtheorem{proposition}[theorem]{Proposition}
\newtheorem{remark}[theorem]{Remark}
\newtheorem{assumption}[theorem]{Assumption}

\theoremstyle{remark}

\newcommand{\vO}[2]{{\color{black}{\bm #1}(#2)}}

\newcommand{\cO}[3]{{\color{black}#1_{#2}(#3)}}

\newcommand{\cP}[3]{{\color{black}{#1}_{#2}(#3)}}
\newcommand{\cPpre}[3]{{\color{black}\hat #1^*_{#2}(#3)}}
\newcommand{\cPest}[3]{{\color{black}\hat #1_{#2}(#3)}}

\newcommand{\seq}[1]{{\color{black}\{#1\}}}

\makeatletter
\newcommand\code{\bgroup\@makeother\_\@makeother\%\@makeother\~\@makeother\$\@codex}
\def\@codex#1{{\normalfont\ttfamily\hyphenchar\font=-1 #1}\egroup}
\makeatother
\let\proglang=\textsf

\usepackage{color} 

\usepackage{lscape}
\usepackage{caption}
\usepackage{multirow}

\begin{document}
\onehalfspacing

%\begin{frontmatter}
\title{Wasserstein distance bounds on the normal approximation of empirical autocovariances and cross-covariances under non-stationarity and stationarity}

%\runtitle{Clustered signal testing}
%\thankstext{T1}{This work is partially supported by University of Missouri Research Board grant and National Science Foundation of USA}
\author{Andreas Anastasiou$^1$ and Tobias Kley$^2$}

\footnotetext[1]{~ Department of Mathematics and Statistics, University of Cyprus, P.O. Box: 20537, 1678, Nicosia, Cyprus.}
\footnotetext[2]{~ Institute for Mathematical Stochastics,
	Georg-August-University of G{\"o}ttingen, Goldschmidtstra{\ss}e 7, 37077 G{\"o}ttingen, Germany.}
% \newline \indent
%\ \ \ \ The research of Tony Cai was supported in part by NSF Grant
%DMS-0604954.}
%\footnotetext[3]{~ Department of Biostatistics, University of North Carolina at Chapel Hill, Chapel Hill, NC, 27514.}
%
\date{}
\maketitle

\begin{abstract}
The autocovariance and cross-covariance functions naturally appear in many time series procedures (e.\,g., autoregression or prediction). Under assumptions, empirical versions of the autocovariance and cross-covariance are asymptotically normal with covariance structure depending on the second and fourth order spectra. Under non-restrictive assumptions, we derive a bound for the Wasserstein distance of the finite sample distribution of the estimator of the autocovariance and cross-covariance to the Gaussian limit. An error of approximation to the second-order moments of the estimator and an $m$-dependent approximation are the key ingredients in order to obtain the bound. As a worked example, we discuss how to compute the bound for causal autoregressive processes of order 1 with different distributions for the innovations. To assess our result, we compare our bound to Wasserstein distances obtained via simulation.
\end{abstract}


\section{Introduction}
\label{sec:Introduction}
Assessing the quality of various asymptotic results has attracted a lot of interest in recent years. One way to measure the error in distributional approximations is to consider explicit upper bounds on the Wasserstein distance between the limiting and the actual distribution of the quantity of interest; to derive such bounds is undoubtedly a technically tedious task.

We consider the empirical autocovariance and cross-covariance
%and derive such bounds explicitly
\begin{enumerate}
	\item without assuming stationarity, and
	\item for the case of weakly stationary time series.
\end{enumerate}
Our aim is to facilitate a bound where the rate, but also explicit constants can be computed for a wide range of time series models. 

We consider the case where a $d$-variate time series $\vO{X}{1}, \ldots, \vO{X}{n}$ is available, i.\,e., $\vO{X}{t}$ are $\mathbb{R}^{d}$-valued, $t=1,\ldots,n$. The components of $\vO{X}{t}$ are denoted by $\cO{X}{a}{t}$, $a=1,\ldots,d$. We are interested in the empirical cross-covariance and autocovariance, defined as
\begin{equation}\label{def:gam_tilde}
\cPpre{\gamma}{ab}{k} :=
\frac{1}{n} \sum\limits_{t=1}^{n-k} (\cO{X}{a}{t+k} - \bar X_a) (\cO{X}{b}{t} - \bar X_b), \quad k=0,\ldots,n-1,\\
\end{equation}
where $\bar X_j := \frac{1}{n} \sum_{t=1}^n \cO{X}{j}{t}$, $j=a,b$, denotes the empirical mean. For $k=-n+1,\ldots,-1$ we define $\cPpre{\gamma}{ab}{k} := \cPpre{\gamma}{ba}{-k}$.
Other definitions, that are asymptotically equivalent under regularity conditions, also exist in the literature. For example, see \cite{Anderson1971}, Chapter~8, for some common variants in the case of the autocovariance and in particular Corollary~8.4.1 in  \cite{Anderson1971} for a result asserting that these variants converge to the same Gaussian limit, under specific regularity conditions. %Autocovariances are, for example, important for forecasting procedures \citep{Kley_etal2019}.

In the case of stationary data where the population means are known we may substitute the empirical means in~\eqref{def:gam_tilde} by their population counterparts as below
\begin{equation*}
\frac{1}{n} \sum\limits_{t=1}^{n-k} (\cO{X}{a}{t+k} - \E \cO{X}{a}{t+k}) (\cO{X}{b}{t} - \E \cO{X}{b}{t}), \quad k=0,\ldots,n-1.\\
\end{equation*}
This corresponds to assuming that $\seq{\vO{X}{t}}$ is centered (i.\,e., $\E \vO{X}{t} = 0$), and working with the following definition of the empirical cross-covariance:
\begin{equation}\label{def:gam_hat}
\cPest{\gamma}{ab}{k} := \frac{1}{n} \sum_{t=1}^{n-k} \cO{X}{a}{t+k} \cO{X}{b}{t}, \quad k=0,\ldots,n-1,
\end{equation}
and $\cPest{\gamma}{ab}{k} := \cPest{\gamma}{ba}{-k}$, $k=-n+1,\ldots,-1$.
%Related empirical quantities, that are often used in practice, are the empirical cross-correlation functions, which are defined as
%\[\cPest{\rho}{ab}{k} := \frac{\cPest{\gamma}{ab}{k}}{\sqrt{\cPest{\gamma}{aa}{0} \cPest{\gamma}{bb}{0}}}.\]
Autocovariances and cross-covariances are important for many time series methods; e.\,g., autoregression~\citep{Jirak2012,Jirak2014} and forecasting~\citep{BrockwellDavis2006,Kley_etal2019}.

Under conditions, it can be shown that $\cPpre{\gamma}{ab}{k}$ and $\cPest{\gamma}{ab}{k}$ are consistent estimates for\begin{equation}
\label{gamma_ab}
\cP{\gamma}{ab}{k} := \E[\cO{X}{a}{t+k} \cO{X}{b}{t}].
\end{equation}
The asymptotic normality for the distribution of the estimator holds as well. We have that
\begin{equation}
\label{eqn:weak_conv_hat_gamma}
\sqrt{n} \big(\cPpre{\gamma}{ab}{k} - \cP{\gamma}{ab}{k}\big) \xrightarrow[n \rightarrow \infty]{} N, \quad N \sim  \mathcal{N}(0, \Sigma_{ab}(k));
\end{equation}
see, for example, Exercise~7.10.36 in \cite{Brillinger1975}. The asymptotic variance
\begin{equation}
\label{def:Sigma}
\cP{\Sigma}{ab}{k} := \lim_{n \rightarrow \infty} \Var\left(n^{-1/2} \sum_{t=1}^{n-k} \cO{X}{a}{t+k} \cO{X}{b}{t} \right)
%= \sum_{u=-\infty}^{\infty} \Big(
%\cP{\gamma}{a b}{u + k} \cP{\gamma}{b a}{u - k}
%+ \cP{\gamma}{a a}{u} \cP{\gamma}{b b}{u} + \cP{c}{a b a b}{u + k, u, k} \Big).
\end{equation}
depends on the second and fourth order moment structure of the underlying data; cf. eq.~(7.6.11) in~\cite{Brillinger1975}.
It is usually straightforward to compute $\cP{\Sigma}{ab}{k}$.
Details for the case of an AR(1) time series that we consider in Section~\ref{sec:expl} are provided in Section \ref{app:tech_details:AR1_expl} of the online supplement.
To prove~\eqref{eqn:weak_conv_hat_gamma}, it is common practice to make assumptions limiting the intensity of the dependence structure and the moments of the random variables involved (such as the summability of cumulants); cf.~\cite{HorvathKokoszka2008} for the discussion of cases where normality fails.

%In this paper, we obtain upper bounds on a distributional distance between the distribution of the empirical autocovariance and cross-covariance function and the limiting normal distribution.
We now provide the general framework and notation used throughout the paper.
For $\mathbb{R}^d$-valued random vectors $\boldsymbol{U}$ and $\boldsymbol{V}$, we work with the 1-Wasserstein metric defined as
\begin{equation}
\label{classes_functions}
d_{{\mathrm{W}}}(\mathcal{L}(\boldsymbol{U}),\mathcal{L}(\boldsymbol{V})):=\sup_{h\in\mathcal{H}}|\mathbb{E}[h(\boldsymbol{U})]-\mathbb{E}[h(\boldsymbol{V})]|, \quad
\mathcal{H}=\{h:\mathbb{R}^d\rightarrow\mathbb{R}\,|\,\text{$\|h\|_{\mathrm{Lip}}\leq1$}\},
\end{equation}
where $\mathcal{L}(\boldsymbol{U})$ is the law of $\boldsymbol{U}$. Furthermore, for any vector $\boldsymbol{x} = (x_1, \ldots, x_d)$, we denote its Euclidean norm by $|\boldsymbol{x}| := (\sum_{i=1}^d x_i^2)^{1/2}$ and $\|h\|_{\mathrm{Lip}}=\sup_{\boldsymbol{u}\not=\boldsymbol{v}}|h(\boldsymbol{u})-h(\boldsymbol{v})|/|\boldsymbol{u}-\boldsymbol{v}|$.
In this paper, we refer short to the distance in \eqref{classes_functions} as the Wasserstein distance.
%In \eqref{classes_functions}, $\mathcal{H}_{\mathrm{K}}$ gives the Kolmogorov distance and $\mathcal{H}_{\mathrm{W}}$ leads to the 1-Wasserstein metric.
% In this paper, we focus on the latter and refer to it as the Wasserstein metric from now on; it is denoted by $d_{{\mathrm{W}}}$.
%The aforementioned metrics are denoted by $d_{\mathrm{K}}$ and $d_{\mathrm{W}}$, respectively.
The main purpose of the paper is to assess the quality of the distributional approximation in~\eqref{eqn:weak_conv_hat_gamma} through upper bounds on the Wasserstein distance between the actual distribution of the quantity of interest on the left-hand side of \eqref{eqn:weak_conv_hat_gamma} and its limiting normal distribution; for centered data, this is achieved in Theorems~\ref{general_theorem_nonstat} and \ref{general_theorem} for the case of a non-stationary or weakly stationary sequence, respectively. Combining Theorem \ref{general_theorem} with Lemma~\ref{lem:diff_pre_est} to bound the Wasserstein distance between $\sqrt{n} \big(\cPpre{\gamma}{ab}{k} - \cP{\gamma}{ab}{k}\big)$ and $\sqrt{n} \big(\cPest{\gamma}{ab}{k} - \cP{\gamma}{ab}{k}\big)$ in the stationary case, we obtain a bound when the data are non-centered; details can be found in Sections~\ref{sec:bound_original_data} and~\ref{app:non_central}.

%Departing from the case where the quantity of interest is a sum of random variables, \cite{anastasiou_reinert17} were the first to obtain optimal-order upper bounds for the normal approximation of the Maximum Likelihood Estimator (MLE) for the case of independent and identically distributed (i.i.d.) random variables under natural regularity conditions. The MLE is, in general, not a linear function and the strategy followed in order to obtain bounds  consisted in splitting, through Taylor expansions, the quantity of interest (the standardised MLE) into two terms with the first one being a linear statistic of the score function. Benefiting from this special linear form, Stein's method, which is a powerful probabilistic technique first introduced in \cite{Stein1972}, was employed to bound the first term, while the second term was bounded using alternative techniques, such as Taylor expansions, conditional expectations, and known probability inequalities. A series of papers followed and explicit upper bounds quantifying the quality of the normal approximation of the MLE have been obtained under different strategies and settings. In the recent contribution of \cite{Anastasiou_mdependence}, the independence assumption is relaxed and the normal approximation of the MLE is assessed under the presence of a local dependence structure between the random variables. These results on dependent variables can lead to a whole new area of applications; assessing the quality of asymptotic results in time series.

Our approach depends on the existence of an $m$-dependent sequence, which allows us to use Stein's method, a powerful probabilistic technique first introduced in \cite{Stein1972}, under a local dependence structure. Stein's method is particularly powerful in assessing whether a given random variable has a distribution close to a target distribution in the presence of such dependence structures between the random variables. The bounds obtained through Stein's method are explicit in terms of the constants and in terms of the sample size; see for example \cite{Anastasiou_mdependence}, where bounds for the normal approximation of the maximum likelihood estimator are provided under a local dependence structure between the random variables.
%This advantage of Stein's method over other known techniques, such as characteristic functions, that can give answers on distributional distances, is what makes it suitable under the dependence setting promoted in the current paper.

%The Berry-Esseen type bounds for the Kolmogorov distance are among the most popular contributions in the area of assessing the quality of normal approximations related to general (mainly linear) statistics; see for example \cite{Koroljuk_1994} for the case of U-statistics and \cite{Erickson} for sums of $m$-dependent random variables. The concept of $m$-dependence is quite crucial in the results presented in the paper and it is thoroughly explained in Section~\ref{sec:local_dep}.
There has been a lot of interest recently on the assessment of the quality of the normal approximation related to the sum $\sum_{i=1}^{n}X_i$, where $X_1, \ldots, X_n$ are centered and follow a specific dependence structure.
While at first sight, it seems that the empirical autocovariance and cross-covariance fit into this framework (replace $n$ by $n-k$ and $X_i$ by $n^{-1}\cO{X}{a}{t+k} \cO{X}{b}{t}$), the results in the literature for $\sum_{i=1}^{n}X_i$ do not immediately provide us with the result that we are interested in; an explicit finite-sample bound assessing the quality of the approximation in \eqref{eqn:weak_conv_hat_gamma}. Amongst other reasons, this is due to the fact that the empirical autocovariance and cross-covariance are biased. We consider the empirical autocovariance and cross-covariance to be of such fundamental importance for applications that results to assess their finite sample distributional approximation, fully explicit in terms of the underlying process/model parameters, segment size $n$ and lag $k$, should be available.

We now continue to discuss work related to assessing the quality of the normal approximation for sums of dependent data.
Staying in the setting of explicit bounds but moving away from the $m$-dependence structure that we use, \cite{Roellin2018} provides bounds on the Wasserstein distance between the distribution of $\sum_{i=1}^{n}X_i$, where $X_1,\ldots, X_n$ is a discrete time martingale difference sequence, and the standard normal distribution. The bound is of the order $\mathcal{O}(n^{-1/2}\log n)$ and the strategy followed to obtain the upper bounds consists of a combination of Stein's method and Lindeberg's argument. In their work related to the Polyak-Ruppert averaged stochastic gradient descent, \cite{ABE_2019} derive an explicit upper bound on the distributional distance between the distribution of the summation of a multivariate martingale difference sequence and the multivariate normal distribution.
%This consists a multivariate extension of the univariate work carried out in \cite{Roellin2018}. However, in this multivariate setting, the bound is not expressed in terms of the Wasserstein distance because of a sublety of Stein's method in the scenario of random vectors.
In their recent work, \cite{Fan_Ma_2020} extend the results of \cite{Roellin2018} by relaxing conditions used in the latter. Apart from the setting of discrete time martingales, work has been done on assessing the normal approximation of a sum of random variables when these satisfy specific mixing conditions; see \cite{Sunklodas_2007} and \cite{Sunklodas_2011} for the cases of strong and $\varphi$-mixing conditions, respectively.
\cite{DedeckerRio2008} provide bounds for the Wasserstein distance between the distribution of $\sum_{i=1}^{n}X_i$ and the normal distribution, when either strong mixing assumptions are satisfied or when $(X_i)_{i \in \mathbb{Z}}$ is either an ergodic martingale difference sequence or an ergodic stationary sequence that satisfies specific projective criteria.

Moving away from the scenario of explicit constants in the bounds, \cite{Dedecker_Merlevede} provide, in the case of $X_1, \ldots, X_n$ being a martingale difference sequence, rates of convergence for minimal distances between linear statistics of the form $\sum_{i=1}^{n}c_{n,i}X_n$, where $c_{n,i}\in \mathbb{R}$, and their limiting Gaussian distribution. \cite{Fan_2019} gives rates of convergence for the Central Limit Theorem of a martingale difference sequence with conditional moment assumptions. For $X_1, \ldots, X_n$ a stationary sequence with finite $p \in (2,3]$ moments, \cite{Jirak2016} proves under a weak dependence condition a Berry-Esseen theorem and shows convergence rates in $L^q$-norm, where $q\geq 1$.
The obtained bounds are though not explicit, in the sense that they depend on a varying absolute constant not given explicitly.
%Also, in \cite{DedeckerRio2008} the underlying process is assumed to be stationary and centered, which does not allow us to immediately obtain our results. Furthermore, the results presented in the current paper hold for non-stationary sequences as well.
%We chose to work with the neighbourhood assumption instead, because it is easier to work within applications.
%Since the bounds are in terms of the Wasserstein distance, one can use the result in \eqref{2020a} and obtain explicit upper bounds in terms of the Kolmogorov distance that are of orders $\mathcal{O}(n^{-1/4})$ and $\mathcal{O}(n^{-1/4} (\log n)^{1/2})$, respectively. Such explicit bounds provide the opportunity to obtain confidence intervals and perform hypothesis tests that are legitimate in a non-asymptotic sense.
%The main results are given in Theorems \ref{general_theorem_nonstat} and \ref{general_theorem}.
%Exact conservative confidence intervals for the parameter of interest are provided in Remark \ref{rem:CI}.

Apart from the machinery employed, the proof methodology followed, and the focus to the specific statistics of the empirical autocovariance and cross-covariance functions, the results presented in this paper are novel in three additional main aspects. Firstly, our results are applicable to non-stationary data sequences. Secondly, our focus is not only on rates of convergence, but the Wasserstein distance bounds derived in the paper are fully explicit in terms of the sample size $n$, the lag $k$, as well as constants that are related to the underlying data; this makes the bound completely computable in examples. Thirdly, the assumptions that we have used are non-restrictive, and they are partly based on an $m$-dependence approximation of the original time series, which is convenient to work with in applications, making our results applicable in a wide range of scenarios. In the case where the range of dependence is finite, for example independent observations or a moving average process of fixed order, the order of our bound is $\mathcal{O}(n^{-1/2})$. In more general cases where the serial dependence vanishes quickly at large lags and moments of order eight exist, the order of our bound is $\mathcal{O}(n^{-1/2} \log n)$. A discussion on the order of the bound can be found in Remark~\ref{rem:order} and, in more detail, in Section~\ref{sec:order_bound}.

The paper is organized as follows. In Section~\ref{sec:upper_bound_general}, we give our main result in the general case; this is an upper bound on the Wasserstein distance between the distribution of the empirical autocovariance and cross-covariance functions~$\cPest{\gamma}{ab}{k}$, defined in~\eqref{def:gam_hat}, and their limiting normal distribution. We highlight that the data are not necessarily obtained from a stationary process.
In Section~\ref{sec:local_dep}, we state and discuss the key assumption for the weakly stationary case.
In Section~\ref{sec:upper_bound}, the main result under stationarity is given.
Sections~\ref{sec:bound_mdep_appr_known} and~\ref{sec:bound_original_data} are devoted to computing bounds in terms of moments of the $m$-dependent approximation or the original centered process, respectively; details on the computation of bounds in terms of moments of the original uncentered process are deferred to Section \ref{app:non_central}. A detailed explanation of the order of the bound with respect to the sample size $n$ is given in Section~\ref{sec:order_bound}.
In Section~\ref{sec:expl}, we apply our general results to the specific case of a causal autoregressive process of order~1.
In Section~\ref{sec:Proofs}, our main result is proven.
%To this end we state and prove a generalized result that is applicable to non-stationary time series.
Section~\ref{sec:Discussion} concludes the paper with a brief discussion on the results.
Technical details on the computation of the bound from Section~\ref{sec:bound_mdep_appr_known}, step-by-step proofs that were not included in the main text, technical details regarding computation and simulation, as well as additional tables for the example in Section~\ref{sec:expl} are provided in a supplement, which is available online.
Sections, results, et cetera that are numbered with letters from the Latin alphabet are always to be found in the supplement.

\section{Main results}
\label{sec:main_result}
\subsection{The explicit upper bound for the general case}
\label{sec:upper_bound_general}
In this section we present a general result that does not require the data to be from a stationary process. To apply the result, an $m$-dependent sequence of the same length, $n$, and dimension, $d$, with finite sixth moments needs to exist. The closeness of the data to the $m$-dependent approximation will determine the size of the bound.

For ease of presentation, some notation is in order.
%In Sections~\ref{sec:vfy_ass2} and~\ref{sec:bnd_D}, it will be shown that the conditions we impose are satisfied in a broad range of examples.
%For any vector $v$ or matrix $M$, their transpose is denoted by $v^{\intercal}$ and $M^{\intercal}$, respectively.
For any vector $\boldsymbol{x} = (x_1, \ldots, x_d)$, we denote its Euclidean norm by $|\boldsymbol{x}| := (\sum_{i=1}^d x_i^2)^{1/2}$, while for a random vector $\boldsymbol{X}$, its $L^q$-norm is denoted by $\| \boldsymbol{X} \|_q := (\E[|\boldsymbol{X}|^q])^{1/q}$, $q \geq 1$.
We denote $\IN := \{1, 2, \ldots\}$ and $\IN_0 := \IN \cup \{0\}$.
Recall, the $r^{th}$ order joint cumulant of a random vector $(\zeta_1, \ldots, \zeta_r)$ is defined as
\begin{equation}\label{def:cum}
\cum(\zeta_1, \ldots, \zeta_r) := \sum_{\nu} (-1)^{p-1} (p-1)! \Big( \E \prod_{j \in \nu_1} \zeta_j\Big) \cdots \Big(\E  \prod_{j \in \nu_p} \zeta_j\Big),
\end{equation}
where the sum is with respect to all partitions $\nu := \{\nu_1, \ldots, \nu_p\}$ of $\{1, \ldots, r\}$; cf. \cite{Brillinger1975}. The general result for the case of a not necessarily stationary sequence is given in Theorem~\ref{general_theorem_nonstat} below. Its proof is deferred to Section~\ref{sec:Proofs}.
\begin{theorem}
	\label{general_theorem_nonstat}
Let $\vO{X}{1}, \ldots, \vO{X}{n}$ be a sequence of $d$-variate random vectors, and assume $\E(\vO{X}{t}) = 0$ for all $t \in \left\lbrace 1,\ldots, n\right\rbrace$. Fix $a,b \in  \left\lbrace 1,\ldots, d\right\rbrace$, and $k \in \left\lbrace 0,\ldots, n-1\right\rbrace$, and let $\hat{\gamma}_{ab}(k)$ be defined as in \eqref{def:gam_hat}. Fix $m \in \mathbb{N}_0$, and let $\boldsymbol{Y}(1), \ldots, \boldsymbol{Y}(n)$ be a sequence of $m$-dependent $d$-variate random vectors, such that $\| \vO{Y}{t} \|_6 < \infty$ for all $t=1,\ldots,n$ and assume that
\begin{equation}\label{def:tildeSigma}
\tilde{\Sigma}_{ab}(k) := \Var\left(n^{-1/2} \sum_{t=1}^{n-k} \cO{Y}{a}{t+k} \cO{Y}{b}{t} \right) > 0.
\end{equation}
Let
\begin{equation}
\nonumber
\tilde D_{a,t} := \| \cO{X}{a}{t} - \cO{Y}{a}{t} \|_2 < \infty, \quad t \in \left\lbrace 1,\ldots, n\right\rbrace,
\end{equation}	
and if $b \neq a$, let $\tilde D_{b,t}$ be defined analogously. Moreover, for $t = 1, \ldots, n-k$, let
\begin{equation}
\nonumber
\tilde{Z}(t) := \cO{Y}{a}{t+k} \cO{Y}{b}{t} - \E[\cO{Y}{a}{t+k} \cO{Y}{b}{t}]
\end{equation}
and define the quantities
\begin{equation}
\label{tildeS_mn}
\tilde Q_t := 
\E\left|\left(\tilde{Z}(t)\tilde{Z}(A_t) - \E\left( \tilde{Z}(t)\tilde{Z}(A_t)\right)\right)\tilde{Z}(B_t)\right| + \frac{1}{2}\E\left|\tilde{Z}(t)\left(\tilde{Z}(A_t)\right)^2\right|,
\end{equation}
where $\tilde{Z}(A) = \sum_{j \in A}\tilde{Z}(j)$ for $A \subset \mathbb{N}$ and
\begin{align}
\nonumber & A_t := \{\ell = 1, \ldots, n-k : | \ell - t | \leq m + k \},\\
\nonumber & B_t := \{ \ell = 1, \ldots, n-k : | \ell - t | \leq 2(m + k) \}.
\end{align}
Finally, let
\begin{equation}
\nonumber \tilde K_t := \tilde D_{a,t+k}\| \cO{X}{b}{t} \|_2 + \tilde D_{a,t+k} \tilde D_{b,t} + \| \cO{X}{a}{t+k} \|_2 \tilde D_{b,t}.
\end{equation}
Then, for any $\gamma \in \IR$ and any $\sigma^2 > 0$, 
\begin{align}
\label{final_upper_bound_nonstat}
d_{{\mathrm{W}}}\left(\mathcal{L}\left(n^{1/2}\left(\cPest{\gamma}{ab}{k} - \gamma\right)\right),\mathcal{N}(0,\sigma^2)\right) 
	\nonumber  & \leq \frac{1}{\sqrt{n}} \sum_{t=1}^{n-k} \tilde K_t + \frac{1}{\sqrt{n}}\sum_{t=1}^{n-k}\left| \frac{n}{n-k} \gamma - \E[\cO{Y}{a}{t+k}\cO{Y}{b}{t}]\right| \\
	& \quad + \sqrt{\frac{2}{\pi \sigma^2}}\left|\sigma^2 - \tilde{\Sigma}_{ab}(k)\right|
	+ \frac{2}{n^{3/2}\left(\tilde{\Sigma}_{ab}(k)\right)^{3/2}} \sum_{t=1}^{n-k} \tilde{Q}_t.
\end{align}
\end{theorem}
In the theorem and throughout the rest of the paper, we use the convention to distinguish notation related to the $m$-dependent approximation with the tilde symbol (e.g. $\tilde{D}_{a,t}, \tilde{Z}(t), \tilde{Q}_t$).

Due to the non-restrictive assumptions and explicitness of the constants, Theorem~\ref{general_theorem_nonstat} can, for example, be used to show asymptotic normality of sequences of estimators where the underlying model or the lag $k$ depends on the segment length $n$. It can also be applied to models with time-varying coefficients. Because of the countless situations in which the result can potentially be applied, but this paper only offers limited space, we will focus on one of the most relevant situations for applications in the following sections: the case of weakly stationary data.

%In the situation described in Section~\ref{subsec:assumptions} denote the cross-covariance of lag $k$ by $\cP{\gamma}{a b}{k} := \Cov(\cO{X}{a}{k}, \cO{X}{b}{0})$ and the fourth order cumulants by
%\[\cP{c}{a_1 b_1 a_2 b_2}{k_1, k_2, k_3} := \cum(\cO{X}{a_1}{k_1}, \cO{X}{b_1}{k_2}, \cO{X}{a_2}{k_3}, \cO{X}{b_2}{0}).\]
\subsection{The key assumption for the stationary case}
\label{sec:local_dep}

From this section onwards, we consider $\seq{\vO{X}{t} : t \in \mathbb{Z}}$ to be a $d$-variate, centered and weakly stationary process, denoted by $\seq{\vO{X}{t}}$, from which a sequence $\vO{X}{1}, \ldots, \vO{X}{n}$ is available, with $\vO{X}{t}$ being $\mathbb{R}^{d}$-valued, $t=1,\ldots,n$. The main assumption used for the result under stationarity is given below.
%\begin{assumption}\label{as:stat_mom}
%Let $\seq{\vO{X}{t}}$ be $d$-variate, centered and weakly stationary, with $\|\vO{X}{t}\|_6 < \infty$.
%\end{assumption}
%Further, as already mentioned, we will use $m$-dependent processes to approximate the data. To this end, we assume
\begin{assumption}\label{as:m_dep_appr}
	For a given $m \in \IN_0$ there exists an $m$-dependent, $d$-variate process $\seq{\vO{Y}{t}}$ where, for $a \in \left\lbrace 1,\ldots,d\right\rbrace$, and $t \in \left\lbrace 1,\ldots, n\right\rbrace$,
\begin{equation}
\nonumber
\tilde D_{a,t}^{(q)} := \| \cO{X}{a}{t} - \cO{Y}{a}{t} \|_q < \infty.
\end{equation}
The number $q \geq 1$ is specified whenever we refer to Assumption \ref{as:m_dep_appr}.
\end{assumption}
{\raggedright{We denote by}} 
\begin{equation}
\label{A:mdep}
\tilde D_{a}^{(q)} := \sup_{1 \leq t \leq n} \tilde D_{a,t}^{(q)}.
\end{equation}
We highlight that if there is a choice on the $m$-dependent sequence, then one could define $\tilde{D}_{a}^{(q)} := \inf_{\bm{Y}} \sup_{1 \leq t \leq n} \| \cO{X}{a}{t} - \cO{Y}{a}{t} \|_q$, with the infimum taken with respect to all possible choices of $\seq{\vO{Y}{t}}$ that satisfy $\| \cO{X}{a}{t} - \cO{Y}{a}{t} \|_q < \infty$; in the interest of obtaining an easily computable bound, however, we state our result for a specific choice.

Even though in this section we assume stationarity of $\seq{\vO{X}{t}}$ in order to allow for a meaningful definition of $\cP{\gamma}{ab}{k}$, our method of proof, as already stated, does not require stationarity. We state and discuss the result for the stationary case in full detail because of its relevance for applications and because it sheds light on the more general result.
%We now comment on the assumptions made in the stationary case. 
Assumption~\ref{as:m_dep_appr} implies that the original process can be approximated in $L^q$ by an $m$-dependent sequence.
For Theorem~\ref{general_theorem}, approximation in $L^2$ is sufficient; i.\,e., we require Assumption~\ref{as:m_dep_appr} with $q=2$.
In Section~\ref{sec:bound_original_data} we explain the general steps to obtaining a bound in terms of properties of the original process $\seq{\vO{X}{t}}$. Lemmas~\ref{lem:bnd_var_mdep_est} and~\ref{lem:bnd1_Smn} can be used to pursue such a bound; Assumption~\ref{as:m_dep_appr} with $q=4$ and $q=6$, respectively, is then required.
We do not require $\seq{\vO{Y}{t}}$ to be stationary or centered, though in applications this will often be the case.
In the example discussed in Section~\ref{sec:expl},  $\| \cO{X}{a}{t} - \cO{Y}{a}{t} \|_q$ is actually independent of $t$; details on how to compute $\tilde D_{a}^{(q)}$ as in \eqref{A:mdep} for the example are available in Section~\ref{app:tech_details:AR1_expl}.
If $\seq{(\vO{X}{t}, \vO{Y}{t})}$ is jointly stationary up to moments of order $q$, then $\tilde D_{a}^{(q)} = \tilde D_{a,0}^{(q)}$ and the supremum in~\eqref{A:mdep} could be omitted.
%If~\eqref{A:mdep} depends on $t$, define $D_{a,q,m} = \sup_{t} \| \cO{X}{a}{t} - \cO{Y}{a}{t} \|_q$ and our results still hold; cf. Theorem~\ref{general_theorem_nonstat}.
Our Assumption~\ref{as:m_dep_appr} is similar in spirit to Assumption~2.1 in \cite{AueEtAl2009}; also see the  examples provided in their Section~4 that illustrate how to apply such a framework to several popular time series models.
Note the following important difference though.
The quantity $\tilde D_{a}^{(q)}$ gives a bound to the goodness of the $m$-dependent approximation measured in $L^q$ and while a larger $m$ will typically result in a better approximation (i.\,e., a smaller $\tilde D_{a}^{(q)}$), there is no requirement at the rate of decay that we would usually have if we were deriving an asymptotic result.
For our main results, that are finite sample in nature, we only require that $\cO{X}{a}{t} - \cO{Y}{a}{t}$ is in $L^q$; i.\,e. the quantity $\tilde D_{a,t}^{(q)}$ is finite.
\subsection{The explicit upper bound for centered stationary data}
\label{sec:upper_bound}
The upper bound on the quantity of interest in the case of a weakly stationary sequence is given below. 
\begin{theorem}
\label{general_theorem}
Let $\seq{\vO{X}{t}}$ be a $d$-variate, centered and weakly stationary process. Fix $a,b \in \left\lbrace 1,\ldots, d \right\rbrace$, and $k \in \left\lbrace 0,\ldots, n-1\right\rbrace$, and let $\hat{\gamma}_{ab}$ and $\gamma_{ab}(k)$ be defined as in \eqref{def:gam_hat} and \eqref{gamma_ab}, respectively. Fix $m \in \mathbb{N}_0$, and let $\seq{\vO{Y}{t}}$ be a process as in Assumption~\ref{as:m_dep_appr}, which we assume holds with $q=2$, and also $\|\vO{Y}{t}\|_6 < \infty$ for all $t = 1,\ldots, n$. For given $n \in \mathbb{N}$, assume that both $\cP{\Sigma}{ab}{k}$ and $\tilde{\Sigma}_{ab}(k)$, defined in \eqref{def:Sigma} and \eqref{def:tildeSigma}, respectively, are positive. Finally, with $\tilde D_{a}^{(q)}$ as in \eqref{A:mdep}, let,
\begin{equation}
\nonumber \tilde K := \tilde D_{a}^{(2)} \| \cO{X}{b}{0} \|_2 + \tilde D_{a}^{(2)} \tilde D_{b}^{(2)} + \| \cO{X}{a}{0} \|_2 \tilde D_{b}^{(2)}.
\end{equation}
For $A_t$ and $B_t$ as in Theorem~\ref{general_theorem_nonstat}, and $\tilde Q_t$ as in~\eqref{tildeS_mn}, we have that
\begin{align}
\label{final_upper_bound}
d_{{\mathrm{W}}}\left(\mathcal{L}\left(\sqrt{n}\left(\cPest{\gamma}{ab}{k} - \cP{\gamma}{ab}{k}\right)\right),\mathcal{N}\left(0,\cP{\Sigma}{ab}{k}\right)\right)
	\nonumber & \leq \frac{k}{\sqrt{n}} |\cO{\gamma}{ab}{k}| + \frac{\sqrt{2}}{\sqrt{\pi\cO{\Sigma}{ab}{k}}}\left|\cO{\Sigma}{ab}{k} - \tilde{\Sigma}_{ab}(k)\right|\\
	& \quad + \frac{2(n-k)}{\sqrt{n}} \tilde K + \frac{2}{\left(n\tilde{\Sigma}_{ab}(k)\right)^{3/2}} \sum_{t=1}^{n-k} \tilde Q_t.
\end{align}
\end{theorem}
\noindent
\textbf{Proof of Theorem~\ref{general_theorem}.}
Choose $\gamma := \cP{\gamma}{ab}{k}$ and $\sigma^2 := \cO{\Sigma}{ab}{k}$. Note that the conditions of Theorem~\ref{general_theorem} imply that the conditions of Theorem~\ref{general_theorem_nonstat} are satisfied. The bound in the stationary case then follows from $\tilde K_t \leq \tilde K$ for all $t = 1,\ldots, n-k$ and
\begin{equation*}
\begin{split}
& \frac{1}{\sqrt{n}}\sum_{t=1}^{n-k}\left| \frac{n}{n-k} \gamma - \E[\cO{Y}{a}{t+k}\cO{Y}{b}{t}]\right| \\
& \leq \frac{1}{\sqrt{n}}\sum_{t=1}^{n-k}\left| \frac{n}{n-k} \cP{\gamma}{ab}{k} - \cP{\gamma}{ab}{k} \right| + \frac{1}{\sqrt{n}}\sum_{t=1}^{n-k} \E\left| \cO{X}{a}{t+k}\cO{X}{b}{t} - \cO{Y}{a}{t+k}\cO{Y}{b}{t}\right| \\
& \leq \frac{n-k}{\sqrt{n}} \left| \frac{k}{n-k} \cP{\gamma}{ab}{k} \right| + \frac{n-k}{\sqrt{n}}\tilde K,
\end{split}
\end{equation*}
where we used \eqref{gamma_ab} and a telescoping sum argument (Lemma~\ref{lem:approx_jnt_moments} with $\alpha = 1$ and $p = 2$).
\qed

\begin{remark}
	The four terms that make up the right-hand side of~\eqref{final_upper_bound} can roughly be interpreted as follows: (i) $k n^{-1/2} |\cO{\gamma}{ab}{k}|$, is due to the fact that $\cPest{\gamma}{ab}{k}$ is a biased estimate for $\cP{\gamma}{ab}{k}$, but the limiting normal distribution has mean equal to zero; (ii) $\sqrt{2/(\pi\cO{\Sigma}{ab}{k})} |\cO{\Sigma}{ab}{k} - \tilde{\Sigma}_{ab}(k)|$, is related to the fact that the variance of $\cPest{\gamma}{ab}{k}$ may differ from $n^{-1} \Sigma_{ab}(k)$; \linebreak (iii) $2(n-k) n^{-1/2} \tilde K$, is due to our method of proof where we use the $m$-dependent approximation, and (iv) $2(n\tilde{\Sigma}_{ab}(k))^{-3/2} \sum_{t=1}^{n-k} \tilde Q_t$, is due to an application of Stein's method; cf.\ Lemma~\ref{Lemma_locally_dependent}.
\end{remark}
The following remark provides a brief discussion of the computation and of the order of the bound; detailed explanations are given in Section~\ref{sec:order_bound}.
\begin{remark}
	\label{rem:order}
	At first glance, the bound might seem slightly complicated, especially due to the expression $\tilde{Q}_{t}$. In Section~\ref{sec:bound_mdep_appr_known}, we explain two methods that allow to bound $\tilde Q_t$ by expressions whose exact value can be computed in examples. In Section~\ref{sec:expl}, we then calculate the exact value of such a bound term by term for the case of a causal autoregressive process. To obtain a rate, we choose $m$ as a function of $n$. The choice that allows optimization of the order of the bound with respect to $n$ depends on the underlying process $\seq{\vO{X}{t}}$. In Section~\ref{sec:order_bound}, we discuss two general scenarios where the bound is of the order $\mathcal{O}(n^{-1/2})$ or $\mathcal{O}(n^{-1/2} \log n)$, respectively.
\end{remark}
\subsection{The bound in Theorem \ref{general_theorem} when the $m$-dependent approximation is known}
\label{sec:bound_mdep_appr_known}
Let $\vO{X}{t}$ be such that, for given $n$, $k$, and $m$, we can compute $\cO{\gamma}{ab}{k}$, $\cO{\Sigma}{ab}{k}$, $\| \cO{X}{a}{0} \|_2$ and $\| \cO{X}{b}{0} \|_2$.
Assume further that we may choose $\vO{Y}{t}$ such that $\tilde{\Sigma}_{ab}(k)$, $\tilde D_{a}^{(2)}$ and $\tilde D_{b}^{(2)}$ can be computed.
Then, the only missing piece to obtain the upper bound in~\eqref{final_upper_bound} is $\tilde Q_t$, defined in~\eqref{tildeS_mn}.
The absolute joint moments in the definition of $\tilde Q_t$ can be inconvenient.
To address potential problems in the computation of $\tilde Q_t$, we now describe two ways to bound $\tilde Q_t$ by quantities that can be explicitly computed in examples.
Firstly, if $\seq{\vO{Y}{t}}$ is stationary (otherwise see below, within Method 1), we bound $\tilde Q_t$ in terms of $\|\vO{Y}{0}\|_6$, which is finite from the statement of Theorem \ref{general_theorem}. Secondly, we obtain a better bound for $\tilde Q_t$ when $m$ is large. The price we pay for the second method is a more complicated computation and the requirement that $\|\vO{Y}{0}\|_8 < \infty$.

{\raggedright{\textbf{Method 1 to bound $\tilde Q_t$.}}} Denoting $\mu_{jq} := \|\cO{Y}{j}{0} \|_q$, we have
\begin{equation}\label{bnd:Qt_naive}
\begin{split}
& \tilde Q_t \leq 
\sum_{j_1 \in A_t} \sum_{j_2 \in B_t} \Big( \E\big| \tilde{Z}(t) \tilde{Z}(j_1) \tilde{Z}(j_2) \big|
+ \E\big| \tilde{Z}(t) \tilde{Z}(j_1) \big| \E\big|  \tilde{Z}(j_2) \big| \Big) + \frac{1}{2} \sum_{j_1 \in A_t} \sum_{j_2 \in A_t}  \E\big| \tilde{Z}(t) \tilde{Z}(j_1) \tilde{Z}(j_2) \big| \\
& \leq |A_t| |B_t| \Big( \mu_{a6}^3 \mu_{b6}^3 + \mu_{a4}^2 \mu_{b4}^2 \mu_{a2} \mu_{b2} \Big) + \frac{1}{2} |A_t|^2 \mu_{a6}^3 \mu_{b6}^3 \leq \frac{5}{2} (4m + 4k + 1)^2 \|\vO{Y}{0}\|_6^6.
\end{split}
\end{equation}
Employing the triangle inequality, a generalized version of H\"{o}lder's inequality, and the stationarity of $\vO{Y}{t}$, the joint moments in the definition of~$\tilde Q_t$ were broken up into moments of the marginals.
If $\seq{\vO{Y}{t}}$ is not stationary we can use the $\sup_{t=1,\ldots,n} \|\vO{Y}{t}\|_6$ instead.
The bound in~\eqref{bnd:Qt_naive} is particularly simple and straightforward to compute.
In essence, we see a product of the marginal moments $\|\cO{Y}{j}{0} \|_q$ for $j=a,b$ and $q=2,4,6$ scaled by a multiple of~$m^2$.
A bound of the order~$m^2$ is most useful when $m$ is small.
To improve upon~\eqref{bnd:Qt_naive} in the case when~$m$ is large, we next derive a bound for $\tilde Q_t$ in terms of joint (non-absolute) moments.

{\raggedright{\textbf{Method 2 to bound $\tilde Q_t$.}}} We apply the Cauchy-Schwarz inequality and $\E(\tilde{Z}(t)) = 0$ to obtain  
\begin{equation}\label{bnd:Smn_new_bound}
\begin{split}
\tilde Q_t \leq \Var \Big( \tilde{Z}(t) \tilde{Z}(A_t) \Big)^{1/2} \Var\Big( \tilde{Z}(B_t) \Big)^{1/2}
 + \frac{1}{2}  \Big[ \E \Big( \tilde{Z}(t) \tilde{Z}(A_t) \Big)^2 \Big]^{1/2} \Var \Big( \tilde{Z}(A_t) \Big)^{1/2}.
\end{split}
\end{equation}
A crucial difference between the right-hand side of~\eqref{bnd:Smn_new_bound} and the first bound in~\eqref{bnd:Qt_naive} is that the former is in terms of joint moments of $\tilde{Z}(t)$ and the later in terms of joint moments of $|\tilde{Z}(t)|$.
Using standard combinatorial arguments (cf.\ Theorem 2.3.2 in \cite{Brillinger1975}), the right-hand side in~\eqref{bnd:Smn_new_bound} can be computed from cumulants of $\seq{\vO{Y}{t}}$.
These arguments are straightforward but tedious and therefore deferred to Section~\ref{app:tech_details:Sec23}.
Another important advantage of the second method is that, in the common situation where serial dependence is less pronounced at larger lags, such that cumulants are summable, the bound obtained by the second method is of the order $\mathcal{O}(m)$, which, compared to the $\mathcal{O}(m^2)$ bound obtained by Method 1, is much advantageous when $m$ is large. Intuitively, this can be seen from the fact that the variance of a sum of $m$ elements of a short range dependent sequence is of the order~$m$.
Additional details are available in the proof of Proposition~\ref{prop:rate} that can be found in Section~\ref{prf:prop:rate}.

\subsection{The bound with respect to the original data}
\label{sec:bound_original_data}
In Section~\ref{sec:bound_mdep_appr_known} we explained computational details regarding a bound for the case when $\seq{\vO{Y}{t}}$, the $m$-dependent approximation, is known. The method described required the computation of joint moments of $\seq{\vO{Y}{t}}$ in order to obtain $\tilde{\Sigma}_{ab}(k)$ and the right-hand side of either~\eqref{bnd:Qt_naive} or~\eqref{bnd:Smn_new_bound}. If such computation is possible, then numerically evaluating the bound obtained from~\eqref{final_upper_bound} in combination with~\eqref{bnd:Qt_naive} or~\eqref{bnd:Smn_new_bound}, for fixed values of $n$ and $m$, is the preferred method. The aim of this section is to facilitate our result for situations where a bound that depends on $\seq{\vO{Y}{t}}$ might be inconvenient (e.\,g., when $\seq{\vO{Y}{t}}$ is unknown).

There are, at least two, good reasons to pursue a bound that only depends on quantities defined in terms of $\seq{\vO{X}{t}}$. The first reason is a philosophical one. Noting that the statistic of interest, $\cPest{\gamma}{ab}{k}$, is defined in terms of the original process $\seq{\vO{X}{t}}$ we observe that the left-hand side of~\eqref{final_upper_bound} only depends on $\seq{\vO{X}{t}}$, too. Therefore, the right-hand side of~\eqref{final_upper_bound} being defined, amongst others, in terms of $\tilde{\Sigma}_{ab}(k)$ and $\tilde Q_t$, both depending on $\seq{\vO{Y}{t}}$, can be considered a discrepancy. The second reason is a practical one. In Sections~\ref{sec:order_bound} and~\ref{prf:prop:rate} it can be seen that the discussion of asymptotic properties of the bound can be simplified when the dependence on $m$ is not via properties of $\vO{Y}{t}$.

To obtain a bound in terms of moments of $\seq{\vO{X}{t}}$, it suffices to quantify the effect of replacing $\tilde{\Sigma}_{ab}(k)$ by $\cP{\Sigma}{ab}{k}$ and the effect of replacing $\tilde Q_t$ by 
\begin{equation}
\label{notation_for_fourth_term}
Q_t := \E\left|\left(Z(t) Z(A_t) - \E\Big( Z(t) Z(A_t) \Big)\right) Z(B_t) \right| + \frac{1}{2}\E\left| Z(t) \Big( Z(A_t) \Big)^2\right|,
\end{equation}
where $Z(t) := \cO{X}{a}{t+k} \cO{X}{b}{t} - \E[\cO{X}{a}{t+k} \cO{X}{b}{t}]$, $t \in \IN$, and $Z(A) = \sum_{j \in A}Z(j), A \subset \mathbb{N}$.

In Section~\ref{app:lemmas_for_S2.5} we provide results that can be used to derive a bound in terms of moments of $\seq{\vO{X}{t}}$. Further, in Section~\ref{app:non_central}, we discuss the case of non-centred data and provide a result to derive a bound for $d_{{\mathrm{W}}}\left(\mathcal{L}\left(\sqrt{n}\left(\cPpre{\gamma}{ab}{k} - \cP{\gamma}{ab}{k}\right)\right),\mathcal{N}(0,\Sigma_{ab}(k))\right)$ in this case.

\subsection{Explanation on the order of the bound}
\label{sec:order_bound}
In Remark~\ref{rem:order} we have stated the outcomes of the asymptotic analysis of our bound. In this section, the details are provided. We begin by making the conditions we work under precise. For simplicity, we consider only the case where the underlying process $\seq{\vO{X}{t}}$ and the lag $k$ are not allowed to change with $n$. The two regimes we consider are:

{\raggedright{\textbf{Regime 1.}}} Let $\seq{\vO{X}{t}}$ be $d$-variate, centered and stationary, with $\|\vO{X}{0}\|_6 < \infty$, and also be $M$-dependent (for a fixed $M \in \IN_0$).

{\raggedright{\textbf{Regime 2.}}} Let $\seq{\vO{X}{t}}$ be $d$-variate, centered and stationary, with $\|\vO{X}{0}\|_8 < \infty$, and it also satisfies Assumption~\ref{as:m_dep_appr} with $q=8$, and~\eqref{eqn:sum_cum} and~\eqref{A:mdep_geom}, below.

We assume stationarity up to moment of order 6 or 8 in Regimes~1 and~2, respectively. In Regime~2, we require summability of cumulants up to order 8; i.\,e., for $p=2, \ldots, 8$, we have
\begin{equation}\label{eqn:sum_cum}
\sum_{k_1, \ldots, k_{p-1} = -\infty} (1+|k_j|) | \cum(\cO{X}{a_1}{k_1}, \ldots, \cO{X}{a_{p-1}}{k_{p-1}}, \cO{X}{a_p}{0} | < \infty,
\end{equation}
for $j = 1, \ldots, p-1$ and any $p$ tuple $a_1, \ldots, a_p$. Further, we require that the $m$-dependent approximation from Assumption~\ref{as:m_dep_appr} is good enough such that the $L^q$-error vanishes at an exponential rate; i.\,e., there exist constants $K \geq 0$ and $\rho \in (0,1)$ such that for every $n \in \IN$ we can choose $m = m_n \in \IN_0$ and an $m$-dependent, $d$-variate process $\seq{\vO{Y^{(m)}}{t}}$ that satisfies
\begin{equation}
\label{A:mdep_geom}
\tilde D_{a}^{(q)} := \sup_{1 \leq t \leq n} \| \cO{X}{a}{t} - \cO{Y^{(m)}}{a}{t} \|_q \leq K \rho^m, \quad \text{ for }\; a=1,\ldots,d;\; q=8.
\end{equation}

\begin{remark}\label{rem:asym_reg}
	(i) Examples for Regime~1 include moving average processes of finite order and independent data. In this regime, \eqref{eqn:sum_cum} holds for $p=2,\ldots,6$, as cumulants vanish if one of the variables is independent of the others. Further, for any $M$-dependent process, as in Regime~1, the canonical choice for the $m$-dependent approximation of Assumption~\ref{as:m_dep_appr} is $\vO{Y^{(m)}}{t} = \vO{X}{t}$ for $m \geq M$. Choosing the quantity $m$ in the bound~\eqref{final_upper_bound} as $m = \min\{M, n\}$ we see that a stronger version of~\eqref{A:mdep_geom} is satisfied, where we have $\tilde{D}_a^{(q)} = 0$ for $n \geq M$.
	\vspace{0.05in}
	\\
	(ii) As an example for Regime 2, one can consider a linear process $\vO{X}{t} = \sum_{j=0}^{\infty} \vO{\Psi}{j} \vO{\epsilon}{t-j}$ where the spectral norms of the coefficients satisfy $\| \vO{\Psi}{j}\|_2 \leq \rho^j$ for some $\rho \in (0,1)$ and the innovations are i.\,i.\,d with $\|\vO{\epsilon}{t}\|_8 < \infty$. Then, it can be shown that~\eqref{eqn:sum_cum} holds and~\eqref{A:mdep_geom} holds with $C := \|\vO{\epsilon}{t}\|_8 \rho/(1-\rho)$. In particular, causal autoregressive processes are included.
\end{remark}

The following proposition gives the order of the bound~\eqref{final_upper_bound} in Theorem~\ref{general_theorem}. The proof is in Section \ref{prf:prop:rate}.
\begin{proposition}\label{prop:rate}
(i) In Regime~1, with $m := \min\{M,n\}$, the order of the bound is $\mathcal{O}(n^{-1/2})$.\\	
(ii) In Regime~2, with $m := C \log n$, $C \geq \frac{3}{2\log(1/\rho)}$, where $\rho$ is as in \eqref{A:mdep_geom}, the order of the bound is $\mathcal{O}(n^{-1/2} \log n)$.
\end{proposition}
\section{Examples}\label{sec:expl}
\subsection{Causal autoregressive processes of order 1}\label{sec:expl:AR1}
As an example, for which we discuss the result of Theorem~\ref{general_theorem}, we now consider the case where the data stem from a causal AR(1) process $\seq{\cO{X}{}{t}}$ that satisfies $\cO{X}{}{t} = \alpha \cO{X}{}{t-1} + \cO{\varepsilon}{}{t},$ with $|\alpha| < 1$, where $\seq{\cO{\varepsilon}{}{t}}$ are i.\,i.\,d.\ and satisfy $\E |\cO{\varepsilon}{}{t} |^8 < \infty$. We consider $\alpha \in \{0, 0.1, 0.3, 0.5, 0.7\}$ and three cases for the distribution of the innovations:
\begin{itemize}
	\item $\varepsilon(t) \sim \mathcal{N}(0,1)$, or
	\item $\varepsilon(t) \sim \nu^{-1/2} (\nu-2)^{1/2} t_{\nu}$, where we choose $\nu \in \{ 9, 14 \}$.
\end{itemize}
We have chosen the normal distribution as an example with light tails, the scaled $t_{9}$-distribution as a distribution with heavier tails that still satisfies the condition of existence of the 8th moments, and the scaled $t_{14}$-distribution as an example in-between. Note that for each of these three cases we have standardized cumulants of orders 1 and 2; i.\,e., $\kappa_1 := \E\left(\varepsilon(t)\right) = 0$ and $\kappa_2 := \Var\left( \varepsilon(t)\right) = 1$. Cumulants of higher order depend on the distribution of the innovations. If $\varepsilon(t) \sim \mathcal{N}(0,1)$, then cumulants of order higher than or equal to 3 vanish; i.\,e., $\kappa_p := \cum_p(\varepsilon(t)) = 0$, for $p = 3,4,\ldots$.
In the case when $\varepsilon(t) \sim \nu^{-1/2} (\nu-2)^{1/2} t_{\nu}$, $\nu > 8$, we have that cumulants of orders $p = 3, 5, 7$ vanish due to symmetry, and that cumulants of order 4, 6 and 8 are
\[\kappa_4 = \frac{6}{\nu-4}, \quad \kappa_6 = \frac{240}{(\nu-4)(\nu-6)}, \quad \text{and} \quad \kappa_8 = \frac{5040 (5 \nu - 22)}{(\nu-4)^2 (\nu-6) (\nu-8)}.\]
\proglang{R} code and instructions to replicate the results of Section \ref{sec:expl} are available on \linebreak \url{https://github.com/tobiaskley/ccf_bounds_replication_package}.

\subsection{Computing the bound}\label{sec:comp_bnd}
We compute the bound from Theorem~\ref{general_theorem} in combination with~\eqref{bnd:Smn_new_bound} of the second method to bound $\tilde Q_t$, described in Section~\ref{sec:bound_mdep_appr_known}, where the data stem from an AR(1) process as described in Section~\ref{sec:expl:AR1}. Details of how the bound is obtained in the case of the example are deferred to Section \ref{app:tech_details:AR1_expl}. Note that, for given autoregressive parameter $\alpha$, distribution of $\varepsilon(t)$, segment length $n$, and lag $k$ the bound is still a function of $m$. We denote the bound by $B_n(m)$ to emphasize that it can be computed for different values of $m$. Further, we denote by $m^* := \arg\min_{m=0,1,\ldots, m_{\max}} B_n(m)$ the value of $m$ for which the minimum is achieved. We have introduced the upper bound $m_{\max}$ as a stopping rule for computations which we chose large enough such that $m^* < m_{\max}$ was satisfied in all cases of our example, meaning that the minimum is not obtained for $m = m_{\max}$. We chose $m_{\max} = 30$. In Figure~\ref{fig:bnd_fct_m}, values of the bound $B_n(m)$ are shown as they depend on $m$, for different $n$ and different distributions of $\varepsilon(t)$.
Comparing the plots in Figure~\ref{fig:bnd_fct_m} from left to right, it can be seen that $m^*$ increases very slowly as $n$ increases.
This is unsurprising because in this example of the causal AR(1) process, we are under Regime 2 explained in Section \ref{sec:order_bound}; recall the asymptotic considerations of Proposition \ref{prop:rate} where $m^* = \mathcal{O}(\log n)$, which leads to $B_n(m^*) = \mathcal{O}(n^{-1/2} \log n)$ .
Comparing the plots in Figure~\ref{fig:bnd_fct_m} from top to bottom, it can be seen that the value of the bound gets larger as the tails get heavier.
We expect this as well, as the cumulants of the distribution of the innovations $\varepsilon(t)$ become larger when we have distributions with heavier tails.

\begin{figure}[t]
	\hspace*{0.5cm}
	\begin{minipage}{1.1\linewidth}
		\includegraphics[scale = 0.35]{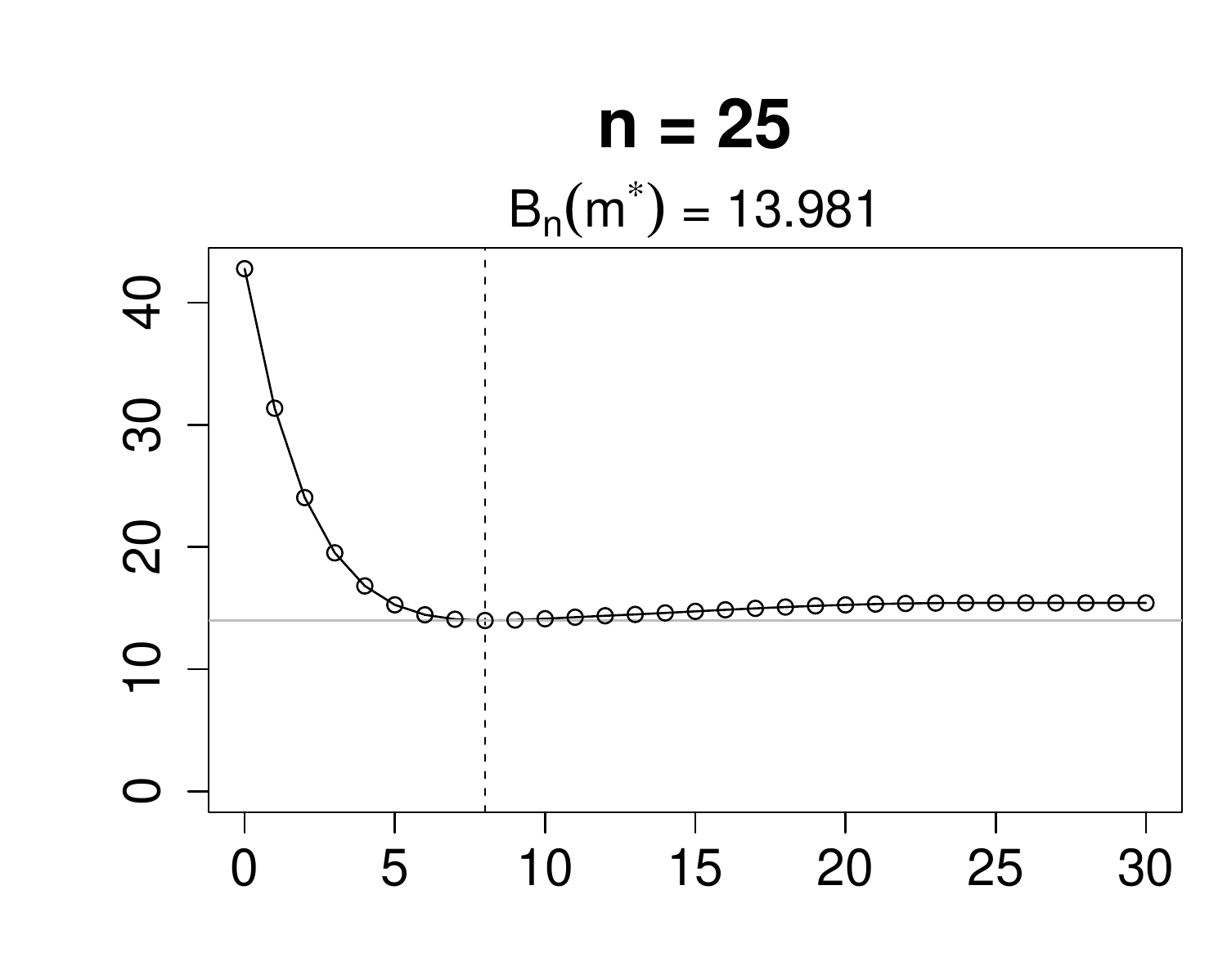}
		\hspace*{-0.55cm}
		\includegraphics[scale = 0.35]{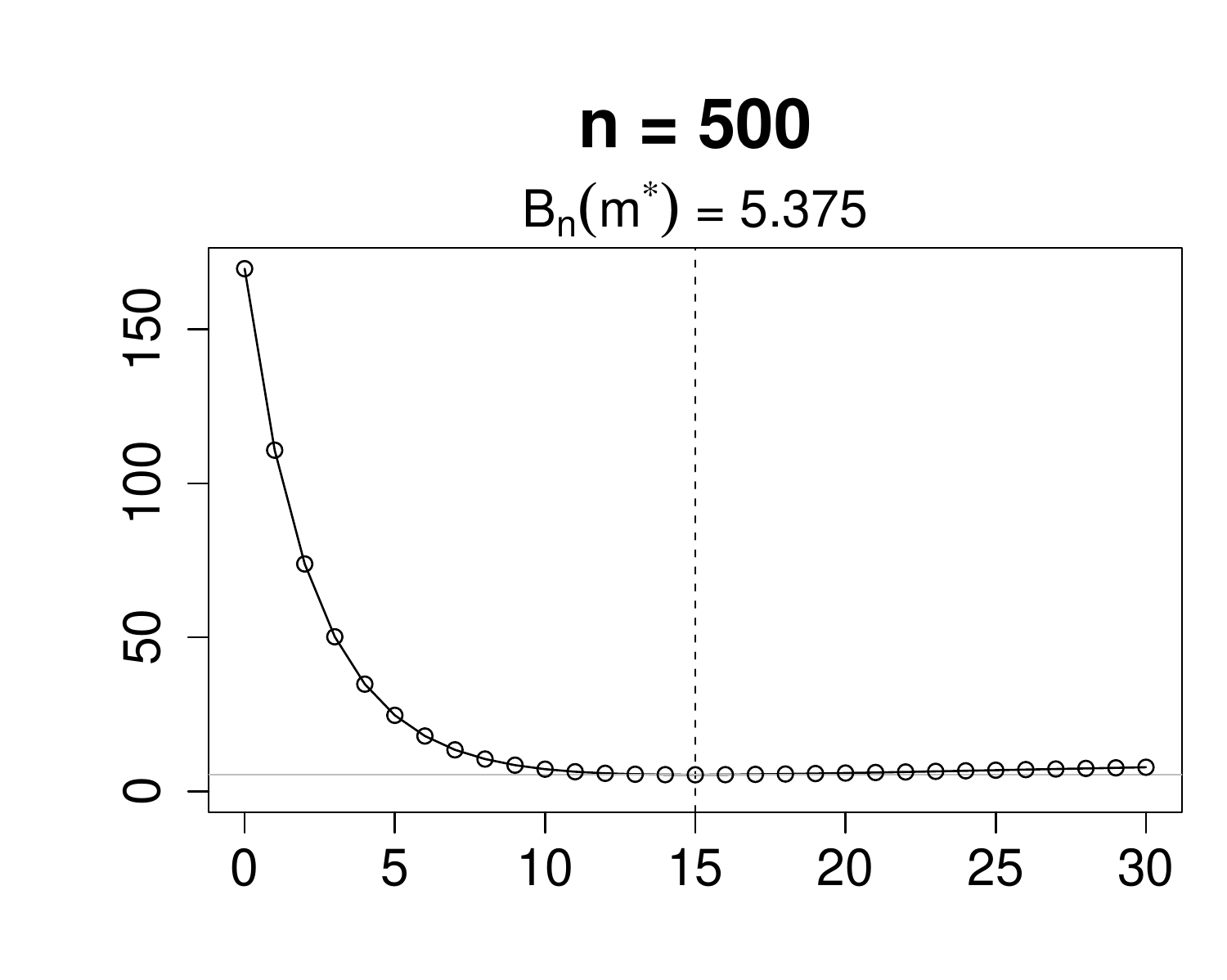}
		\hspace*{-0.55cm}
		\includegraphics[scale = 0.35]{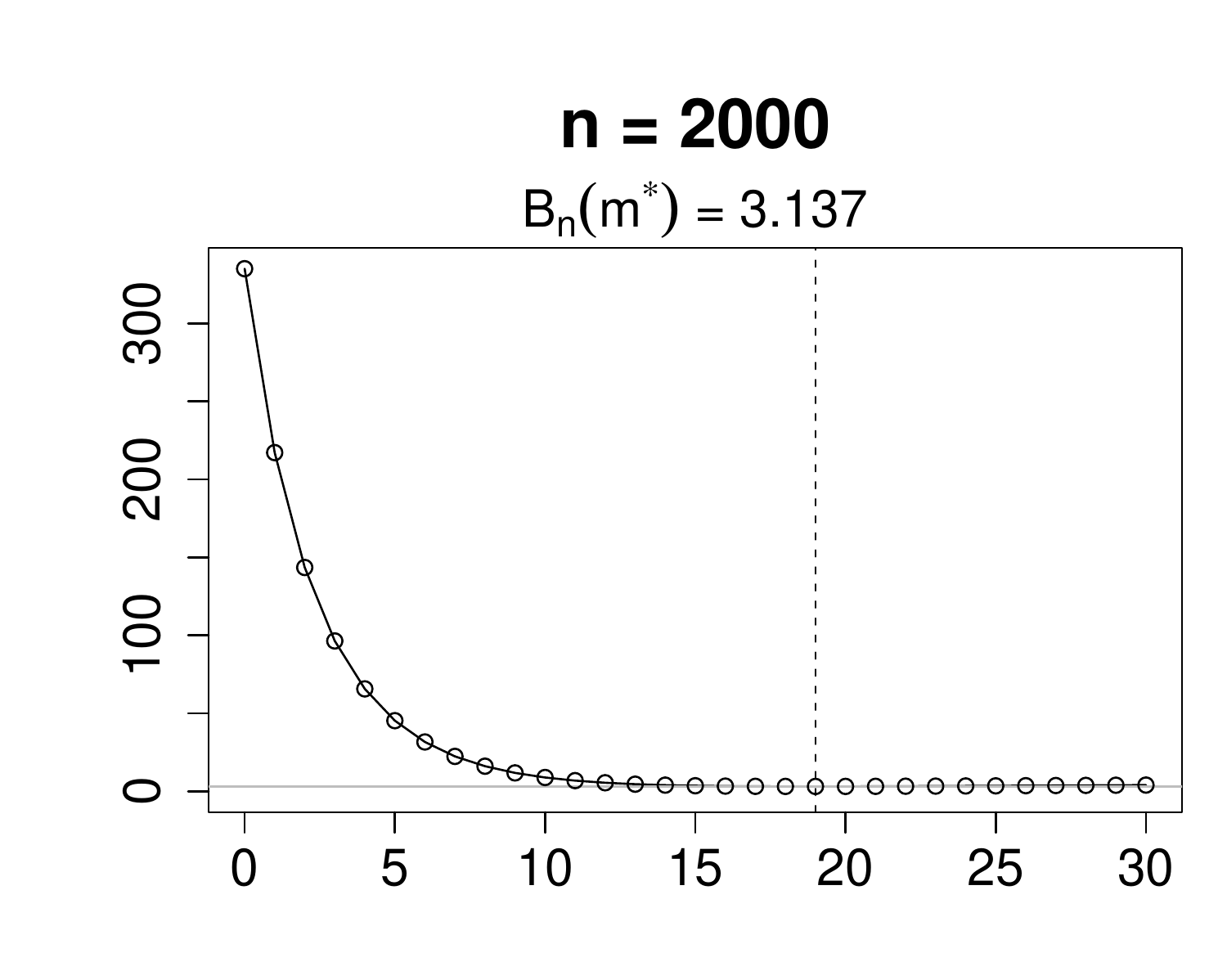}
		
		%\vspace*{-0.1cm}
		\includegraphics[scale = 0.35]{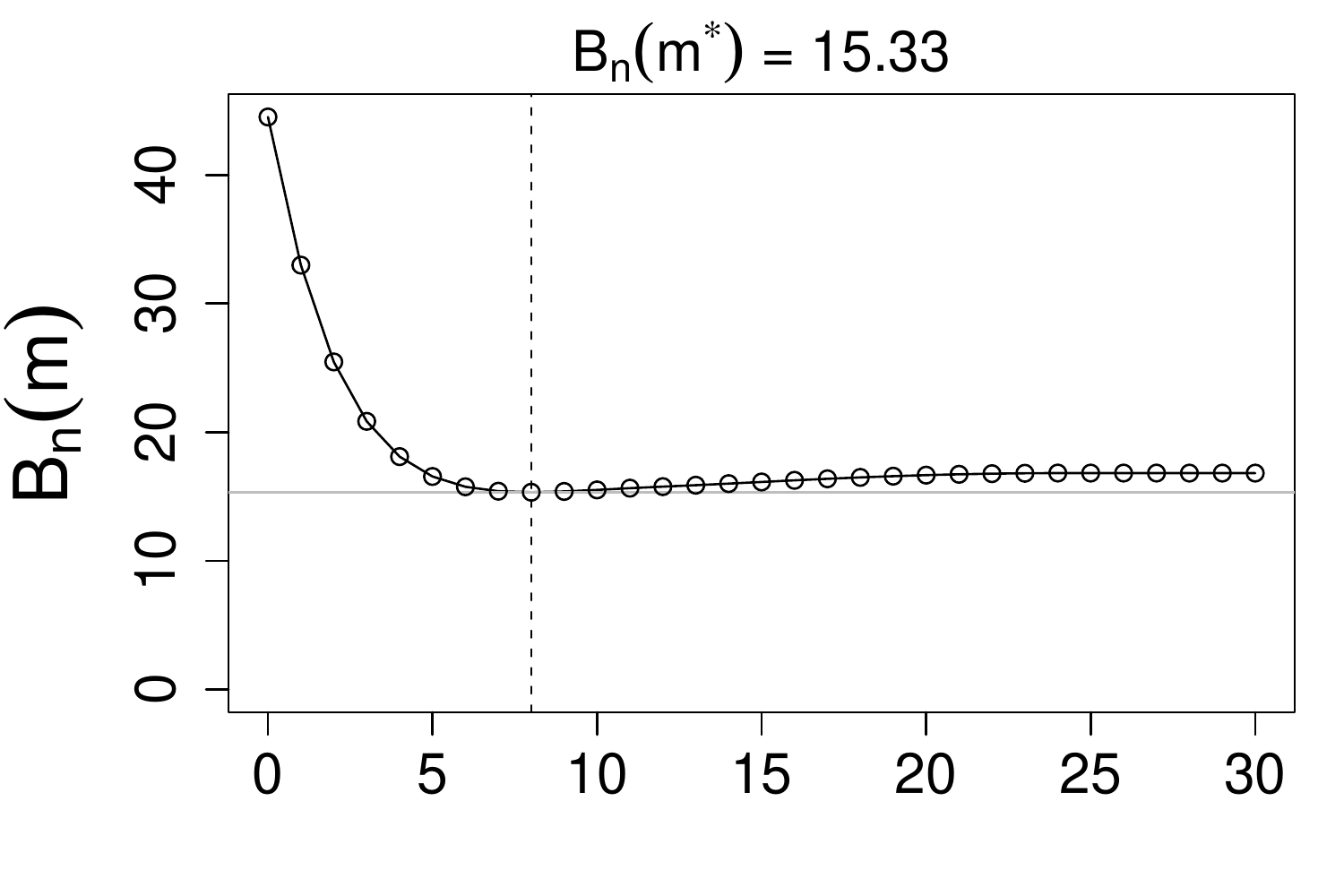}
		\hspace*{-0.55cm}
		\includegraphics[scale = 0.35]{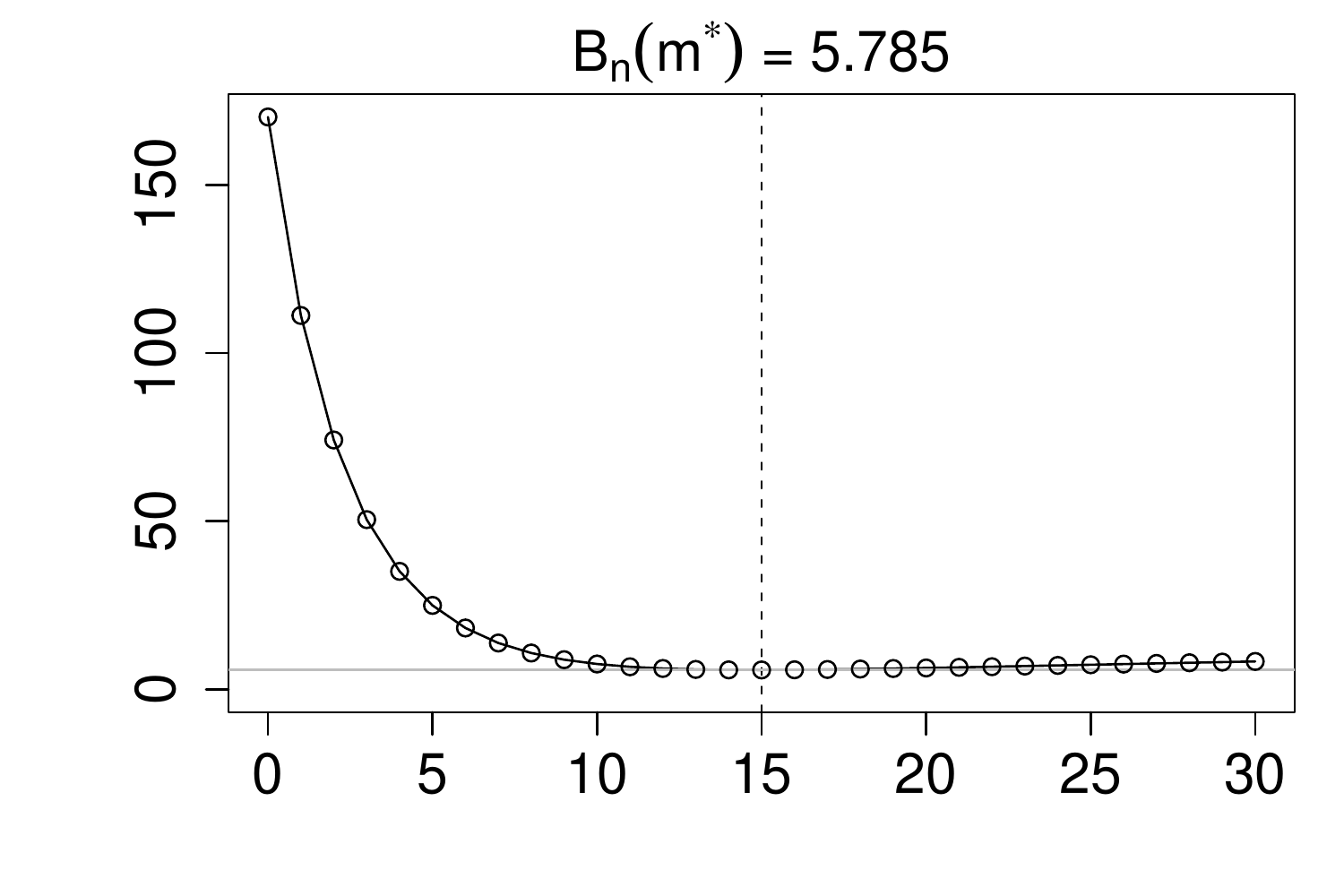}
		\hspace*{-0.55cm}
		\includegraphics[scale = 0.35]{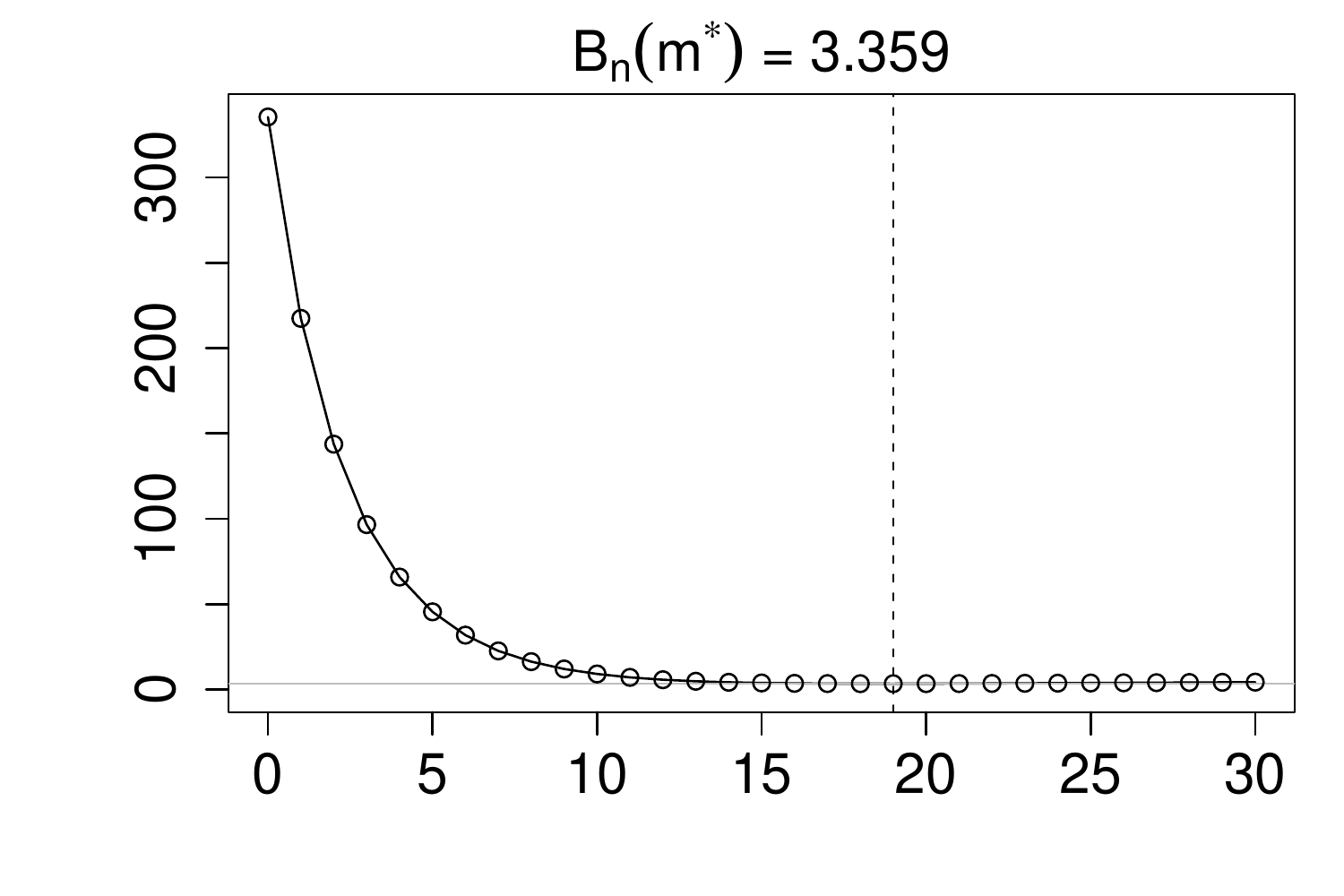}
		
		%\vspace*{-0.1cm}
		\includegraphics[scale = 0.35]{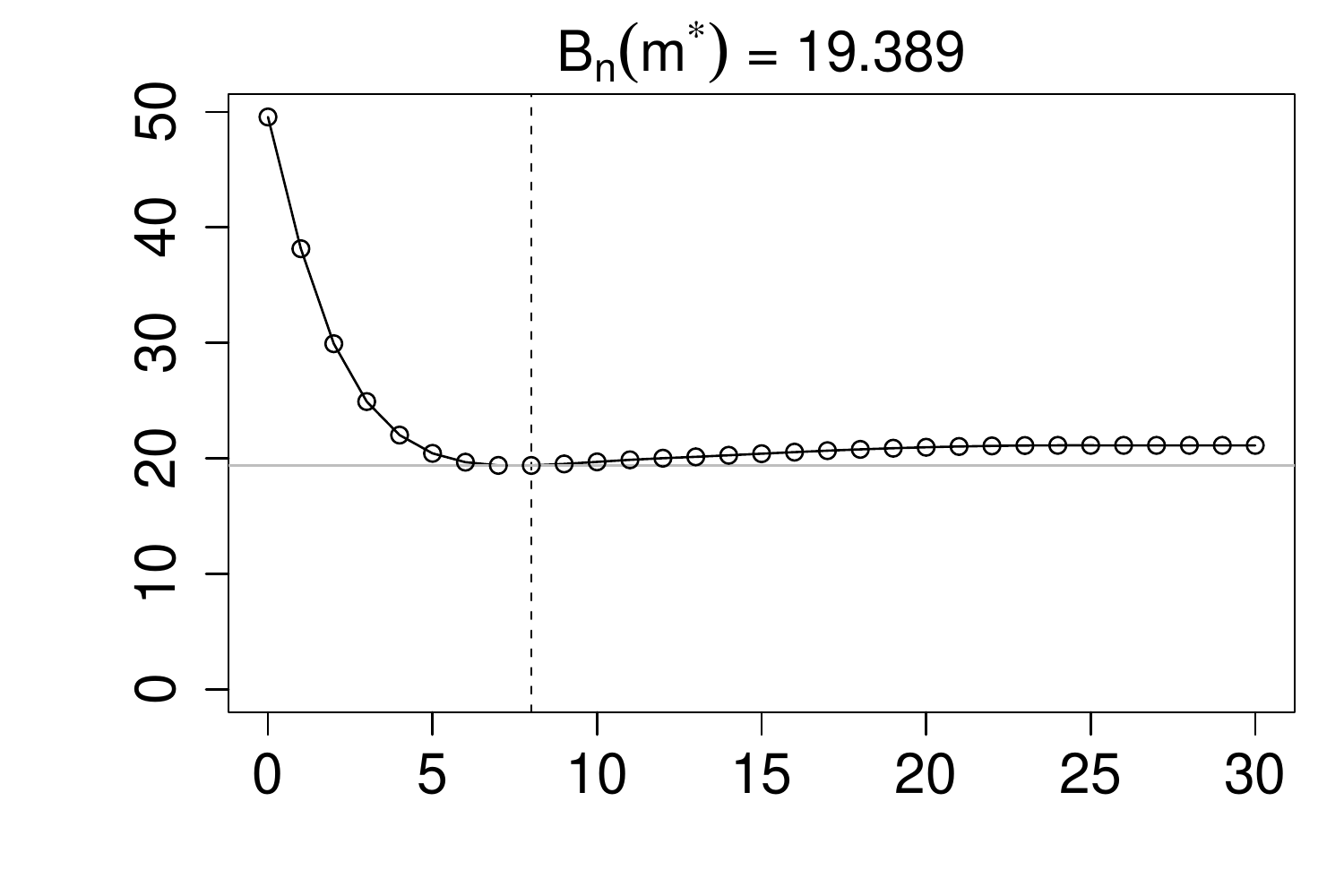}
		\hspace*{-0.55cm}
		\includegraphics[scale = 0.35]{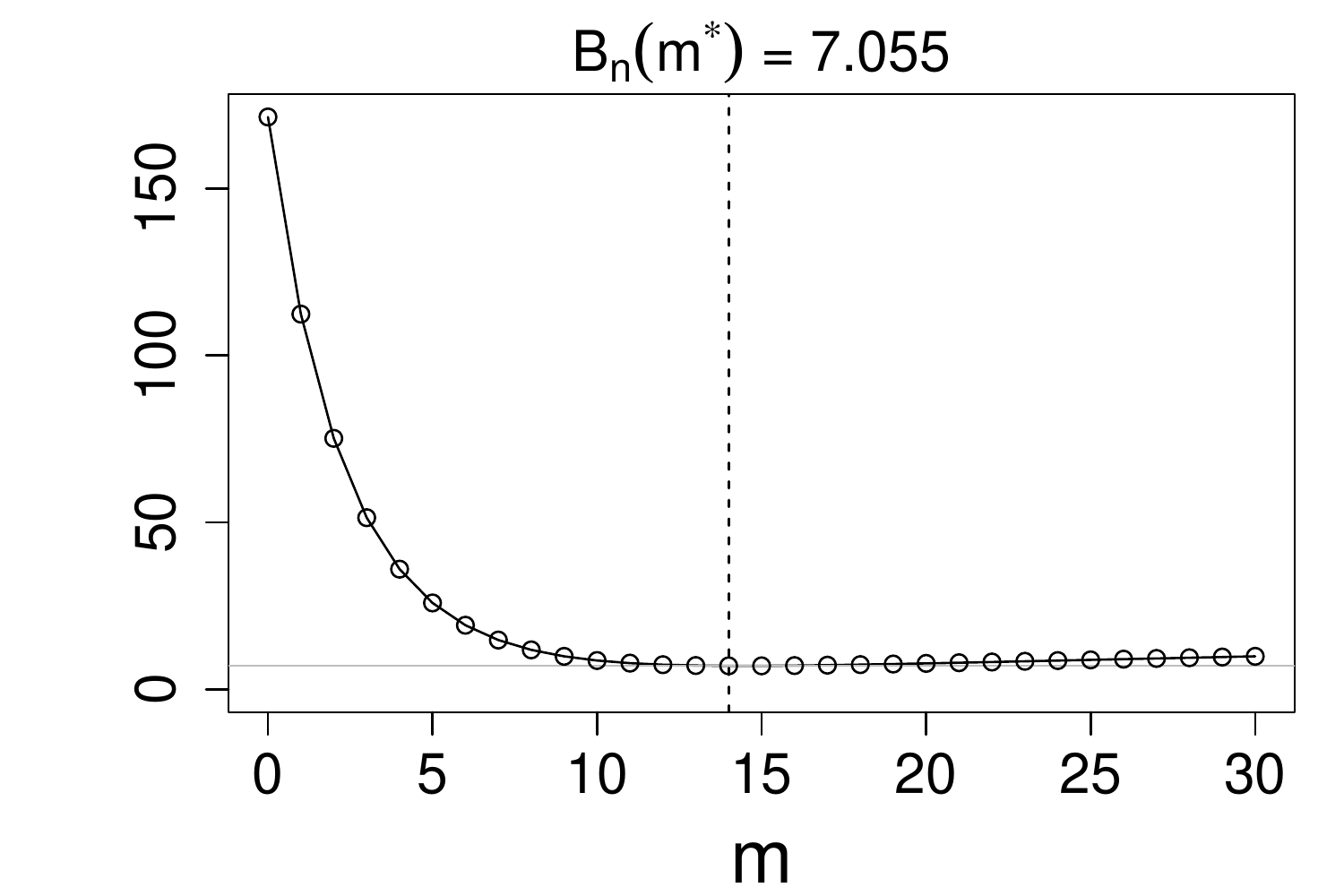}
		\hspace*{-0.55cm}
		\includegraphics[scale = 0.35]{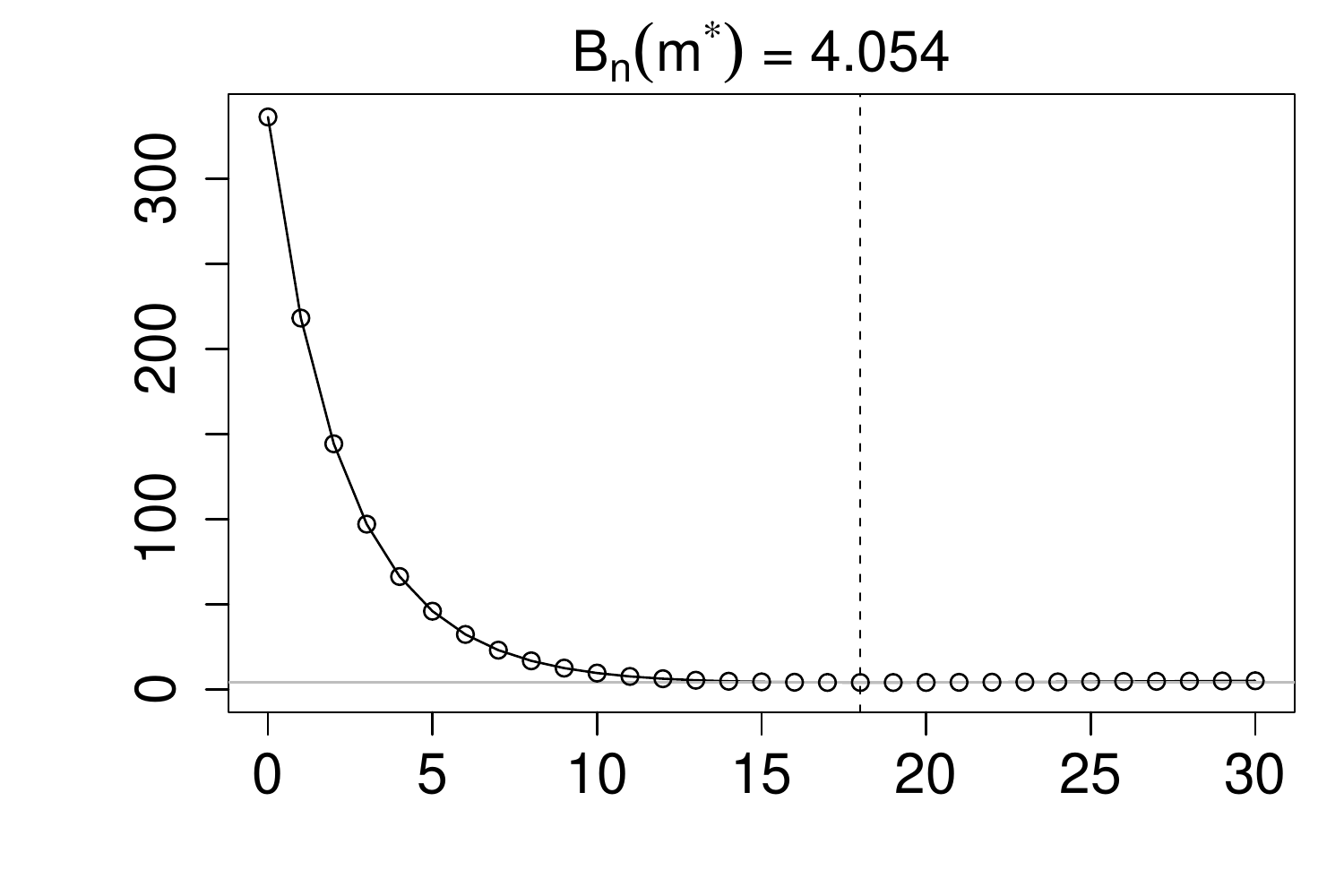}
	\end{minipage}
	%\vspace{-0.4cm}
	\caption{\it Value of the bound from Theorem~\ref{general_theorem} for empirical autocovariances of lag $k = 0$ and for an AR(1) process with $\alpha = 0.7$ as a function of $m$. The dashed vertical line indicates $m^*$ where the minimum is achieved. The gray horizontal line indicates the minimum value.  Top, middle and bottom row show $\varepsilon(t) \sim \mathcal{N}(0,1)$, $\varepsilon(t) \sim \sqrt{12/14} \ t_{14}$ and $\varepsilon(t) \sim \sqrt{7/9} \ t_{9}$, respectively. The left, center, and right columns show $n = 25, 500$, and $2000$, respectively.}
	\label{fig:bnd_fct_m}
\end{figure}

% latex table generated in R 4.1.0 by xtable 1.8-4 package
% Mon Jun 14 18:07:26 2021
\begin{table}[t]
	\caption{Value of the bound from Theorem~\ref{general_theorem} in combination with \eqref{bnd:Smn_new_bound}, with $m = m^*$ to minimise the bound as described in Section~\ref{sec:comp_bnd}, for empirical autocovariances, for a range of lags $k$ and sample sizes $n$. The data stems from an AR(1) process with $\varepsilon(t) \sim \sqrt{7/9} \ t_{9}$ where $\alpha$ takes a range of values. \label{tab:bnd_t9}}
	\centering
	\begin{tabular}{ll |r@{\extracolsep{0.3cm}}r@{\extracolsep{0.3cm}}r@{\extracolsep{0.1cm}}r@{\extracolsep{0.3cm}}r@{\extracolsep{0.3cm}}r@{\extracolsep{0.3cm}}r@{\extracolsep{0.2cm}}r@{\extracolsep{0.1cm}}r@{\extracolsep{0.2cm}}r}
		\hline
		$k$ & $\alpha$ $\vert$ $n$ & \multicolumn{1}{l}{25} & \multicolumn{1}{l}{50} & \multicolumn{1}{l}{75} & \multicolumn{1}{l}{100} & \multicolumn{1}{l}{150} & \multicolumn{1}{l}{    200} & \multicolumn{1}{l}{250} & \multicolumn{1}{l}{500} & \multicolumn{1}{l}{1000} & \multicolumn{1}{l}{2000} \\ 
		\hline
		0 & 0            &  0.912 &  0.645 &  0.527 &  0.456 &  0.372 &  0.322 &  0.288 &  0.204 &  0.144 &  0.102 \\ 
		& 0.1          & 11.003 &  9.294 &  8.822 &  8.707 &  7.509 &  6.658 &  6.091 &  4.779 &  3.770 &  2.773 \\ 
		& 0.3          & 16.088 & 12.932 & 11.481 & 10.287 &  8.937 &  8.192 &  7.484 &  5.751 &  4.386 &  3.375 \\ 
		& 0.5          & 16.952 & 13.518 & 11.760 & 10.689 &  9.245 &  8.365 &  7.706 &  5.955 &  4.579 &  3.514 \\ 
		& 0.7          & 16.871 & 14.042 & 12.434 & 11.367 &  9.945 &  9.018 &  8.343 &  6.531 &  5.087 &  3.961 \\ 
		1 & 0            &  2.564 &  1.818 &  1.485 &  1.286 &  1.050 &  0.909 &  0.813 &  0.574 &  0.406 &  0.287 \\ 
		& 0.1          &  7.711 &  5.701 &  4.808 &  4.285 &  3.686 &  3.350 &  3.118 &  2.283 &  1.708 &  1.326 \\ 
		& 0.3          &  9.811 &  7.567 &  6.601 &  5.939 &  5.063 &  4.554 &  4.217 &  3.213 &  2.491 &  1.908 \\ 
		& 0.5          & 12.512 &  9.980 &  8.716 &  7.861 &  6.825 &  6.133 &  5.671 &  4.394 &  3.402 &  2.641 \\ 
		& 0.7          & 14.968 & 12.828 & 11.376 & 10.387 &  9.092 &  8.247 &  7.644 &  5.983 &  4.668 &  3.647 \\ 
		2 & 0            &  4.088 &  2.916 &  2.385 &  2.067 &  1.688 &  1.462 &  1.308 &  0.924 &  0.653 &  0.462 \\ 
		& 0.1          & 10.398 &  7.659 &  6.409 &  5.667 &  4.804 &  4.309 &  3.925 &  2.833 &  2.074 &  1.561 \\ 
		& 0.3          & 10.801 &  8.273 &  7.175 &  6.405 &  5.424 &  4.850 &  4.467 &  3.353 &  2.557 &  1.913 \\ 
		& 0.5          & 12.211 &  9.739 &  8.459 &  7.632 &  6.592 &  5.921 &  5.471 &  4.210 &  3.233 &  2.486 \\ 
		& 0.7          & 14.392 & 12.610 & 11.215 & 10.236 &  8.954 &  8.111 &  7.509 &  5.866 &  4.565 &  3.556 \\ 
		\hline
	\end{tabular}
	
	\vspace{1cm}
	
	\caption{Value of the true 1-Wasserstein distance considered in Theorem~\ref{general_theorem} for empirical autocovariances, for a range of lags $k$ and sample sizes $n$. The data stems from an AR(1) process with $\varepsilon(t) \sim \sqrt{7/9} \ t_{9}$ where $\alpha$ takes a range of values. \label{tab:W1_t9}}
	\centering
	\begin{tabular}{ll |rrrrrrrrrr}
		\hline
		$k$ & $\alpha$ $\vert$ $n$ & \multicolumn{1}{l}{    25} & \multicolumn{1}{l}{    50} & \multicolumn{1}{l}{    75} & \multicolumn{1}{l}{   100} & \multicolumn{1}{l}{   150} & \multicolumn{1}{l}{   200} & \multicolumn{1}{l}{   250} & \multicolumn{1}{l}{   500} & \multicolumn{1}{l}{  1000} & \multicolumn{1}{l}{  2000} \\ 
		\hline
		0 & 0            & 0.288 & 0.218 & 0.184 & 0.163 & 0.136 & 0.120 & 0.109 & 0.080 & 0.058 & 0.041 \\ 
		& 0.1          & 0.294 & 0.222 & 0.188 & 0.166 & 0.139 & 0.123 & 0.111 & 0.081 & 0.059 & 0.042 \\ 
		& 0.3          & 0.354 & 0.266 & 0.224 & 0.198 & 0.165 & 0.145 & 0.131 & 0.095 & 0.069 & 0.049 \\ 
		& 0.5          & 0.536 & 0.401 & 0.336 & 0.296 & 0.246 & 0.216 & 0.194 & 0.140 & 0.101 & 0.072 \\ 
		& 0.7          & 1.185 & 0.891 & 0.746 & 0.655 & 0.544 & 0.475 & 0.428 & 0.307 & 0.219 & 0.156 \\ 
		1 & 0            & 0.072 & 0.040 & 0.028 & 0.021 & 0.015 & 0.011 & 0.009 & 0.005 & 0.002 & 0.001 \\ 
		& 0.1          & 0.103 & 0.069 & 0.055 & 0.047 & 0.038 & 0.032 & 0.029 & 0.020 & 0.014 & 0.010 \\ 
		& 0.3          & 0.256 & 0.187 & 0.155 & 0.135 & 0.111 & 0.097 & 0.087 & 0.062 & 0.044 & 0.031 \\ 
		& 0.5          & 0.524 & 0.384 & 0.319 & 0.279 & 0.230 & 0.200 & 0.180 & 0.128 & 0.091 & 0.065 \\ 
		& 0.7          & 1.282 & 0.951 & 0.791 & 0.693 & 0.572 & 0.499 & 0.448 & 0.320 & 0.227 & 0.161 \\ 
		2 & 0            & 0.083 & 0.045 & 0.031 & 0.024 & 0.016 & 0.013 & 0.010 & 0.005 & 0.003 & 0.001 \\ 
		& 0.1          & 0.088 & 0.049 & 0.034 & 0.026 & 0.018 & 0.014 & 0.012 & 0.006 & 0.004 & 0.002 \\ 
		& 0.3          & 0.167 & 0.113 & 0.091 & 0.078 & 0.063 & 0.055 & 0.049 & 0.034 & 0.024 & 0.017 \\ 
		& 0.5          & 0.449 & 0.329 & 0.272 & 0.237 & 0.195 & 0.170 & 0.152 & 0.109 & 0.077 & 0.055 \\ 
		& 0.7          & 1.307 & 0.966 & 0.802 & 0.701 & 0.578 & 0.504 & 0.452 & 0.322 & 0.229 & 0.162 \\ 
		\hline
	\end{tabular} 
\end{table}

In Table~\ref{tab:bnd_t9} the values of the bound $B_n(m^*)$ for different values of $k$, $\alpha$, and $n$ are shown for the case where $\varepsilon(t) \sim \sqrt{7/9} \ t_{9}$. The numbers for the cases where $\varepsilon(t) \sim \mathcal{N}(0,1)$ or $\varepsilon(t) \sim \sqrt{12/14} \ t_{14}$ are shown in Tables \ref{tab:bnd_N01} and \ref{tab:bnd_t14}, respectively, in Section \ref{sec:adSimRes}.
We chose to present the case with the heaviest tails in the main paper, because in this case the convergence of the estimator of the autocovariance and cross-covariance functions to the Gaussian limit is the slowest.
We have omitted considering negative $\alpha$, because in the case considered the results are the same as for $-\alpha$.
It can be seen that the value of the bound increases as $|\alpha|$ increases.
Comparing the bounds across tables we see that for most cases the value of the bound is larger for heavier tails.
It can be seen that the value of the bound decreases as $n$ increases.\enlargethispage{1cm}

\FloatBarrier

For comparison with our bound, as displayed in Tables~\ref{tab:bnd_t9}, \ref{tab:bnd_N01}, and \ref{tab:bnd_t14}, we also present simulated numbers for the true Wasserstein distance in Tables~\ref{tab:W1_t9}, \ref{tab:W1_N01}, and \ref{tab:W1_t14}, with Tables \ref{tab:bnd_N01}--\ref{tab:W1_t14} shown in Section \ref{sec:adSimRes}. Additional details about simulation of the true Wasserstein distance are deferred to Section \ref{sec:SimW1}. By inspection of the numbers, it can be seen that, as expected, our bound is always larger than the true Wasserstein distance obtained by simulation.

\section{Proof of Theorem~\ref{general_theorem_nonstat}}
\label{sec:Proofs}

Before the main proof of this section, we discuss a useful lemma that summarizes the Stein's method result used in this paper which is applicable to a general local dependence condition. Consider a set of random variables $\left\lbrace \xi_i, i \in J \right\rbrace$, for a finite index set $J$.
%For any $A \subset J$ denote $A^{c} = \left\lbrace i \in J : i \notin A \right\rbrace$.
Then, the local dependence condition is
\begin{itemize}
	\item[(LD)] For each $i \in J$ there exist $A_i \subset B_i\subset J$ such that $\xi_i$ is independent of $\left\lbrace \xi_j : j \notin A_i \right\rbrace$ and $\left\lbrace \xi_j : j \in A_i \right\rbrace$ is independent of $\left\lbrace \xi_k : k \notin B_i \right\rbrace$.
\end{itemize}
For any $A \subset J$, we now denote by
\begin{equation}
\label{tau_eta}
\xi(A) = \sum_{j\in A}^{}\xi_j.
\end{equation}
\begin{remark}Consider an $m$-dependent sequence of random variables $X_1, \ldots, X_n$. Then the sets of random variables $\left\lbrace X_j : j\leq i \right\rbrace$ and $\left\lbrace X_j : j>i+m \right\rbrace$ are independent for each\linebreak $i = 1,\ldots,n$.
	Thus, (LD) is satisfied with $J := \{1,\ldots,n\}$, $A_i := \{\ell \in J : |\ell - i| \leq m\}$, and $B_i := \{\ell \in J : |\ell - i| \leq 2m\}$. 
\end{remark}
The following lemma gives an upper bound on the Wasserstein distance between the distribution of a sum of random variables satisfying Condition (LD) above and the normal distribution. The random variables are assumed to have mean zero and the variance is not necessarily equal to one. The proof is in Section \ref{app:proofs1} and is based on the steps followed for the proof of Theorem 4.13 in p.134 of \cite{Chen_et_al_book2011}.
\begin{lemma}
	\label{Lemma_locally_dependent}
	Let $\left\lbrace\xi_i, i \in J\right\rbrace$ be an $\R$-valued random field with mean zero, satisfying Condition (LD). Denote $W := \sum_{i \in J}\xi_i$ and assume that $0 < \sigma^2 := {\rm var}(W) < \infty$.
	Then, with $\xi(A)$ as in \eqref{tau_eta}, we have for the Wasserstein distance, $d_{{\mathrm{W}}}$, defined in \eqref{classes_functions}, that
	\begin{equation}
	\nonumber d_{{\mathrm{W}}}(\mathcal{L}(W),\mathcal{N}(0,\sigma^2)) \leq \frac{2}{\sigma^3}\sum_{i \in J}\left\lbrace \E\left|\left(\xi_i\xi(A_i) - \E(\xi_i\xi(A_i))\right)\xi(B_i)\right| + \frac{1}{2}\E\left|\xi_i(\xi(A_i))^2\right| \right\rbrace.
	\end{equation}
\end{lemma}

\smallskip

\noindent
{\textbf{Proof of Theorem~\ref{general_theorem_nonstat}}}.
Firstly, for $Z^{*}(t) = \cO{X}{a}{t+k}\cO{X}{b}{t} - \frac{n}{n-k} \gamma$ we have that
\begin{equation}
\label{new_representation}
\sqrt{n}\left(\cPest{\gamma}{ab}{k} - \gamma\right) = \sqrt{n}\left(\frac{1}{n}\sum_{t=1}^{n-k}\cO{X}{a}{t+k}\cO{X}{b}{t} - \gamma\right) = \frac{1}{\sqrt{n}}\sum_{t=1}^{n-k} Z^{*}(t).
\end{equation}
For $\tilde{Z}(t) = \cO{Y}{a}{t+k}\cO{Y}{b}{t} - \E[\cO{Y}{a}{t+k}\cO{Y}{b}{t}]$, the triangle inequality and \eqref{new_representation} yield
\begin{equation*}
d_{{\mathrm{W}}}\left(\mathcal{L}\left(\sqrt{n}\left(\cPest{\gamma}{ab}{k} - \gamma\right)\right),\mathcal{N}(0,\sigma^2)\right) = d_{{\mathrm{W}}}\left(\mathcal{L}\left(\frac{1}{\sqrt{n}}\sum_{t=1}^{n-k} Z^{*}(t)\right), \mathcal{N}(0, \sigma^2)\right) \leq D_1 + D_2 + D_3,
\end{equation*}
where
\begin{align}
& \label{T1} D_1 := d_{{\mathrm{W}}}\left(\mathcal{L}\left(\frac{1}{\sqrt{n}}\sum_{t=1}^{n-k} Z^{*}(t)\right), \mathcal{L}\left(\frac{1}{\sqrt{n}}\sum_{t=1}^{n-k} \tilde{Z}(t)\right) \right),\\
\label{T2}
& D_2 := d_{{\mathrm{W}}}\left(\mathcal{L}\left(\frac{1}{\sqrt{n}}\sum_{t=1}^{n-k} \tilde{Z}(t)\right), \mathcal{N}\left(0, \tilde{\Sigma}_{ab}(k)\right)\right), \\
\label{T3}
& D_3 := d_{{\mathrm{W}}}\left(\mathcal{N}\left(0, \tilde{\Sigma}_{ab}(k)\right), \mathcal{N}\left(0, \sigma^2\right)\right),
\end{align}
where $\tilde{\Sigma}_{ab}(k)$ is as in \eqref{def:tildeSigma}. We now proceed to find upper bounds for \eqref{T1}, \eqref{T2} and \eqref{T3}.

\smallskip

\noindent
{\textbf{Bound for \eqref{T1}}:} With $h \in \mathcal{H}$, since $\|h\|_{\mathrm{Lip}} \leq 1$, then
%\begin{equation*}
%h\left(\frac{1}{\sqrt{n}}\sum_{t=1}^{n-k} Z^{*}_t \right) = h\left(\frac{1}{\sqrt{n}}\sum_{t=1}^{n-k} \tilde Z_t \right) + \frac{1}{\sqrt{n}}\sum_{t=1}^{n-k}\left(Z^{*}_t - \tilde Z_t \right)h'(R),
%\end{equation*}
%where $R$ is a quantity between $\frac{1}{\sqrt{n}}\sum_{t=1}^{n-k} Z^{*}_t$ and  $\frac{1}{\sqrt{n}}\sum_{t=1}^{n-k} \tilde Z_t$. Therefore, we have that
\begin{align}
\nonumber &  \left|\E\left[h\left(\frac{1}{\sqrt{n}}\sum_{t=1}^{n-k} Z^{*}(t) \right) - h\left(\frac{1}{\sqrt{n}}\sum_{t=1}^{n-k} \tilde{Z}(t) \right)\right] \right| \leq \frac{1}{\sqrt{n}}\sum_{t=1}^{n-k}\E\left| Z^{*}(t) - \tilde{Z}(t) \right|\\
\label{T1b} & \leq \frac{1}{\sqrt{n}}\sum_{t=1}^{n-k}\E\left|\cO{X}{a}{t+k}\cO{X}{b}{t} - \cO{Y}{a}{t+k}\cO{Y}{b}{t}\right| \\
\nonumber & \quad + \frac{1}{\sqrt{n}}\sum_{t=1}^{n-k}\left| \frac{n}{n-k} \gamma - \E[\cO{Y}{a}{t+k}\cO{Y}{b}{t}]\right|.
\end{align}
Next, we bound $\eqref{T1b} \leq \frac{1}{\sqrt{n}}\sum_{t=1}^{n-k} K_t$, by a telescoping sum argument, made precise by Lemma~\ref{lem:approx_jnt_moments_nonstat} with $\alpha = 1$, $ p = 2$, $X_1 := \cO{X}{a}{t+k}$, $X_2 := \cO{X}{b}{t}$, $Y_1 := \cO{Y}{a}{t+k}$, and $Y_2 := \cO{Y}{b}{t}$. Therefore, we have
\begin{equation}
\label{bound_T1}
\eqref{T1} \leq \frac{1}{\sqrt{n}}\sum_{t=1}^{n-k} K_t + \frac{1}{\sqrt{n}}\sum_{t=1}^{n-k}\left| \frac{n}{n-k} \gamma - \E[\cO{Y}{a}{t+k}\cO{Y}{b}{t}]\right|.
\end{equation}
%since $h \in \mathcal{H}_{\mathrm{W}}$ of \eqref{classes_functions} implies that $\|h\|_{\mathrm{Lip}} \leq 1$.

\smallskip

\noindent
{\textbf{Bound for \eqref{T2}}:} Here we use Stein's method. Let $W := \frac{1}{\sqrt{n}}\sum_{t=1}^{n-k}\tilde{Z}(t)$ and note that $\E(\tilde{Z}(t)) = 0$ and ${\rm var}(W) = \tilde{\Sigma}_{ab}(t)$. Lemma~\ref{Lemma_locally_dependent} then yields that
\begin{equation}
\label{bound_second_term}
\eqref{T2} \leq \frac{2}{\left(\tilde{\Sigma}_{ab}(k)\right)^{3/2}}\sum_{t=1}^{n-k} \tilde Q_t.
\end{equation}

\smallskip

\noindent
{\textbf{Bound for \eqref{T3}}:} 
Using the results in pages 69 and~70 of \cite{Malliavin_Calculus}, it is straightforward to conclude that
\begin{equation}
\label{bound_third_term}
\eqref{T3} \leq \frac{\sqrt{2}}{\sqrt{\pi\cO{\Sigma}{ab}{k}}}\left|\cO{\Sigma}{ab}{k} - \tilde{\Sigma}_{ab}(k)\right|.
\end{equation}
Combining \eqref{bound_T1}, \eqref{bound_second_term} and \eqref{bound_third_term}, yields the result of Theorem \ref{general_theorem_nonstat} as in \eqref{final_upper_bound_nonstat}.
\hfill$\square$

\section{Discussion}
\label{sec:Discussion}
In this article, we have obtained upper bounds on the Wasserstein distance between the true distribution of the estimator of the autocovariance and cross-covariance functions and their limiting Gaussian distribution for non-stationary and stationary data. Compared with existing results in the literature for general linear statistics and apart from the machinery employed (partly based on Stein's method) and the proof methodology followed, the results presented in this paper are novel in three main aspects. Firstly, the results of the paper are applicable to non-stationary data sequences; see Theorem~\ref{general_theorem_nonstat}. Secondly, our focus is not only on rates of convergence, but the derived bounds are fully explicit in terms of the sample size, the lag and the constants depending on the time series model. This allows to compute the bound in examples. Thirdly, the assumptions that we have used are non-restrictive, and they are partly based on an $m$-dependence approximation of the original time series, which is convenient to work with in applications. In contrast, existing results are focused on rather more restrictive structures that are often probabilistic in nature and difficult to verify in practice, such as the case of strong and $\varphi$-mixing conditions or the discrete time martingales setting.
%\section*{Acknowledgement}
%AA would like to thank the School of Business and Economics of Humboldt-Universit\"{a}t zu Berlin for the kind hospitality, where work on this project began.

\bibliography{paper}

%%%%%%%%%%%%%%%%%%%%%%%%%%%%%%%%%%%%%%%%%%%%%%
%% Supplementary Material, including data   %%
%% sets and code, should be provided in     %%
%% {supplement} environment with title      %%
%% and short description. It cannot be      %%
%% available exclusively as external link.  %%
%% All Supplementary Material must be       %%
%% available to the reader on Project       %%
%% Euclid with the published article.       %%
%%%%%%%%%%%%%%%%%%%%%%%%%%%%%%%%%%%%%%%%%%%%%%
%\begin{supplement}
%\stitle{Supplementary material for “Finite sample distributional error bounds for empirical autocovariances and cross-covariances under non-stationarity and stationarity”}
%\sdescription{
%	In the supplementary material we provide technical details regarding computation and simulation as well as tables with additional results that were omitted from Section~\ref{sec:expl}. We also provide step-by-step proofs for the results that were not included in the main text.
%}
%\end{supplement}

\clearpage

{\noindent\Huge\textbf{Supporting Information for\medskip\\
		\Large``Wasserstein distance bounds on the normal approximation of empirical autocovariances and cross-covariances under non-stationarity and stationarity''}}
\appendix
\section{Technical Details regarding Section~\ref{sec:bound_mdep_appr_known}}\label{app:tech_details:Sec23}
Here, we explain how the right-hand side of~\eqref{bnd:Smn_new_bound} can be computed from cumulants, up to order~8, of $\seq{\vO{Y}{t}}$ that is centered and weakly stationary. First, note that
\[ \E \Big( \tilde{Z}(t) \tilde{Z}( A_t) \Big)^2 = \Var \Big( \tilde{Z}(t) \tilde{Z}( A_t) \Big) + \Bigg[\E \Big( \tilde{Z}(t) \tilde{Z}( A_t) \Big)\Bigg]^2,\]
and
\begin{equation*}
\begin{split}
& \Var \Big( \tilde{Z}(t)  \tilde{Z}( A_t) \Big)
= \sum_{j_1 \in A_t} \sum_{j_2 \in A_t} \Cov \big( \tilde{Z}(t) \tilde{Z}(j_1), \tilde{Z}(t) \tilde{Z}(j_2) \big) \\
&  = \sum_{j_1 \in A_t} \sum_{j_2 \in A_t} \left[ \cum ( \tilde{Z}(t), \tilde{Z}(j_1), \tilde{Z}(t), \tilde{Z}(j_2) )
+ \cum ( \tilde{Z}(t), \tilde{Z}(t)) \cum ( \tilde{Z}(j_1), \tilde{Z}(j_2) )\right.\\
&\qquad\qquad\qquad\left. + \cum ( \tilde{Z}(t), \tilde{Z}(j_2)) \cum ( \tilde{Z}(j_1), \tilde{Z}(t) ) \right] \\
& = \sum_{j_1 \in A_t} \sum_{j_2 \in A_t} \cum ( \tilde{Z}(t), \tilde{Z}(t), \tilde{Z}(j_1), \tilde{Z}(j_2) )
+ \tilde C(0) \Var\Big(\sum_{j \in A_t} \tilde{Z}(j) \Big)
+ \Big[ \E \Big( \tilde{Z}(t) \sum_{j \in A_t} \tilde{Z}(j) \Big) \Big]^2,
\end{split}
\end{equation*}
where
\begin{equation}\label{def:Cu}
\begin{split}
\tilde C(u) := \cum(\tilde{Z}(0), \tilde{Z}(u))
& = \cum(\cO{Y}{a}{k}, \cO{Y}{b}{0}, \cO{Y}{a}{u+k}, \cO{Y}{b}{u}) \\
& \quad + \cO{\tilde\gamma}{aa}{u} \cO{\tilde\gamma}{bb}{u}
+ \cO{\tilde\gamma}{ab}{k-u} \cO{\tilde\gamma}{ab}{k+u},
\end{split}
\end{equation}
with $\cO{\tilde\gamma}{ab}{k} := \cum(\cO{Y}{a}{k}, \cO{Y}{b}{0})$.

Noting that $A_t$ and $B_t$ in \eqref{tildeS_mn} are finite sets of consecutive integers of the form $\{s, \ldots, e\}$, we conclude that it suffices to compute or bound $\tilde C(0)$ and the following three quantities
\begin{equation}\label{bnd_parts}
\sum_{j_1=s}^e \sum_{j_2=s}^e \cum ( \tilde{Z}(t), \tilde{Z}(t), \tilde{Z}(j_1), \tilde{Z}(j_2) ), \quad 
\Var\Big(\sum_{j=s}^e \tilde{Z}(j) \Big), \quad \text{and} \quad
\E \Big( \tilde{Z}(t) \sum_{j=s}^e \tilde{Z}(j) \Big),
\end{equation}
where $s \leq t \leq e$. The second and the third quantity in~\eqref{bnd_parts} only depend on second and fourth order moments of $\vO{Y}{t}$ and, in this sense, are easier to obtain. We have that
\begin{equation*}
\begin{split}
\Var\Big(\sum_{j=s}^e \tilde{Z}(j) \Big)
& = \sum_{j_1=s}^e \sum_{j_2=s}^e \Cov( \tilde{Z}(0), \tilde{Z}(j_2 - j_1))
= \sum_{|u| \leq e-s} (e-s+1-|u|) \Cov( \tilde{Z}(0), \tilde{Z}(u) ) \\
& = (e-s+1) \sum_{|u| \leq e-s} \Big(1-\frac{|u|}{e-s+1}\Big) \tilde C(u),
\end{split}
\end{equation*}
with $\tilde C(u)$ as defined in~\eqref{def:Cu},
and similarly
\begin{equation*}
\begin{split}
\E \Big( \tilde{Z}(t) \sum_{j=s}^e \tilde{Z}(j) \Big)
& = \sum_{j=s}^e \Cov(\tilde{Z}(0), \tilde{Z}(j-t)) = \sum_{u=s-t}^{e-t} \Cov(\tilde{Z}(0), \tilde{Z}(u))
= \sum_{u=s-t}^{e-t} \tilde C(u).
\end{split}
\end{equation*}
For the first term in~\eqref{bnd_parts}, we have, by Theorem 2.3.2 in \cite{Brillinger1975}, that
\begin{equation}\label{cumEight}
\begin{split}
& \sum_{j_1=s}^e \sum_{j_2=s}^e \cum (\tilde{Z}(t), \tilde{Z}(t), \tilde{Z}(j_1), \tilde{Z}(j_2) )
= \sum_{j_1=s}^e \sum_{j_2=s}^e \sum_{\nu} \prod_{r=1}^R \cum ( \tilde Y_{i,j} : (i,j) \in \nu_r )
\end{split}
\end{equation}
where $\tilde Y_{i,j}$ are defined in Table~\ref{tab1}(b) and the sum is with respect to all indecomposable partitions $\nu := \{\nu_1, \ldots, \nu_R\}$, where $|\nu_r| \geq 2$, of Table~\ref{tab1}(a).
\begin{table}[t]
	\centering
	\begin{subtable}[h]{0.35\textwidth}
		\centering
		\begin{tabular}{cc}
			\toprule
			(1,1) & (1,2) \\
			(2,1) & (2,2) \\
			(3,1) & (3,2) \\
			(4,1) & (4,2) \\
			\bottomrule
		\end{tabular}
		\caption{Table of indices to be partitioned}
		\label{tab:1a}
	\end{subtable}
	\hspace{1cm}
	\begin{subtable}[h]{0.55\textwidth}
		\centering
		\begin{tabular}{ll}
			\toprule
			$\tilde Y_{1,1} := \cO{Y}{a}{t+k}$   & $\tilde Y_{1,2} := \cO{Y}{b}{t}$ \\
			$\tilde Y_{2,1} := \cO{Y}{a}{t+k}$   & $\tilde Y_{2,2} := \cO{Y}{b}{t}$ \\
			$\tilde Y_{3,1} := \cO{Y}{a}{j_1+k}$ & $\tilde Y_{3,2} := \cO{Y}{b}{j_1}$ \\
			$\tilde Y_{4,1} := \cO{Y}{a}{j_2+k}$ & $\tilde Y_{4,2} := \cO{Y}{b}{j_2}$ \\
			\bottomrule
		\end{tabular}
		\caption{Variables of which the cumulants are considered}
		\label{tab1b}
	\end{subtable}
	\caption{Table of indices and variables to be partitioned for the addends of the sum in~\eqref{cumEight}.\label{tab1}}
\end{table}

The number of ways to partition a table with 8 entries is the 8th Bell number $B_8 = 4140$. It is straight forward to enumerate all partitions algorithmically, for example using the \proglang{R} function \code{all_indecomposable_partitions()} that we provide as part of the replication package, and to verify that only 545 of them are indecomposable and have at least two elements for each of the sets of the partitions and are therefore the ones to consider for distributions with $\E(\vO{Y}{t}) = 0$. For the case where all cumulants of odd orders vanish, such as in the examples of Section~\ref{sec:expl}, we only consider partitions where the number of elements in the sets is even and there are 249 such partitions. Further, for the case where $\vO{Y}{t}$ is centered and normally distributed, only indecomposable partitions with sets of exactly 2 elements will be considered and there are 48 of them.

\section{Technical details regarding Section \ref{sec:bound_original_data}}
\label{app:lemmas_for_S2.5}
We provide Lemmas~\ref{lem:bnd_var_mdep_est} and~\ref{lem:bnd1_Smn} that can be used to yield a bound on~\eqref{final_upper_bound} in terms of quantities defined only via $\seq{\vO{X}{t}}$ and the approximation errors $\tilde D_{j}^{(q)}$ from Assumption~\ref{as:m_dep_appr}. For the statement of the results in this section, we will define, in~\eqref{Kqm}, the quantities $\tilde K_{p}^{(\alpha)}$. Next, we state Lemma~\ref{lem:approx_jnt_moments} which asserts that $\tilde K_{p}^{(\alpha)}$ is a bound for the approximation error (measured in~$L^{\alpha}$) of products of $\seq{\vO{X}{t}}$ by the corresponding products of $\seq{\vO{Y}{t}}$.

\begin{lemma}\label{lem:approx_jnt_moments}
	Let $\alpha \geq 1$, $p \in \IN$, $a_1, \ldots,a_p \in \{1,\ldots,d\}$, and $t_1, \ldots, t_p \in \IZ$.
	In addition, let $\seq{\vO{X}{t}}$ be $d$-variate, centered and weakly stationary, and assume $\| \vO{X}{0} \|_{\alpha p} < \infty$.
	Grant Assumption~\ref{as:m_dep_appr} with $q = \alpha p$.
	Define
	\begin{equation}
	\label{Kqm}
	\tilde K^{(\alpha)}_{p} = \tilde K^{(\alpha)}_{p}(\{a_1, \ldots, a_p\}):= \sum_{(\ell_1, \ldots, \ell_p) \in \Lambda_p} \prod_{i=1}^p \left(\tilde D_{a_i}^{(\alpha p)} \right)^{\ell_i} \| \cO{X}{a_i}{0} \|_{\alpha p}^{1-\ell_i},
	\end{equation}
	where $\tilde D_{a_i}^{(\alpha p)}$ is as in \eqref{A:mdep}, $\Lambda_1 := \{1\}$, $\Lambda_2 := \{(1,0), (1,1), (0,1)\}$ and
	\[\Lambda_p := \{1\} \times \{0,1\}^{p-1}
	\cup \bigcup_{j=2}^{p-1} \{0\}^{j-1} \times \{1\} \times \{0,1\}^{p-j}
	\cup \{0\}^{p-1} \times \{1\}, \quad p = 3, 4, \ldots.\]
	Then, we have that	
	\begin{equation}\label{eqn:approx_jnt_q_moments}
	\Big\| \prod_{i=1}^p \cO{X}{a_i}{t_i} - \prod_{i=1}^p \cO{Y}{a_i}{t_i} \Big\|_{\alpha}
	\leq \tilde K^{(\alpha)}_{p}.
	\end{equation}
\end{lemma}
The proof of Lemma~\ref{lem:approx_jnt_moments} is available in Section \ref{app:proofs:jnt_moments}.
The case where $p := 2$, $a_1 := a$, and $a_2 := b$ with 
\[\tilde K_{2}^{(\alpha)} := \tilde K_{2}^{(\alpha)}(\{a,b\}) := \tilde D_{a}^{(2 \alpha)} \| \cO{X}{b}{0} \|_{2\alpha} + \tilde D_{a}^{(2 \alpha)} \tilde D_{b}^{(2 \alpha)} + \tilde D_{b}^{(2 \alpha)} \| \cO{X}{a}{0} \|_{2\alpha}, \]
appears repeatedly throughout the proofs, lemmas that follow, and also in the main theorems.

\begin{remark}\label{rem:Kp}
	(i) The quantity $\tilde K^{(\alpha)}_{p}$ in \eqref{Kqm} depends on the set of indices, but it does not depend on their order.
	This can be seen from the fact that the set $\Lambda_p$ is closed under permutations of the indices, which is the case as $\Lambda_p$ is closed under transpositions (swapping any two indices yields another element of $\Lambda_p$). Thus, the bound has the desirable property to remain the same regardless of the order of the indices~$a_i$ on the left-hand side of~\eqref{eqn:approx_jnt_q_moments}.\\
	%Otherwise we would have considered the minimum of the bound with respect to all permutations of the indices to improve the bound.\\
	(ii) In applications, $\tilde D_{a}^{(\alpha p)} $ will typically be of order $o(1)$ and $\| \cO{X}{a}{0} \|_{\alpha p} = \mathcal{O}(1)$, as \hbox{$m \rightarrow \infty$}. Therefore, $\tilde K^{(\alpha)}_{p,m} = \mathcal{O}( \max_{a \in \{a_1, \ldots, a_q\}} \tilde D_{a}^{(\alpha p)} )$ which can be seen from the fact that each element of $\Lambda_p$ has at least one index 1.\\
	(iii) To obtain a bound on the closeness of joint moments of $\seq{\vO{X}{t}}$ to the joint moments of the $m$-dependent sequence $\seq{\vO{Y}{t}}$ we may employ Lemma~\ref{lem:approx_jnt_moments} with $\alpha = 1$
	\[\Bigg| \E\Big[\prod_{i=1}^p \cO{X}{a_i}{t_i} \Big] - \E\Big[ \prod_{i=1}^p \cO{Y}{a_i}{t_i}\Big] \Bigg| \leq \Big\| \prod_{i=1}^p \cO{X}{a_i}{t_i} - \prod_{i=1}^p \cO{Y}{a_i}{t_i} \Big\|_{1} \leq \tilde K_{p}^{(1)}. \]
\end{remark}

The second term of~\eqref{final_upper_bound} depends on $\seq{\vO{Y}{t}}$ via $\tilde{\Sigma}_{ab}(k)$. 
Lemma~\ref{lem:bnd_var_mdep_est}, below, quantifies the error made by replacing $\tilde{\Sigma}_{ab}(k)$, defined in terms of $\seq{\vO{Y}{t}}$, by $n \Var\big(\cPest{\gamma}{a b}{k}\big)$ defined in terms of $\seq{\vO{X}{t}}$.
The proof of Lemma~\ref{lem:bnd_var_mdep_est} is available in Section \ref{app:proofs:bnd_var_mdep_est}.
\begin{lemma}\label{lem:bnd_var_mdep_est}
	Grant Assumption~\ref{as:m_dep_appr} with $q=4$, then
	\begin{equation}
	\nonumber |\tilde{\Sigma}_{ab}(k) - \cP{\Sigma}{ab}{k}|
	\leq \Big| n \Var\big(\cPest{\gamma}{a b}{k}\big) - \cP{\Sigma}{a b}{k} \Big|
	+ 2 \frac{(n-k)^2}{n} \Big(\max_{i=a,b} \tilde D_{i}^{(4)} \Big) \tilde F,
	\end{equation}
	where
	\begin{equation}
	\label{eq:Fm}
	\begin{split}
	& \tilde F :=  |\cP{\gamma}{a b}{k}| \big(\|\cO{X}{a}{0}\|_2 + \|\cO{X}{b}{0}\|_2 + \min_{i=a,b} \tilde D_{i}^{(2)} \big) \\
	&\;\; + \frac{1}{2} \Big( \max_{i=a,b} \tilde D_{i}^{(2)} \Big) \big(\|\cO{X}{a}{0}\|_2 + \|\cO{X}{b}{0}\|_2 + \min_{i=a,b} \tilde D_{i}^{(2)} \big)^2 \\
	&\;\; + (\|\cO{X}{a}{0}\|_4 + \tilde D_{a}^{(4)})(\|\cO{X}{b}{0}\|_4 + \tilde D_{b}^{(4)})\\
	&\qquad\times\Big( \|\cO{X}{a}{0}\|_4 + \|\cO{X}{b}{0}\|_4 + \tilde D_{a}^{(4)} + \tilde D_{b}^{(4)} \Big). \\
	\end{split}
	\end{equation}
\end{lemma}

To bound the fourth term in~\eqref{final_upper_bound}, which depends on $\seq{\vO{Y}{t}}$ via $\tilde Q_t$ and $\tilde{\Sigma}_{ab}(k)$, we will use Lemma~\ref{lem:bnd_var_mdep_est}, above, to replace $\tilde{\Sigma}_{ab}(k)$ by $\cP{\Sigma}{ab}{k}$ and Lemma~\ref{lem:bnd1_Smn}, below, to replace $\tilde Q_t$, defined in terms of $\seq{\vO{Y}{t}}$, by $Q_t$, defined in~\eqref{notation_for_fourth_term} in terms of $\seq{\vO{X}{t}}$.
We state Lemma~\ref{lem:bnd1_Smn} first and defer the detailed discussion regarding a bound for the fourth term in~\eqref{final_upper_bound} to the end of this section.
The proof of Lemma~\ref{lem:bnd1_Smn} is available in Section \ref{app:proofs:bnd1_Smn}.

\begin{lemma}\label{lem:bnd1_Smn}
	Grant Assumption~\ref{as:m_dep_appr} with $q=6$.
	Let $\tilde{Q}_t$ and $\tilde{Z}(t)$ be as in and below~\eqref{tildeS_mn}, respectively, and $Q_t$ and $Z(t)$ be as in and below~\eqref{notation_for_fourth_term}, respectively. Then,
	\begin{equation*}
	\begin{split}
	|Q_t - \tilde{Q}_t|
	\leq \tilde K_{2}^{(3)} \Big( |A_t| |B_t| + \frac{1}{2} |A_t|^2 \Big) \tilde C_{1}
	+ \tilde K_{2}^{(2)} | A_t | |B_t| \tilde C_{2}
	+ \tilde K_{2}^{(1)}  |B_t| \tilde C_{3},
	\end{split}
	\end{equation*}
	with $\tilde K_{2}^{(\alpha)} := \tilde K_{2}^{(\alpha)}(\{a,b\})$, $\alpha = 1,2,3$, as defined in Lemma~\ref{lem:approx_jnt_moments}, and
	\begin{equation*}
	\begin{split}
	\tilde C_{1} & := 6 \big( 2 \| \cO{X}{a}{0} \|_{6} \| \cO{X}{b}{0} \|_{6} + \tilde K_{2}^{(3)} \big)^2 + 2 (\tilde K_{2}^{(3)})^2, \\
	\tilde C_{2} & := 8 \Big( 2 \| \cO{X}{a}{0} \|_{4} \| \cO{X}{b}{0} \|_{4} + \tilde K_{2}^{(2)} \Big) \Big( \| \cO{X}{a}{0} \|_{2} \| \cO{X}{b}{0} \|_{2} + \tilde K_{2}^{(1)} \Big), \\
	\tilde C_{3} & := 2 \Big| \sum_{j \in A_t} \E( Z(t)  Z(j) ) \Big|.
	\end{split}
	\end{equation*}
\end{lemma}

\begin{remark}
	Both $|Q_t - \tilde{Q}_t|$ and the bound in Lemma~\ref{lem:bnd1_Smn} depend on $m$. The bound has three addends, each of which is factored into three terms according to their behavior with respect to $m$: \\
	(i) $\tilde K_{2}^{(\alpha)}$, which will decrease as $m$ increases, in typical applications (cf.\ Remark~\ref{rem:Kp}(ii));\\
	(ii) polynomials in the variables $|A_t|$ and $|B_t|$, which increase as $m$ increases; \\
	(iii) $\tilde C_1$, $\tilde C_2$, and $\tilde C_3$, which are non-negative and typically bounded: $\tilde C_{1}$ and $\tilde C_{2}$ are bounded if $|\tilde K_{2}^{(\alpha)}|$ is bounded, and $\tilde C_{3}$ if cumulants up to order 4 of $\seq{\vO{X}{t}}$ are sumable; cf.\ discussion of the third term in \eqref{bnd_parts}.\\
	In applications, where $\tilde K_{2}^{(\alpha)} = o(1/m^2)$, the bound will vanish, uniformly with respect to $t$.
\end{remark}

Using Lemmas~\ref{lem:bnd_var_mdep_est} and~\ref{lem:bnd1_Smn}, we now bound $2 (n\tilde{\Sigma}_{ab}(k))^{-3/2} \sum_{t=1}^{n-k} \tilde Q_t$ in~\eqref{final_upper_bound} by a multiple of $2 (n\cO{\Sigma}{ab}{k})^{-3/2} \sum_{t=1}^{n-k} Q_t,$ where the last expression relies only on $\seq{\vO{X}{t}}$. We have
\begin{align}
\label{eqn:outl_bnd_tildeSmn_1}
 \frac{2 n^{-3/2}}{\left(\tilde{\Sigma}_{ab}(k)\right)^{3/2}} \sum_{t=1}^{n-k} \tilde{Q}_t
& \leq \frac{2 n^{-3/2}}{\left(\tilde{\Sigma}_{ab}(k)\right)^{3/2}} \sum_{t=1}^{n-k} Q_t + \frac{2 n^{-3/2}}{\left(\tilde{\Sigma}_{ab}(k)\right)^{3/2}} \sum_{t=1}^{n-k} |Q_t - \tilde{Q}_t|\\
\nonumber & = \left(\frac{2 n^{-3/2}}{\left(\cO{\Sigma}{ab}{k}\right)^{3/2}} \sum_{t=1}^{n-k} Q_t + \frac{2 n^{-3/2}}{\left(\cO{\Sigma}{ab}{k}\right)^{3/2}}\sum_{t=1}^{n-k}|Q_t - \tilde{Q}_t|\right) \left(\frac{\cO{\Sigma}{ab}{k}}{\tilde{\Sigma}_{ab}(k)}\right)^{3/2}.
\end{align}
If for a specific example $\tilde{\Sigma}_{ab}(k)$ is tractable, but $\tilde Q_t$ is not, then it suffices to employ Lemma~\ref{lem:bnd1_Smn} and the bound in \eqref{eqn:outl_bnd_tildeSmn_1} is usable.
In the case where $\tilde{\Sigma}_{ab}(k)$ is not tractable, but $\cO{\Sigma}{ab}{k}$ is, we additionally employ Lemma~\ref{lem:bnd_var_mdep_est} to find $n$ and $m$ large enough such that $|\tilde{\Sigma}_{ab}(k) - \cP{\Sigma}{ab}{k}| \leq C \cdot \cP{\Sigma}{ab}{k}$, for some fixed $C \in (0,1)$, as this implies that $\big(\Sigma_{ab}(k) / \tilde{\Sigma}_{ab}(k) \big)^{3/2} \leq 1/(1-C)^{3/2}$. Finally, note that we can bound $Q_t$ using either of the two strategies described in Section~\ref{sec:bound_mdep_appr_known} to bound $\tilde Q_t$ with a computable quantity. The only adaptation required to employ these methods is to replace $\seq{\vO{Y}{t}}$ in the computations by $\seq{\vO{X}{t}}$.

\section{The bound for non-centered stationary data}\label{app:non_central}
The bound in Theorem~\ref{general_theorem} is for $\cPest{\gamma}{ab}{k}$, defined in~\eqref{def:gam_hat}, for data assumed to be centered.
In practice, though, we often see $\cPpre{\gamma}{ab}{k}$, defined in~\eqref{def:gam_tilde}, with the empirical centering added for the case when the data are not centered. In this section, we rigorously describe the effect on the autocovariance and cross-covariance finite sample distribution when the empirical mean, instead of the population mean, is used for centering, which is often done in practice. Lemma~\ref{lem:diff_pre_est} below, provides a bound on the Wasserstein distance between the distributions of $\cPpre{\gamma}{ab}{k}$ and $\cPest{\gamma}{ab}{k}$, which are the autocovariance and cross-covariance functions with and without empirical centering, respectively. The lemma can be employed to get an upper bound on $d_{{\mathrm{W}}}\left( \mathcal{L}\left( \sqrt{n}\left(\cPpre{\gamma}{ab}{k} - \cP{\gamma}{ab}{k}\right) \right), \mathcal{N}\left(0,\cP{\Sigma}{ab}{k}\right) \right)$, where the data are non-centered; this is explained in Remark~\ref{rem:noncentreddata}(i).
\begin{lemma}\label{lem:diff_pre_est}
	Assume that $\seq{\vO{X}{t}}$ is a weakly stationary sequence with $\E \vO{X}{t} = 0$ and $\sum_{u=-\infty}^{\infty} |u| |\cP{\gamma}{jj}{u}| < \infty$ for $j=a,b$. Then,
	\begin{align*}
	& d_{\rm W}\big( \mathcal{L}\left( n^{1/2} (\cPpre{\gamma}{ab}{k} - \cP{\gamma}{ab}{k})\right), \mathcal{L}\left( n^{1/2} (\cPest{\gamma}{ab}{k} - \cP{\gamma}{ab}{k} ) \right) \big)\\
	&\leq n^{-1/2} \Big( \sum_{u=-\infty}^{\infty} |\cP{\gamma}{aa}{u}| \Big)^{1/2} \Big( \sum_{u=-\infty}^{\infty} |\cP{\gamma}{bb}{u}| \Big)^{1/2} \\
	&\;\;	+ \frac{|k|}{n^{3/2}} \Bigg( \sum_{u=-\infty}^{\infty} |\cP{\gamma}{aa}{u}| + \frac{1}{n} \sum_{u=-\infty}^{\infty} |u| |\cP{\gamma}{aa}{u}| \Bigg)^{1/2}\\
	& \qquad\qquad \times \Bigg( \sum_{u=-\infty}^{\infty} |\cP{\gamma}{bb}{u}| + \frac{1}{n} \sum_{u=-\infty}^{\infty} |u| |\cP{\gamma}{bb}{u}| \Bigg)^{1/2}
	\end{align*}
\end{lemma}
\begin{proof}
	Let $h \in \mathcal{H}$, as defined in~\eqref{classes_functions}. Since $\|h\|_{{\mathrm{Lip}}} \leq 1$, we have
	\begin{align}
	\nonumber & \left|h\left( n^{1/2} (\cPpre{\gamma}{ab}{k} - \cP{\gamma}{ab}{k})\right) - h\left(n^{1/2} (\cPest{\gamma}{ab}{k} - \cP{\gamma}{ab}{k})\right)\right| \leq  n^{1/2} |\cPpre{\gamma}{ab}{k} - \cPest{\gamma}{ab}{k} |\\
	\nonumber & = \frac{n-k}{n^{1/2}} \Big| \bar X_{a,1:n} \bar X_{b,1:n} - \bar X_{a,1:n} \bar X_{b,1:n-k} - \bar X_{a,k+1:n} \bar X_{b,1:n} \Big| \\
	\nonumber & = \frac{n-k}{n^{1/2}} \Big| (\bar X_{a,1:n} - \bar X_{a,k+1:n})(\bar X_{b,1:n} - \bar X_{b,1:n-k}) - \bar X_{a,k+1:n} \bar X_{b,1:n-k} \Big|.
	\end{align}
	where $\bar X_{a,u:v} := \frac{1}{v-u+1} \sum_{t=u}^v \cO{X}{a}{t}$ and analogously with $b$ instead of $a$.
	%From~\eqref{eqn:exp_pre_est} and~\eqref{eqn:diff_pre_est} we obtain, for $k=0$, that
	%\begin{equation*}
	%\begin{split}
	%& \Big| \E h\Big[ n^{1/2} (\cPpre{\gamma}{ab}{k} - \cP{\gamma}{ab}{k})\Big]
	%	- \E h\Big[ n^{1/2} (\cPest{\gamma}{ab}{k} - \cP{\gamma}{ab}{k})\Big] \Big| \\
	%& \leq \frac{n-k}{n^{1/2}} \|h'\| \E \Big| \bar X_{a,1:n} \bar X_{b,1:n} - \bar X_{a,1:n} \bar X_{b,1:n-k} - \bar X_{a,k+1:n} \bar X_{b,1:n} \Big| \\
	%& \leq \frac{n-k}{n^{1/2}} \|h'\| (\E \bar X_{a,1:n}^2 )^{1/2} (\E \bar X_{b,1:n}^2)^{1/2}
	%\leq \|h'\| \Big( \sum_{k=-\infty}^{\infty} |\cP{\gamma}{aa}{k}| \Big)^{1/2} \Big( \sum_{k=-\infty}^{\infty} |\cP{\gamma}{bb}{k}| \Big)^{1/2} \Big( \frac{n-k}{n^{3/2}}\Big),
	%\end{split}
	%\end{equation*}
	%where we have used that
	%\begin{equation*}
	%\begin{split}
	%	\E \bar X_{j,u:v}^2
	%	& = \frac{1}{(v-u+1)^2} \sum_{s=u}^v \sum_{t=u}^v \E( \cO{X}{j}{s-t} \cO{X}{j}{0}) \\
	%	& = \frac{1}{(v-u+1)^2} \sum_{k=-(v-u+1)}^{v-u+1} (v - u + 1 - |k|) \cP{\gamma}{jj}{k} \\
	%	& = \frac{1}{v-u+1} \sum_{k=-(v-u+1)}^{v-u+1} \Big(1 - \frac{|k|}{v-u+1}\Big) \cP{\gamma}{jj}{k} \\
	%	& \leq \frac{1}{v - u + 1} \sum_{k=-\infty}^{\infty} |\cP{\gamma}{jj}{k}|.
	%\end{split}
	%\end{equation*}
	From the previous equation, we obtain, for $k = 1, \ldots, n-1$, that
	\begin{equation*}
	\begin{split}
	& d_{\rm W}\big( \mathcal{L}\left( n^{1/2} (\cPpre{\gamma}{ab}{k} - \cP{\gamma}{ab}{k}) \right), \mathcal{L}\left( n^{1/2} (\cPest{\gamma}{ab}{k} - \cP{\gamma}{ab}{k} ) \right) \big) \\
	& \leq \frac{n-k}{n^{1/2}}\Big( \| \bar X_{a,1:n} - \bar X_{a,k+1:n} \|_2 \| \bar X_{b,1:n} - \bar X_{b,1:n-k} \|_2 + \|\bar X_{a,k+1:n}\|_2 \|\bar X_{b,1:n-k} \|_2\Big) \\
	& \leq \frac{n-k}{n^{1/2}} \Big(  \frac{k}{n} \Big)^2 \Bigg( \frac{n}{k(n-k)} \sum_{u=-\infty}^{\infty} |\cP{\gamma}{aa}{u}| + \frac{1}{k (n-k)} \sum_{u=-\infty}^{\infty} |u| |\cP{\gamma}{aa}{h}| \Bigg)^{1/2} \\
	& \qquad\qquad\qquad \times \Bigg( \frac{n}{k(n-k)} \sum_{u=-\infty}^{\infty} |\cP{\gamma}{bb}{u}| + \frac{1}{k (n-k)} \sum_{u=-\infty}^{\infty} |u| |\cP{\gamma}{bb}{u}| \Bigg)^{1/2} \\
	& \qquad + \frac{n-k}{n^{1/2}} \frac{1}{n-k} \Big( \sum_{u=-\infty}^{\infty} |\cP{\gamma}{aa}{u}| \Big)^{1/2} \Big( \sum_{u=-\infty}^{\infty} |\cP{\gamma}{bb}{u}| \Big)^{1/2}
	\end{split}
	\end{equation*}
	\begin{equation*}
	\begin{split}
	& = \frac{k}{n^{3/2}} \Bigg( \sum_{u=-\infty}^{\infty} |\cP{\gamma}{aa}{u}| + \frac{1}{n} \sum_{u=-\infty}^{\infty} |u| |\cP{\gamma}{aa}{u}| \Bigg)^{1/2}\\
	& \qquad\qquad \times \Bigg( \sum_{u=-\infty}^{\infty} |\cP{\gamma}{bb}{u}| + \frac{1}{n} \sum_{u=-\infty}^{\infty} |u| |\cP{\gamma}{bb}{u}| \Bigg)^{1/2} \\
	& \qquad + \frac{1}{n^{1/2}}\Big( \sum_{u=-\infty}^{\infty} |\cP{\gamma}{aa}{u}| \Big)^{1/2} \Big( \sum_{u=-\infty}^{\infty} |\cP{\gamma}{bb}{u}| \Big)^{1/2} \\
	\end{split}
	\end{equation*}
	where we have used that
	\begin{equation*}
	\begin{split}
	\E \bar X_{j,u:v}^2
	& = \frac{1}{(v-u+1)^2} \sum_{s=u}^v \sum_{t=u}^v \E( \cO{X}{j}{s-t} \cO{X}{j}{0}) \\
	& = \frac{1}{(v-u+1)^2} \sum_{\ell=-(v-u+1)}^{v-u+1} (v - u + 1 - |\ell|) \cP{\gamma}{jj}{\ell} \\
	& = \frac{1}{v-u+1} \sum_{\ell=-(v-u+1)}^{v-u+1} \Big(1 - \frac{|\ell|}{v-u+1}\Big) \cP{\gamma}{jj}{\ell}
	\\
	& \leq \frac{1}{v - u + 1} \sum_{\ell=-\infty}^{\infty} |\cP{\gamma}{jj}{\ell}|
	\end{split}
	\end{equation*}
	and the fact that for $k = 1, \ldots, n-1$,
	\begin{equation*}
	\begin{split}
	\| \bar X_{a,1:n} - \bar X_{a,k+1:n} \|_2^2
	& = \E \Big( \frac{1}{n} \sum_{i=1}^n \cO{X}{a}{i} - \frac{1}{n-k} \sum_{i=k+1}^n \cO{X}{a}{i} \Big)^2 \\
	& = \E \Big( \frac{k}{n} \frac{1}{n-k} \sum_{i=k+1}^n \cO{X}{a}{i} - \frac{k}{n} \frac{1}{k} \sum_{i=1}^k \cO{X}{a}{i} \Big)^2 \\
	& = \Big( \frac{k}{n} \Big)^2 \Big( \| \bar X_{a,k+1:n} \|_2^2 + \| \bar X_{a,1:k} \|_2^2 - 2 \Cov(\bar X_{a,k+1:n}, \bar X_{a,1:k}) \Big)\\
	& \leq \Big( \frac{k}{n} \Big)^2 \Bigg( \Big( \frac{1}{n-k} + \frac{1}{k} \Big) \Big( \sum_{u=-\infty}^{\infty} |\cP{\gamma}{aa}{u}| \Big) + 2 \frac{1}{k (n-k)} \sum_{u=1}^{\infty} u |\cP{\gamma}{aa}{u}| \Bigg) \\
	& \leq \Big( \frac{k}{n} \Big)^2 \Bigg( \frac{n}{k(n-k)} \sum_{u=-\infty}^{\infty} |\cP{\gamma}{aa}{u}| + \frac{1}{k (n-k)} \sum_{u=-\infty}^{\infty} |u| |\cP{\gamma}{aa}{u}| \Bigg).
	\end{split}
	\end{equation*}
	Similarly, we derive
	\begin{equation*}
	\begin{split}
	\| \bar X_{b,1:n} - \bar X_{b,1:n-k} \|_2^2
	& = \Big( \frac{k}{n} \Big)^2 \| \bar X_{b,n-k+1:n} - \bar X_{b,1:n-k} \|_2^2 \\
	& \leq \Big( \frac{k}{n} \Big)^2 \Bigg( \frac{n}{k(n-k)} \sum_{u=-\infty}^{\infty} |\cP{\gamma}{bb}{u}| + \frac{1}{k (n-k)} \sum_{u=-\infty}^{\infty} |u| |\cP{\gamma}{bb}{u}| \Bigg).
	\end{split}
	\end{equation*}
\end{proof}
\begin{remark}\label{rem:noncentreddata} (i) Lemma~\ref{lem:diff_pre_est} facilitates use of Theorem~\ref{general_theorem} in the context of non-centered stationary time series, say $\vO{N}{t}$. This can be seen as follows. Define $\vO{X}{t} := \vO{N}{t} - \E(\vO{N}{t})$ and note that $\cO{X}{j}{t} - \bar X_j = \cO{N}{j}{t} - \bar N_j$. Thus,
	\begin{equation*}
	\begin{split}
	\cPpre{c}{ab}{k} :=
	\frac{1}{n} \sum\limits_{t=1}^{n-k} (\cO{N}{a}{t+k} - \bar N_a) (\cO{N}{b}{t} - \bar N_b)
	&= \cPpre{\gamma}{ab}{k} =
	\frac{1}{n} \sum\limits_{t=1}^{n-k} (\cO{X}{a}{t+k} - \bar X_a) (\cO{X}{b}{t} - \bar X_b),
	\end{split}
	\end{equation*}
	where $\cPpre{\gamma}{ab}{k}$ is defined in terms of the centered process $\vO{X}{t}$ to which we can apply our results. An upper bound on $d_{{\mathrm{W}}}\left(\mathcal{L}\left(\sqrt{n}\left(\cPpre{c}{ab}{k} - \cP{\gamma}{ab}{k}\right)\right),\mathcal{N}\left(0,\cP{\Sigma}{ab}{k}\right)\right)$, is thus obtain through the triangle inequality
	\begin{align}
	\nonumber & d_{{\mathrm{W}}}\left( \mathcal{L}\left( \sqrt{n}\left(\cPpre{c}{ab}{k} - \cP{\gamma}{ab}{k}\right)\right) ,\mathcal{N}\left(0,\cP{\Sigma}{ab}{k}\right)\right) \\
	& \leq d_{\rm W}\left( \mathcal{L}\left( \sqrt{n}(\cPpre{\gamma}{ab}{k} - \cP{\gamma}{ab}{k}) \right), \mathcal{L}\left(  \sqrt{n} (\cPest{\gamma}{ab}{k} - \cP{\gamma}{ab}{k} ) \right) \right) \label{triangle_middlestep1} \\
	 & \quad + d_{{\mathrm{W}}}\left( \mathcal{L}\left(\sqrt{n}\left(\cPest{\gamma}{ab}{k} - \cP{\gamma}{ab}{k}\right) \right),\mathcal{N}\left(0,\cP{\Sigma}{ab}{k}\right)\right).\label{triangle_middlestep2}
	\end{align}
	The results of Lemma~\ref{lem:diff_pre_est} and Theorem~\ref{general_theorem} are applied to the terms in~\eqref{triangle_middlestep1} and~\eqref{triangle_middlestep2}, respectively, to get the desired bound.\\
	(ii) Lemma~\ref{lem:diff_pre_est} asserts a bound of the order $\mathcal{O}(n^{-1/2})$, uniformly with respect to $k$.\\
	(iii) Note the following, interesting, special cases: (a) if $k = 0$, the bound reduces to
	\[n^{-1/2} \Big( \sum_{u=-\infty}^{\infty} |\cP{\gamma}{aa}{u}|\sum_{u=-\infty}^{\infty} |\cP{\gamma}{bb}{u}|\Big)^{1/2}.\]
	(b) if the data are componentwise uncorrelated (i.\,e., $\cP{\gamma}{aa}{u} = \cP{\gamma}{bb}{u} = 0$ for $u \neq 0$), then the bound reduces to
	\[\Big( n^{-1/2} + k\,n^{-3/2} \Big) \Big( \cP{\gamma}{aa}{0} \cP{\gamma}{bb}{0} \Big)^{1/2}.\]
\end{remark}

\section{Details on the computation of the bound for causal AR(1)}\label{app:tech_details:AR1_expl}
Here, we provide the technical details on how to obtain the bound in the AR(1) example which has been discussed in Section~\ref{sec:expl:AR1}. The bound in Theorem~\ref{general_theorem} is \[\frac{k}{\sqrt{n}} |\cO{\gamma}{ab}{k}| + \frac{\sqrt{2}}{\sqrt{\pi\cO{\Sigma}{ab}{k}}}\left|\cO{\Sigma}{ab}{k} - \cO{\tilde{\Sigma}}{ab}{k}\right| + \frac{2(n-k)}{\sqrt{n}} \tilde K + \frac{2 n^{-3/2}}{\left(\cO{\tilde{\Sigma}}{ab}{k}\right)^{3/2}} \sum_{t=1}^{n-k} \tilde Q_t.\]
Since we are in the setting of univariate time series we drop the indices $a$ and $b$.

In this section we treat the case of a univariate (i.\,e., $a = b$) AR(1) time series
\[\cO{X}{}{t} = \alpha \cO{X}{}{t-1} + \varepsilon_t = \sum_{j=0}^{\infty} \alpha^j \varepsilon_{t-j},\]
with parameter $\alpha \in (-1,1)$ and $\seq{\varepsilon_t}$ i.\,i.\,d. such that $\kappa_p := \cum_p(\varepsilon_t)$ exists for $p=8$ and $\kappa_p = 0$ when $p$ is odd.
We assume that $\kappa_1 = \E \varepsilon_t = 0$ and use the following moving average process as the $m$-dependent approximation
\[ \cO{Y}{}{t} = \sum_{j=0}^{m} \alpha^j \varepsilon_{t-j}.\]

We first treat the first and third term in the bound and the second and fourth term after that.
For the first and third term it suffices to compute
\begin{itemize}
	\item $\tilde D^{(q)} := \left[\E\left(\cO{X}{}{0} - \cO{Y}{}{0}\right)^q\right]^{1/q}$, for $q = 2, 4, 6, 8$, and
	\item $\E\left(\cO{X}{}{t+k} \cO{X}{}{t}\right)^q $, for $q = 1, 2, 3, 4$.
\end{itemize}
Note that $\|\cO{X}{}{t}\|_{2q}$ follows from $\E\left(\cO{X}{}{t+k} \cO{X}{}{t}\right)^q$ with $k=0$

For $\tilde D^{(q)}$, using Theorem~2.3.2 in~\cite{Brillinger1975} about cumulants of products, we obtain
\begin{equation*}
\begin{split}
& \E\left(\cO{X}{}{0} - \cO{Y}{}{0}\right)^q
= \E\left(\sum_{j=m+1}^{\infty} \alpha^j \cO{\varepsilon}{}{-j_{\ell}} \right)^q \\
& = \sum_{j_1=m+1}^{\infty} \cdots \sum_{j_q=m+1}^{\infty} \alpha^{j_1 + \cdots + j_q}
\sum_{\{\nu_1, \ldots, \nu_R\}} \prod_{r=1}^R \cum(\cO{\varepsilon}{}{-j_{\ell}} : \ell \in \nu_r ) \\
& = \sum_{\{\nu_1, \ldots, \nu_R\}} \prod_{r=1}^R \sum_{j=m+1}^{\infty} \alpha^{|\nu_r | j} \kappa_{|\nu_r|}
= \sum_{\{\nu_1, \ldots, \nu_R\}} \prod_{r=1}^R \kappa_{|\nu_r|} \alpha^{|\nu_r | (m+1)} \sum_{j=0}^{\infty} \alpha^{|\nu_r | j} \\
& = \alpha^{q(m+1)}\sum_{\{\nu_1, \ldots, \nu_R\}} \prod_{r=1}^R \frac{\kappa_{|\nu_r|}}{1 - \alpha^{|\nu_r|}}
\end{split}
\end{equation*}
where the sum is with respect to partitions $\{\nu_1, \ldots, \nu_R\}$ of $\{1, \ldots, q\}$. In particular, for $q = 2, 4, 8$, and since cumulants of odd orders vanish for the distributions considered, we have
\begin{equation*}
\begin{split}
\tilde D^{(2)} & = \alpha^{m+1} \Big( \frac{\kappa_2}{1-\alpha^2} \Big)^{1/2}, \\
\tilde D^{(4)} & = \alpha^{m+1} \Big( \frac{\kappa_4}{1-\alpha^4} + 3 \Big( \frac{\kappa_2}{1-\alpha^2} \Big)^2 \Big)^{1/4}, \\
\tilde D^{(8)} & = \alpha^{m+1} \Big( \frac{\kappa_8}{1-\alpha^8} + 28 \frac{\kappa_6}{1-\alpha^6} \frac{\kappa_2}{1-\alpha^2}
+ 35 \Big( \frac{\kappa_4}{1-\alpha^4} \Big)^2 \\
& \qquad + 210 \frac{\kappa_4}{1-\alpha^4} \Big( \frac{\kappa_2}{1-\alpha^2} \Big)^2
+ 105 \Big( \frac{\kappa_2}{1-\alpha^2} \Big)^4 \Big)^{1/8} \\
\end{split}
\end{equation*}

Similarly, we obtain that for $\alpha \neq 0$, we have
\begin{equation*}
\E\left(\cO{X}{}{t+k} \cO{X}{}{t}\right)^q = \alpha^{q |k|} \sum_{\{\nu_1, \ldots, \nu_R\}} \prod_{r=1}^R \frac{\kappa_{|\nu_r|}}{1-\alpha^{|\nu_r|}},
\end{equation*}
where the sum is with respect to partitions $\{\nu_1, \ldots, \nu_R\}$ of $\{1, \ldots, 2 q\}$.

In particular, for $q = 1,2,3,4$ we have,
\begin{equation}\label{XkX0_q}
\begin{split}
\E\left(\cO{X}{}{t+k} \cO{X}{}{t}\right)
& = \alpha^{|k|} \frac{\kappa_2}{1-\alpha^2} =: \cP{\gamma}{}{k}  \\
\E\left(\cO{X}{}{t+k} \cO{X}{}{t}\right)^2
& = \alpha^{2 |k|} \Big( \frac{\kappa_4}{1-\alpha^4} + 3 \Big( \frac{\kappa_2}{1-\alpha^2} \Big)^2 \Big) \\
\E\left(\cO{X}{}{t+k} \cO{X}{}{t}\right)^3 
& = \alpha^{3 |k|} \Big( \frac{\kappa_6}{1-\alpha^6} + 15 \frac{\kappa_4}{1-\alpha^4} \frac{\kappa_2}{1-\alpha^2}
+ 30 \Big( \frac{\kappa_2}{1-a^2} \Big)^3 \Big) \\
\E\left(\cO{X}{}{t+k} \cO{X}{}{t}\right)^4
& = \alpha^{4 |k|} \Big( \frac{\kappa_8}{1-\alpha^8} + 28 \frac{\kappa_6}{1-\alpha^6} \frac{\kappa_2}{1-\alpha^2}
+ 35 \Big( \frac{\kappa_4}{1-\alpha^4} \Big)^2 \\
& \qquad\qquad\qquad 	+ 210 \frac{\kappa_4}{1-\alpha^4} \Big( \frac{\kappa_2}{1-\alpha^2} \Big)^2
+ 105 \Big( \frac{\kappa_2}{1-\alpha^2} \Big)^4 \Big)
\end{split}
\end{equation}

Further, for $\alpha = 0$, we have
\begin{equation*}
\E\left(\cO{X}{}{t+k} \cO{X}{}{t}\right)^q =
\begin{cases}
\E \varepsilon_{t}^{2q} & k = 0 \\
\left(\E \varepsilon_{t}^{q} \right)^2 & k \neq 0. \\
\end{cases}
\end{equation*}

Thus, for the case of $\alpha = 0$ and $k = 0$ the expressions in~\eqref{XkX0_q} are correct when we apply the convention that $0^0 = 1$.
For $\alpha = 0$ and $k \neq 0$ we have 
\begin{equation*}
\begin{split}
\E\left(\cO{X}{}{t+k} \cO{X}{}{t}\right) \
& = 0 \\
\E\left(\cO{X}{}{t+k} \cO{X}{}{t}\right)^2
& = (\kappa_2)^2 \\
\E\left(\cO{X}{}{t+k} \cO{X}{}{t}\right)^3 
& = 0 \\
\E\left(\cO{X}{}{t+k} \cO{X}{}{t}\right)^4
& = \Big( \kappa_4 + 3 (\kappa_2)^2 \Big)^2
\end{split}
\end{equation*}

This covers the relevant pieces for the first and third term in the bound and we now turn our attention to the second and fourth term. To this end we first discuss how to compute joint cumulants of the $m$-dependent approximation in our example and then turn our attention to the second term of the bound.

\medskip

\noindent
\textbf{Joint cumulants of $\seq{\cO{Y}{}{t}}$.}
To compute the bound we will frequently require the joint cumulants of $\seq{\cO{Y}{}{t}}$. For $u_1, \ldots, u_p \in \IZ$ denote
$M := \min\{u_1, \ldots, u_p\}$,
$R := \max\{u_1, \ldots, u_p\} - M$, and $S := \sum_{i=1}^{p} (u_i - M)$.
Then, we have that
\begin{equation}\label{cumXm}
\begin{split}
& \cum( \cO{Y}{}{u_1}, \ldots, \cO{Y}{}{u_p}) \\
& = \sum_{j_1=0}^{m} \cdots \sum_{j_p=0}^{m} \alpha^{j_1 + \ldots + j_p} \cum(\cO{\varepsilon}{}{u_1 - M - j_1}, \ldots, \cO{\varepsilon}{}{u_p - M - j_p}) \\
&
= \begin{cases}
\sum\limits_{j_s=0}^{m-R} \alpha^{p j_s + S} \cum(\cO{\varepsilon}{}{-j_s}, \ldots, \cO{\varepsilon}{}{-j_s})
= \kappa_p \alpha^{S} \cfrac{1 - \alpha^{p (m - R + 1)}}{1 - \alpha^p} & R \leq m \\
0 & R > m.
\end{cases}
\end{split}
\end{equation}
For the second equality, note that $u_{\ell} - M \geq 0$, for all $\ell = 1, \ldots, p$, and $u_s - M = 0$ for at least one $s = 1, \ldots, p$. We now argue that for the cumulants $\cum(\varepsilon_{u_1 - M - j_1}, \ldots, \varepsilon_{u_p - M - j_p})$ with $(j_1, \ldots, j_p) \in \{1,\ldots,m\}^p$ and $j_s$ such that $u_s = M$, at most one of these cumulants is non zero. To this end, note that we have to have $u_s - M - j_s = -j_s = u_{\ell} - M -j_{\ell}$, for all $\ell \in 1,\ldots, p$, including $\ell = s$, which implies that $j_{\ell} = j_s + u_{\ell} - M \in \{1,\ldots,m\}$ for the cumulants that do not vanish. Hence, to ensure that $j_{\ell} \leq m$ for all $\ell$, we have to exclude $j_s > m-R$.
Further, we have that
$ u_{\ell_1} - j_{\ell_1} = u_{\ell_2} - j_{\ell_2}$, which implies that we have to require that
\[\max_{\ell_1, \ell_2 \in \{1, \ldots, p\}} |j_{\ell_1} - j_{\ell_2}| = \max_{\ell_1, \ell_2 \in \{1, \ldots, p\}}  |u_{\ell_1} - u_{\ell_2}| = R \leq m.\]

A very similar, but simpler, argument can be applied to obtain
\begin{equation*}
\cum( \cO{X}{}{u_1}, \ldots, \cO{X}{}{u_p})
= \kappa_p \cfrac{\alpha^{S}}{1 - \alpha^p}.
\end{equation*}

\medskip

\textbf{Computation of $\cO{\Sigma}{ab}{k}$ and $\cO{\tilde{\Sigma}}{ab}{k}$.}
To do so, we note that
\begin{equation*}
\begin{split}
\cP{\Sigma}{ab}{k} & = \sum_{u=-\infty}^{\infty} \cP{\gamma^2}{}{u}
+ \sum_{u=-\infty}^{\infty} \cP{\gamma}{}{u-k} \cP{\gamma}{}{u+k} + \frac{\kappa_4}{(\kappa_2)^2} \cP{\gamma^2}{}{k}\\
& = (\kappa_2)^2 \frac{1 + \alpha^2 + \alpha^{2|k|} \big(1 + \alpha^2 + 2 k (1 - \alpha^2)\big)}{(1-\alpha^2)^3} + \kappa_4 \frac{\alpha^{2|k|}}{(1-\alpha^2)^2}.
\end{split}
\end{equation*}

Further, we have
\begin{equation*}
\begin{split}
\cO{\tilde{\Sigma}}{ab}{k} & = n^{-1} \Var\Big( \sum_{t=1}^{n-k} \cO{Y}{}{t+k}\cO{Y}{}{t} \Big) \\
& = \frac{1}{n} \sum_{t_1=1}^{n-k} \sum_{t_2=1}^{n-k} \tilde C(t_1 - t_2)
= \frac{n-k}{n} \sum_{|u| \leq \min\{n - k,m\} } \Big( 1 - \frac{|u|}{n - k}\Big) \tilde C(u),
\end{split}
\end{equation*}
with $\tilde C(u)$ defined in~\eqref{def:cum}.
The cumulants of $\seq{\cO{Y}{}{t}}$ that appear in the definition of $\tilde C(u)$ can be computed as described in~\eqref{cumXm} of the previous section.

\medskip

It remains to discuss the fourth term of the bound.
Recall that in Section~\ref{sec:upper_bound} we discussed two methods to compute a bound for $\tilde Q_t$.
In this section we use the second method and discuss how to compute the exact value of the right hand side of~\eqref{bnd:Smn_new_bound}.

\medskip

\textbf{Computation of the bound for $\tilde Q_t$.}
We have
\begin{equation*}
\begin{split}
& \tilde Q_t \leq \Var \Big( \tilde Z(t) \sum_{j \in A_t} \tilde Z(j) \Big)^{1/2} \Var\Big(\sum_{j \in B_t} \tilde Z(j) \Big)^{1/2} + \frac{1}{2}  \Big[ \E \Big( \tilde Z(t) \sum_{j \in A_t} \tilde Z(j) \Big)^2 \Big]^{1/2} \Var \Big(\sum_{j \in A_t} \tilde Z(j) \Big)^{1/2} \\
& = \Big( M_{1,t} + \tilde C(0) \times M_{2a,t} + M_{3,t}^2 \Big)^{1/2} M_{2b,t}^{1/2}
+ \frac{1}{2} \Big( M_{1,t} + \tilde C(0) \times M_{2a,t} + 2 M_{3,t}^2 \Big)^{1/2} M_{2a,t}^{1/2},
\end{split}
\end{equation*}

where
\begin{equation*}
\begin{split}
M_{1,t} & := \sum_{j_1 \in A_t} \sum_{j_2 \in A_t} \sum_{\nu} \prod_{r=1}^R \cum ( \tilde Y_{ij} : (ij) \in \nu_r )
= \sum_{\ell_1 \in A_t - t} \sum_{\ell_2 \in A_t - t} \tilde D(\ell_1, \ell_2), \\
M_{2a,t} & := |A_t| \sum_{|u| \leq |A_t| - 1} \Big(1-\frac{|u|}{|A_t|}\Big) \tilde C(u),
\quad M_{2b,t} := |B_t| \sum_{|u| \leq |B_t| - 1} \Big(1-\frac{|u|}{|B_t|}\Big) \tilde C(u), \\
M_{3,t} & := \sum_{u \in A_t} \tilde C(u-t),
\hspace{2.35cm}  \tilde D(u_1, u_2) := \sum_{\nu} \prod_{r=1}^R \cum ( \tilde Y_{ij} : (ij) \in \nu_r )
\end{split}
\end{equation*}
where the sums in the definitions of $M_{1,t}$ and $\tilde D(u_1,u_2)$ are with respect to the indecomposable partitions~$\nu$ of Table~\ref{tab1}(a) and the terms $\tilde Y_{ij}$ are defined in Table~\ref{tab1}(b). We use the original indexing for the definition of $M_{1,t}$ and an indexing shifted by $t$ for the definition of $\tilde D(u_1,u_2)$. More precisely, of the following definitions we use the first in $M_{1,t}$ and the second in $\tilde D(u_1,u_2)$:
\begin{center}
	\begin{tabular}{ll}
		$\tilde Y_{1,1} := \cO{Y}{}{t+k} \stackrel{d}{=} \cO{Y}{}{k}$   & $\tilde Y_{1,2} := \cO{Y}{}{t} \stackrel{d}{=} \cO{Y}{}{0}$ \\
		$\tilde Y_{2,1} := \cO{Y}{}{t+k} \stackrel{d}{=} \cO{Y}{}{k}$   & $\tilde Y_{2,2} := \cO{Y}{}{t} \stackrel{d}{=} \cO{Y}{}{0}$ \\
		$\tilde Y_{3,1} := \cO{Y}{}{j_1+k} \stackrel{d}{=} \cO{Y}{}{u_1+k}$ & $\tilde Y_{3,2} := \cO{Y}{}{j_1} \stackrel{d}{=} \cO{Y}{}{u_1}$ \\
		$\tilde Y_{4,1} := \cO{Y}{}{j_2+k}\stackrel{d}{=} \cO{Y}{}{u_2+k}$ & $\tilde Y_{4,2} := \cO{Y}{}{j_2} \stackrel{d}{=} \cO{Y}{}{u_2}$
	\end{tabular}
\end{center}

Naively computing $\tilde Q_t$ is very costly. For computational efficiency, we therefore note that
\begin{itemize}
	\item $\tilde C(u) = \tilde C(-u)$,
	\item $\tilde D(u_1, u_2) = \tilde D(u_2, u_1)$,
	\item $M_{1,t}$, $M_{2a,t}$ and $M_{3,t}$ are constant on $t=m+k+1, \ldots, n-m-2k$, and
	\item $M_{2b,t}$ is constant on $t=2(m+k)+1, \ldots, n-2m-3k$.
\end{itemize}
Further, we have that the addends $\tilde C(u)$ and $\tilde D(\ell_1, \ell_2)$ do not depend on $t$. Therefore, for a given $m$ and $k$, we first compute $\tilde C(u)$, for $u=0,\ldots,2(m+k)$, and $\tilde D(u_1, u_2)$ for $u_1,u_2 = -(m+k), \ldots, m+k$ that satisfy $u_1 \leq u_2$. Then, computing the bound can be done rather efficiently.
\section{Details on simulating the Wasserstein distance}\label{sec:SimW1}
We obtain the Wasserstein distance, to compare our bound against, via simulation.
Denoting the cdf of $\sqrt{n} (\cPest{\gamma}{ab}{k} - \cP{\gamma}{ab}{k})$ by $F$ and the cdf of $N \sim \mathcal{N}\left(0,\cP{\Sigma}{ab}{k}\right)$ by $G$, we have
\begin{equation}\label{eqn:rep_W1}
W := d_{{\mathrm{W}}}\left(\mathcal{L}\left(\sqrt{n}\left(\cPest{\gamma}{ab}{k} - \cP{\gamma}{ab}{k}\right)\right),\mathcal{N}\left(0,\cP{\Sigma}{ab}{k}\right)\right) = \int_0^1 |F^{-1}(u) - G^{-1}(u)| {\rm d}u,
\end{equation}
where $F^{-1}(u) := \inf\{x \in \IR : F(x) \geq u\}$ is the quantile function of $F$; similarly, $G^{-1}$ is the quantile function for $G$.
We proceed by sampling i.\,i.\,d. $Z_1, \ldots, Z_R \sim F$ and then replace $F$ in~\eqref{eqn:rep_W1} by the empirical distribution $\mathbb{F}(x) := \frac{1}{R} \sum_{r=1}^R I\{Z_r \leq x\}.$ More precisely,
\begin{equation*}
\begin{split}
\eqref{eqn:rep_W1}
& \approx \int_0^1 |\mathbb{F}^{-1}(u) - G^{-1}(u)| {\rm d}u
= \sum_{r=1}^R \int_{(i-1)/R}^{i/R} |Z_{(r)} - G^{-1}(u)| {\rm d}u \\
& \approx \frac{1}{R} \sum_{r=1}^R \Big| Z_{(r)} - \sqrt{\cP{\Sigma}{ab}{k}} \Phi^{-1}\Big(\frac{2i-1}{2 R}\Big)\Big| =: \hat W_i.
\end{split}
\end{equation*}
If $F \neq G$ (and additional regularity conditions hold), then the error in the first approximation is of order $O_P(R^{-1/2})$; cf.\ e.\,g. \cite{MunkCzado1998}.
The error in the second approximation is of order $O(R^{-1})$.
Since we are considering $n$ to the order of thousands, we choose $R = 4 \times 10^6$.
Note that we have chosen $R$ as $2000^2$ and $n=2000$ is the maximum sample length.
Further, to reduce the variance of the (pseudo) random $\hat W_i$ we simulate independent copies of it ($i = 1, \ldots, B = 50$) and then report the average.
%In Figures~\ref{fig:W1_AR_N01}, \ref{fig:W1_AR_t14} and \ref{fig:W1_AR_t9}, the resulting simulated 1-Wasserstein distances are depicted.
\section{Additional technical proofs}
\label{app:proofs}
\subsection{Proof of Lemma~\ref{Lemma_locally_dependent}}\label{app:proofs1}
For $N \sim \mathcal{N}(0,\sigma^2)$ and $h \in \mathcal{H}_{\mathrm{W}}$, let first $f = f_h$ be the solution of the Stein equation 
\begin{equation}
\label{stein_equation}
\sigma^2f'(w) - wf(w) = h(w) - \E[h(N)]
\end{equation}
for the $\mathcal{N}(0,\sigma^2)$ distribution. Because $\forall i \in J$, $\xi_i$ and $W - \xi(A_i)$, with $\xi(A)$ as in~\eqref{tau_eta}, are independent and since $\E\left(\xi_i\right) = 0$, we have that
\begin{equation}
\nonumber \E\left(Wf(W)\right) = \sum_{i \in J}\E\left(\xi_i f(W)\right) = \sum_{i \in J}\E\left(\xi_i\left(f(W) - f(W - \xi(A_i))\right)\right).
\end{equation}
Adding now and subtracting the quantity $\sum_{i \in J}\E\left(\xi_i\xi(A_i)f'(W)\right)$ yields
\begin{equation}
\nonumber \E\left(Wf(W)\right) = \sum_{i \in J}\E\left(\xi_i\left(f(W) - f(W - \xi(A_i)) - \xi(A_i)f'(W)\right)\right) + \E\left(\sum_{i \in J}\xi_i\xi(A_i)f'(W)\right).
\end{equation}
For $\E\left(W^2\right) = {\rm var}(W) = \sigma^2$, we have now that
\[\sigma^2 = \E\left(W^2\right) = \E\left(\sum_{i \in J}\xi_i\sum_{j \in J}\xi_j\right) = \sum_{i \in J}\E\left(\xi_i\xi(A_i)\right)\]
because $\xi_i$ is independent of all the elements that are not in $A_i$. Using now the Stein equation as in \eqref{stein_equation}, we have that
\begin{align}
\label{midstep_stein_lemma}
\nonumber & \E\left(\sigma^2f'(W) - Wf(W)\right) \\
\nonumber &  = \E\left(\sum_{i \in J}\E\left(\xi_i\xi(A_i)\right)f'(W)\right) \\
\nonumber & \qquad - \sum_{i \in J}\E\left(\xi_i\left(f(W) - f(W - \xi(A_i)) - \xi(A_i)f'(W)\right)\right)
- \E\left(\sum_{i \in J}\xi_i\xi(A_i)f'(W)\right)\\
\nonumber & = - \sum_{i \in J}\E\left(\left(\xi_i\xi(A_i) - \E(\xi_i\xi(A_i))\right)f'(W)\right) - \sum_{i \in J}\E\left(\xi_i\left(f(W) - f(W - \xi(A_i)) - \xi(A_i)f'(W)\right)\right)\\
\nonumber & = - \sum_{i \in J}\E\left(\left(\xi_i\xi(A_i) - \E(\xi_i\xi(A_i))\right)\left(f'(W) - f'(W - \xi(B_i))\right)\right) \\
& \qquad  - \sum_{i \in J}\E\left(\xi_i\left(f(W) - f(W - \xi(A_i)) - \xi(A_i)f'(W)\right)\right)
\end{align}
with the last equality being because both $\xi_i$ and $\xi(A_i)$ are independent with $W - \xi(B_i)$, where the notation of $\xi(A)$ is given in \eqref{tau_eta} for any $A \subset J$. Therefore,
\begin{equation}
\nonumber \sum_{i \in J}\E\left(\left(\xi_i\xi(A_i) - \E(\xi_i\xi(A_i))\right)f'(W - \xi(B_i))\right) = \sum_{i \in J}\E\left(\xi_i\xi(A_i) - \E(\xi_i\xi(A_i))\right)\E\left(f'(W - \xi(B_i))\right) = 0.
\end{equation}
Using a first-order Taylor expansion,
\begin{equation}
\label{lemma_Stein_Taylor1}
f'(W) - f'(W - \xi(B_i)) = \xi(B_i)f''(W^*),
\end{equation}
for $W^*$ between $W$ and $W - \xi(B_i)$. Furthermore, using now a second order Taylor expansion of $f(W - \xi(A_i))$ about $W$, leads to
\begin{equation}
\label{lemma_Stein_Taylor2}
f(W) - f(W-\xi(A_i)) - \xi(A_i)f'(W) = -\frac{(\xi(A_i))^2}{2}f''(\tilde{W}),
\end{equation}
where $\tilde{W}$ is between $W$ and $W - \xi(A_i)$. We denote the sup-norm of a function by $\| \cdot \|$ and applying the results of \eqref{lemma_Stein_Taylor1} and \eqref{lemma_Stein_Taylor2} to \eqref{midstep_stein_lemma} yields
\begin{align}
\label{midstep_stein_lemma2}
\nonumber & \left|\E\left[h(W)\right] - \E\left[h(N)\right]\right| = \left|\E\left(\sigma^2f'(W) - Wf(W)\right)\right|\\
\nonumber & = \left|- \sum_{i \in J}\E\left(\left(\xi_i\xi(A_i) - \E(\xi_i\xi(A_i))\right)\xi(B_i)f''(W^*)\right) + \sum_{i \in J}\E\left(\xi_i\frac{(\xi(A_i))^2}{2}f''(\tilde{W})\right)\right|\\
\nonumber & \leq \sum_{i \in J}\E\left|\left(\xi_i\xi(A_i) - \E(\xi_i\xi(A_i))\right)\xi(B_i)f''(W^*)\right| + \sum_{i \in J}\E\left|\xi_i\frac{(\xi(A_i))^2}{2}f''(\tilde{W})\right|\\
\nonumber & \leq \|f''\|\sum_{i \in J}\E\left|\left(\xi_i\xi(A_i) - \E(\xi_i\xi(A_i))\right)\xi(B_i)\right| + \frac{1}{2}\|f''\|\sum_{i \in J}\E\left|\xi_i(\xi(A_i))^2\right|\\
& \leq \|f''\|\sum_{i \in J}\left\lbrace\E\left|\left(\xi_i\xi(A_i) - \E(\xi_i\xi(A_i))\right)\xi(B_i)\right| + \frac{1}{2}\E\left|\xi_i(\xi(A_i))^2\right|\right\rbrace.
\end{align} 
From Section~2.2 of \cite{Chen_et_al_book2011}, we know that the solution of the Stein equation in \eqref{stein_equation} satisfies that $\|f''\| \leq 2\frac{\|h'\|}{\sigma^3}$. Using this result in \eqref{midstep_stein_lemma2} yields
\begin{equation}
\nonumber \left|\E\left[h(W)\right] - \E\left[h(N)\right]\right| \leq 2\frac{\|h'\|}{\sigma^3}\sum_{i \in J}\left\lbrace\E\left|\left(\xi_i\xi(A_i) - \E(\xi_i\xi(A_i))\right)\xi(B_i)\right| + \frac{1}{2}\E\left|\xi_i(\xi(A_i))^2\right|\right\rbrace.
\end{equation}
Since we work for $h \in H_{\mathrm{W}}$ as in~\eqref{classes_functions}, we have that $\|h'\|\leq 1$, which leads to the result of the lemma.$\qquad\hfill\square$

\subsection{Proof of Lemmas~\ref{lem:approx_jnt_moments} and a generalized version of it}\label{app:proofs:jnt_moments}

Lemma~\ref{lem:approx_jnt_moments_nonstat}, stated and proved below, entails Lemma~\ref{lem:approx_jnt_moments} as a special case. The important difference between the two lemmas is that Lemma~\ref{lem:approx_jnt_moments_nonstat} can also be employed in the context of non-stationary data.

\begin{lemma}\label{lem:approx_jnt_moments_nonstat}
	For $p \in \IN$ and $\alpha \geq 1$,
	let $X_1, \ldots, X_p, Y_1, \ldots, Y_p$ be $\R$-valued random variables with $\| X_i \|_{\alpha p} < \infty$ and $\| Y_i \|_{\alpha p} < \infty$, $i = 1,\ldots, p$.
	Then, with $\Lambda_p$ as in Lemma~\ref{lem:approx_jnt_moments}, we have
	\begin{equation}\label{eqn:approx_jnt_q_moments_nonstat}
	\Big\| \prod_{i=1}^p X_i - \prod_{i=1}^p Y_i \Big\|_{\alpha}
	\leq \sum_{(\ell_1, \ldots, \ell_p) \in \Lambda_p} \prod_{i=1}^p \| X_i - Y_i \|_{\alpha p}^{\ell_i} \| X_i \|_{\alpha p}^{1-\ell_i},
	\end{equation}
\end{lemma}
\noindent
\textbf{Proof of Lemma~\ref{lem:approx_jnt_moments}.}
Lemma~\ref{lem:approx_jnt_moments} follows from the more general Lemma~\ref{lem:approx_jnt_moments_nonstat} if we apply it with $X_i := \cO{X}{a_i}{t_i}$ and $Y_i := \cO{Y}{a_i}{t_i}$ such that $\| X_i - Y_i \|_{\alpha p} \leq \tilde D_{a_i}^{(\alpha p)}$.

\clearpage
\noindent
\textbf{Proof of Lemma~\ref{lem:approx_jnt_moments_nonstat}.}
For $p = 1$ the assertion is obvious.
For $p \geq 2$ the assertion follows from the following chain of inequalities, where we denote $D_i := \| X_i - Y_i \|_{\alpha p}$
\begin{equation*}
\begin{split}
&  \Big\| \prod_{i=1}^p X_i - \prod_{i=1}^p Y_i \Big\|_{\alpha}
= \Big\| \sum_{j=1}^p \Big( \prod_{\substack{i=1 \\ i < j}}^p X_i \Big) \big( X_j - Y_j \big) \Big( \prod_{\substack{i=1 \\ i > j}}^p Y_i \Big)  \Big] \Big\|_{\alpha} \\
& \leq \sum_{j=1}^p \Big\| \Big( \prod_{\substack{i=1 \\ i < j}}^p X_i \Big) \big( X_j - Y_j \big) \Big( \prod_{\substack{i=1 \\ i > j}}^p Y_i \Big) \Big\|_{\alpha} \\
& \leq \sum_{j=1}^p \Big( \big\| X_j - Y_j \big\|_{\alpha p} \prod_{\substack{i=1 \\ i < j}}^p \| X_i \|_{\alpha p} \prod_{\substack{i=1 \\ i > j}}^p \big( \| X_i \|_{\alpha p} + \big\| X_i - Y_i \big\|_{\alpha p} \big) \Big) \\
\end{split}
\end{equation*}
\begin{equation*}
\begin{split}
& = D_{1} \big(\| X_2 \|_{\alpha p} + D_2 \big) \cdots \big(\| X_p \|_{\alpha p} + D_p \big) \\
& \quad + I\{p > 2\} \sum_{j=2}^{p-1} \| X_1 \|_{\alpha p} \cdots \| X_{j-1} \|_{\alpha p} D_j
\big(\| X_{j+1} \|_{\alpha p} + D_{j+1} \big) \cdots \big(\| X_p \|_{\alpha p} + D_p \big) \\
& \quad + D_p \| X_1 \|_{\alpha p} \cdots \| X_{p-1} \|_{\alpha p} \\
& = D_1 \sum_{(\ell_2, \ldots, \ell_p) \in \{0,1\}^{p-1}} \| X_2 \|_{\alpha p}^{1-\ell_2} D_2^{\ell_2} \cdots \| X_p \|_{\alpha p}^{1-\ell_p} D_p^{\ell_p} \\
& \quad + I\{p > 2\} \sum_{j=2}^{p-1} \| X_1 \|_{\alpha p} \cdots \| X_{j-1} \|_{\alpha p} D_j \\
& \qquad\qquad
\times \sum_{(\ell_{j+1}, \ldots, \ell_p) \in \{0,1\}^{p-j}} \| X_{j+1} \|_{\alpha p}^{1-\ell_{j+1}} D_{j+1}^{\ell_{j+1}}  \cdots \| X_p \|_p^{1-\ell_p} D_p^{\ell_p} \\
& \quad + \| X_1 \|_{\alpha p} \cdots \| X_{p-1} \|_{\alpha p} D_p,
\end{split}
\end{equation*}
where the first inequality follows due to the triangle inequality that we have because $\alpha \geq 1$, and the second inequality follows by (a generalization of) H\"{o}lder's inequality.\hfill$\square$
\subsection{Proof of Lemma~\ref{lem:bnd_var_mdep_est}}\label{app:proofs:bnd_var_mdep_est}
Note that, by the triangle inequality and the definition of $\cP{\tilde\Sigma}{ab}{k}$ in~\eqref{def:tildeSigma} of the main paper, we have
\begin{equation*}
\begin{split}
|\cP{\tilde\Sigma}{ab}{k} - \cP{\Sigma}{ab}{k}|
\leq &
\Big| \Var\left(n^{-1/2} \cPest{\gamma^{(m)}}{ab}{k} \right) - \Var\left(n^{-1/2} \cPest{\gamma}{a b}{k} \right) \Big|\\
& + \Big| n \Var\big(\cPest{\gamma}{a b}{k}\big) - \cP{\Sigma}{a b}{k} \Big|,
\end{split}
\end{equation*}
where
\begin{equation*}
\cPest{\gamma^{(m)}}{ab}{k} := \frac{1}{n} \sum_{t=1}^{n-k} \cO{Y}{a}{t+k} \cO{Y}{b}{t} \text{ for $k=0,\ldots,n-1$.}
\end{equation*}

We then have the following, more general, bound which we will subsequently use to derive the assertion of the lemma:
\begin{equation*}
\begin{split}
& \Big| \Cov\big(\sqrt{n} \cPest{\gamma}{a_1 b_1}{k_1}, \sqrt{n} \cPest{\gamma}{a_2 b_2}{k_2} \big) - \Cov\big(\sqrt{n} \cPest{\gamma^{(m)}}{a_1 b_1}{k_1}, \sqrt{n} \cPest{\gamma^{(m)}}{a_2 b_2}{k_2} \big) \Big| \\
& \leq \frac{1}{n} \sum_{t_1=1}^{n-k_1} \sum_{t_2=1}^{n-k_2} \Big| \E\big[ \cO{X}{a_1}{t_1+k_1} \cO{X}{b_1}{t_1} \cO{X}{a_2}{t_2+k_2} \cO{X}{b_2}{t_2} \big] \\
& \qquad\qquad\qquad\qquad - \E\big[ \cO{Y}{a_1}{t_1+k_1} \cO{Y}{b_1}{t_1} \cO{Y}{a_2}{t_2+k_2} \cO{Y}{b_2}{t_2} \big]\Big| \\
& \quad + \frac{1}{n} \sum_{t_1=1}^{n-k_1} \sum_{t_2=1}^{n-k_2} \Big| \Big( \E\big[ \cO{X}{a_1}{t_1+k_1} \cO{X}{b_1}{t_1} \big] - \E\big[ \cO{Y}{a_1}{t_1+k_1} \cO{Y}{b_1}{t_1} \big] \Big) \E\big[ \cO{X}{a_2}{t_2+k_2} \cO{X}{b_2}{t_2} \big] \\
& \qquad\qquad + \Big( \E\big[\cO{X}{a_2}{t_2+k_2} \cO{X}{b_2}{t_2} \big] - \E\big[\cO{Y}{a_2}{t_2+k_2} \cO{Y}{b_2}{t_2} \big] \Big) \E\big[ \cO{Y}{a_1}{t_1+k_1} \cO{Y}{b_1}{t_1} \big] \Big| \\
& \leq \frac{(n-k_1)(n-k_2)}{n} \Big( K_{4,m}(\{a_1, b_1, a_2, b_2\}) + K_{2,m}(\{a_1, b_1\}) |\cP{\gamma}{a_2 b_2}{k_2}| \\
& \qquad\qquad\qquad\qquad + \big( |\cP{\gamma}{a_1 b_1}{k_1}| + K_{2,m}(\{a_1, b_1\}) \big) K_{2,m}(\{a_2, b_2\} \Big) \\
& \leq \frac{(n-k_1)(n-k_2)}{n} \Big(\max_{i=a_1,b_1,a_2,b_2} \tilde D_{i}^{(4)}\Big) \\
& \quad \times \Bigg( \sum_{j \in \{a_1, b_1, a_2, b_2\}} \ \prod_{i \in \{a_1, b_1, a_2, b_2\} \setminus \{j\}} (\|\cO{X}{i}{0}\|_4 + \tilde D_{i}^{(4)}) \\
& \qquad + ( \|\cO{X}{a_1}{0}\|_2 + \|\cO{X}{b_1}{t}\|_2 + \min_{i=a_1,b_1} \tilde D_{i}^{(2)} ) |\cP{\gamma}{a_2 b_2}{k_2}| \\
& \qquad + ( \|\cO{X}{a_2}{0}\|_2 + \|\cO{X}{b_2}{t}\|_2 + \min_{i=a_2,b_2} \tilde D_{i}^{(2)} ) |\cP{\gamma}{a_1 b_1}{k_1}| \\
& \qquad + \min\Big\{\max_{i=a_1,b_1} \tilde D_{i}^{(2)}, \max_{i=a_2,b_2} \tilde D_{i}^{(2)}\Big\} \\
& \qquad\quad \times \Big( \|\cO{X}{a_1}{0}\|_2 + \|\cO{X}{b_1}{t}\|_2 + \min_{i=a_1,b_1} \tilde D_{i}^{(2)} \Big)\Big( \|\cO{X}{a_2}{0}\|_2 + \|\cO{X}{b_2}{t}\|_2 + \min_{i=a_2,b_2} \tilde D_{i}^{(2)} \Big) \Bigg) \\
\end{split}
\end{equation*}
where for the third inequality we have, by Lemma~\ref{lem:approx_jnt_moments} with $\alpha = 1$ and $p = 2$, that
\begin{equation*}
\begin{split}
\tilde K_2^{(1)} (\{a,b\}) & = \tilde D_a^{(2)} \|\cO{X}{b}{t}\|_2 + \tilde D_a^{(2)} \tilde D_b^{(2)} + \|\cO{X}{a}{0}\|_2 \tilde D_b^{(2)} \\
& \leq \Big(\max_{i=a,b}\tilde D_i^{(2)}\Big) ( \|\cO{X}{a}{0}\|_2 + \|\cO{X}{b}{t}\|_2 + \min_{i=a,b} \tilde D_i^{(2)} )
\end{split}
\end{equation*}
and, by Lemma~\ref{lem:approx_jnt_moments} with $\alpha = 1$ and $p = 4$, we have that
\begin{equation*}
\begin{split}
& \tilde K_4^{(1)}(\{a_1, b_1, a_2, b_2\}) \\
& = \tilde D_{a_1}^{(4)} (\|\cO{X}{b_1}{0}\|_4 + \tilde D_{b_1}^{(4)})(\|\cO{X}{a_2}{0}\|_4 + \tilde D_{a_2}^{(4)})(\|\cO{X}{b_2}{0}\|_4 + \tilde D_{b_2}^{(4)}) \\
& \quad + \tilde D_{b_1}^{(4)} \|\cO{X}{a_1}{0}\|_4 (\|\cO{X}{a_2}{0}\|_4 + \tilde D_{a_2}^{(4)})(\|\cO{X}{b_2}{0}\|_4 + \tilde D_{b_2}^{(4)}) \\
& \quad + \tilde D_{a_2}^{(4)} \|\cO{X}{a_1}{0}\|_4 \|\cO{X}{b_1}{0}\|_4 (\|\cO{X}{b_2}{0}\|_4 + \tilde D_{b_2}^{(4)}) \\
& \quad + \tilde D_{b_2}^{(4)} \|\cO{X}{a_1}{0}\|_4 \|\cO{X}{b_1}{0}\|_4 \|\cO{X}{a_2}{0}\|_4 \\
& \leq \Big(\max_{i=a_1,b_1,a_2,b_2} \tilde D_{i}^{(4)}\Big) \sum_{j \in \{a_1, b_1, a_2, b_2\}} \ \prod_{i \in \{a_1, b_1, a_2, b_2\} \setminus \{j\}} (\|\cO{X}{i}{0}\|_4 + \tilde D_{i}^{(4)}).
\end{split}
\end{equation*}
For the third inequality we had further employed the fact that
\[\max_{i=j_1,j_2} \tilde D_{i}^{(2)} \leq \max_{i=a_1,b_1,a_2,b_2} \tilde D_{i}^{(4)}, \quad \text{for any $j_1, j_2 \in \{a_1, b_1, a_2, b_2\}$}.\]
The assertion of the lemma follows, after some simplications applied to the special case where we have $k_1 = k_2$, $a_1 = a_2$ and $b_1 = b_2$.\hfill$\square$

\subsection{Proof of Lemma~\ref{lem:bnd1_Smn}}\label{app:proofs:bnd1_Smn}

%The quantity of interest is $|A - A_m|$, where
%\[A := \left\|\left( Z_t \sum_{j \in A_t} Z_j - \E\Bigg( Z_t \sum_{j \in A_t} Z_j \Bigg)\right) \sum_{j \in B_t} Z_j \right\|_1 + \frac{1}{2} \left\| Z_t \Bigg(\sum_{j \in A_t} Z_j \Bigg)^2\right\|_1,\]
%and 
%\[A_m := \left\|\left( \tilde Z_t \sum_{j \in A_t} \tilde Z_j - \E\Bigg( \tilde Z_t \sum_{j \in A_t} \tilde Z_j \Bigg)\right) \sum_{j \in B_t} \tilde Z_j \right\|_1 + \frac{1}{2} \left\| \tilde Z_t \Bigg(\sum_{j \in A_t} \tilde Z_j \Bigg)^2\right\|_1.\]

We obtain the assertion from the following chain of inequalities
\begin{align}
\nonumber & |Q_t - \tilde Q_t| \\
\nonumber & \leq \Big| \|( Z_t Z(A_t) - \E( Z_t Z(A_t) )) Z(B_t) \|_1
- \|( \tilde Z_t \tilde Z(A_t) - \E( \tilde Z_t \tilde Z(A_t) )) \tilde Z(B_t) \|_1 \Big| \\
\nonumber & \qquad + \Big| \frac{1}{2} \| Z_t Z(A_t)^2 \|_1 - \frac{1}{2} \| \tilde Z_t \tilde Z(A_t)^2 \|_1 \Big| \\
\nonumber & \leq \E\Big| ( Z_t Z(A_t) - \E( Z_t Z(A_t) ) ) Z(B_t) - ( \tilde Z_t \tilde Z(A_t) - \E( \tilde Z_t \tilde Z(A_t))) \tilde Z(B_t) \Big| \\
\nonumber & \qquad + \frac{1}{2} \E\left| Z_t Z(A_t)^2 - \tilde Z_t \tilde Z(A_t)^2 \right| \\
\nonumber & \leq \E\left| Z_t Z(A_t) Z(B_t) - \tilde Z_t \tilde Z(A_t) \tilde Z(B_t) \right|
+ \E\left| \E( Z_t Z(A_t) ) Z(B_t) - \E( \tilde Z_t \tilde Z(A_t) ) \tilde Z(B_t) \right| \\
\nonumber & \qquad + \frac{1}{2} \E\left| Z_t Z(A_t)^2 - \tilde Z_t \tilde Z(A_t)^2 \right| \\
\label{bnd1_Smn:A} & \leq \sum_{j_1 \in A_t} \sum_{j_2 \in B_t} \E\left| Z_t  Z_{j_1}  Z_{j_2} - \tilde Z_t \tilde Z_{j_1} \tilde Z_{j_2} \right|
+ \frac{1}{2} \sum_{j_1 \in A_t} \sum_{j_2 \in A_t} \E\left| Z_t  Z_{j_1}  Z_{j_2} - \tilde Z_t \tilde Z_{j_1} \tilde Z_{j_2} \right| \\
\label{bnd1_Smn:B} & \qquad + \E\left| \E\Bigg( Z_t \sum_{j \in A_t} Z_j \Bigg)  \sum_{j \in B_t} (Z_j - \tilde Z_j) \right| \\
\label{bnd1_Smn:C} & \qquad + \E \left| \Bigg[ \E\Bigg( Z_t \sum_{j \in A_t} Z_j \Bigg) - \E\Bigg( \tilde Z_t \sum_{j \in A_t} \tilde Z_j \Bigg) \Bigg] \sum_{j \in B_t} (Z_j + \tilde Z_j - Z_j) \right|,
\end{align}
where we have used the triangle inequality and reverse triangle inequality.
The asserted bound for $|Q_t - \tilde Q_t|$ will be the sum of the bounds we now derive for~\eqref{bnd1_Smn:A}, \eqref{bnd1_Smn:B} and~\eqref{bnd1_Smn:C}. First, we derive two preliminary bounds to control $\| Z_0 \|_{\alpha}$ and $\| Z_0 - \tilde Z_0 \|_{\alpha}$, $\alpha = 1, 2, 3$. Employing the triangle inequality and H\"{o}lder's inequality we have
\begin{equation}\label{bnd_Z_alpha}
\| Z_0 \|_{\alpha}
\leq \| \cO{X}{a}{t+k} \cO{X}{b}{t} \|_{\alpha} + |\E[\cO{X}{a}{t+k} \cO{X}{b}{t}]| \leq 2 \| \cO{X}{a}{0} \|_{2 \alpha} \| \cO{X}{b}{0} \|_{2 \alpha}
\end{equation}
and, using the triangle inequality and Lemma~\ref{lem:approx_jnt_moments}, we have
\begin{equation}\label{bnd_diff_Z_alpha}
\begin{split}
\| Z_0 - \tilde Z_0 \|_{\alpha}
%& = \Big( \E \Big| \cO{X}{a}{t+k} \cO{X}{b}{t} - \cO{Y}{a}{t+k} \cO{Y}{b}{t} \\
%& \qquad + \E[\cO{Y}{a}{t+k} \cO{Y}{b}{t}] - \E[\cO{X}{a}{t+k} \cO{X}{b}{t}] \Big|^{\alpha} \Big)^{1/\alpha} \\
& \leq 2 \| \cO{X}{a}{t+k} \cO{X}{b}{t} - \cO{Y}{a}{t+k} \cO{Y}{b}{t} \|_{\alpha} \leq 2 \tilde K_{2}^{(\alpha)}(\{a,b\})
\end{split}
\end{equation}

Next, to bound~\eqref{bnd1_Smn:A}, we use Lemma~\ref{lem:approx_jnt_moments_nonstat} with $\alpha = 1$, $p=3$, $X_1, X_2, X_3$ equal to $Z_t,  Z_{j_1},  Z_{j_2}$ and $Y_1, Y_2, Y_3$ equal to $\tilde Z_t, \tilde Z_{j_1}, \tilde Z_{j_2}$, respectively.

Then, $\|X_i\|_{\alpha} = \|Z_0\|_{\alpha}$ and $\|X_i - Y_i\|_{\alpha} = \|Z_0 - \tilde Z_0\|_{\alpha}$, such that
\begin{equation}\label{bnd_3Z}
\begin{split}
& \E\left| Z_t  Z_{j_1}  Z_{j_2} - \tilde Z_t \tilde Z_{j_1} \tilde Z_{j_2} \right| \\
& \leq \| Z_0 - \tilde Z_0 \|_3 \Big( 3 \| Z_0 \|_3^2 + 3 \| Z_0 \|_3 \| Z_0 - \tilde Z_0 \|_3 + \| Z_0 - \tilde Z_0 \|_3^2 \Big) \\
& \leq 2 \tilde K_{2}^{(3)} \Big( 3 \big( 2 \| \cO{X}{a}{0} \|_{6} \| \cO{X}{b}{0} \|_{6} + \tilde K_{2}^{(3)} \big)^2 + (\tilde K_{2}^{(3)})^2 \Big) \\
\end{split}
\end{equation}
Clearly, this implies that
\begin{equation}\label{bnd1_Smn:A2}
\eqref{bnd1_Smn:A} \leq 2 \tilde K_{2}^{(3)} \Big( 3 \big( 2 \| \cO{X}{a}{0} \|_{6} \| \cO{X}{b}{0} \|_{6} + \tilde K_{2}^{(3)} \big)^2 + (\tilde K_{2}^{(3)})^2 \Big) \Big( |A_t| |B_t| + \frac{1}{2} |A_t|^2 \Big)
\end{equation}
For \eqref{bnd1_Smn:B} we have
\begin{equation}\label{bnd1_Smn:B2}
\eqref{bnd1_Smn:B} = \Big| \sum_{j \in A_t} \E( Z_t  Z_j ) \Big| \times \E\left|  \sum_{j \in B_t} (Z_j - \tilde Z_j) \right| \leq \Big| \sum_{j \in A_t} \E( Z_t  Z_j ) \Big| \times  |B_t| \times 2 \tilde K_{2}^{(1)},
\end{equation}
where we have used~\eqref{bnd_diff_Z_alpha} with $\alpha=1$.
Finally, to bound~\eqref{bnd1_Smn:C}, we use Lemma~\ref{lem:approx_jnt_moments_nonstat} similar to how we used it in~\eqref{bnd_3Z}, but with $p = 2$, and obtain
\begin{equation*}
\begin{split}
\E\left| Z_t  Z_j - \tilde Z_t \tilde Z_j \right|
& \leq \| Z_0 - \tilde Z_0 \|_2 \Big( 2 \| Z_0 \|_2 + \| Z_0 - \tilde Z_0 \|_2 \Big) \\
& \leq 2 \tilde K_{2}^{(2)} \Big( 4 \| \cO{X}{a}{0} \|_{4} \| \cO{X}{b}{0} \|_{4} + 2 \tilde K_{2}^{(2)} \Big),
\end{split}
\end{equation*}
where the second inequality followed from the preliminary bounds~\eqref{bnd_Z_alpha} and~\eqref{bnd_diff_Z_alpha}, with $\alpha = 2$.
Thus, we have
\begin{equation}\label{bnd1_Smn:C2}
\begin{split}
\eqref{bnd1_Smn:C}
& =  \Bigg| \E\Bigg( Z_t \sum_{j \in A_t} Z_j \Bigg) - \E\Bigg( \tilde Z_t \sum_{j \in A_t} \tilde Z_j \Bigg) \Bigg| \times \E \Bigg[ \Big| \sum_{j \in B_t} (Z_j + \tilde Z_j - Z_j) \Big| \Bigg] \\
& \leq \sum_{j \in A_t} \E \Big| Z_t Z_j - \tilde Z_t \tilde Z_j \Big|
\times \sum_{j \in B_t}  \Big( \|Z_0\|_1 + \| \tilde Z_0 - Z_0 \|_1 \Big) \\
& \leq 2 \tilde K_{2}^{(2)} \Big( 4 \| \cO{X}{a}{0} \|_{4} \| \cO{X}{b}{0} \|_{4} + 2 \tilde K_{2}^{(2)} \Big) | A_t | \\
&  \qquad \times 2 \Big( \| \cO{X}{a}{0} \|_{2} \| \cO{X}{b}{0} \|_{2} + \tilde K_{2}^{(1)} \Big) |B_t|.
\end{split}
\end{equation}
Summing the three bounds obtained in~\eqref{bnd1_Smn:A2}, \eqref{bnd1_Smn:B2} and \eqref{bnd1_Smn:C2} yields the assertion.
\hfill$\square$

\subsection{Proof of Proposition~\ref{prop:rate}}\label{prf:prop:rate}The bound given in~\eqref{final_upper_bound} consists of four terms that we now discuss one by one, for Regimes~1 and~2.
	For Regime~1, we assume that $\vO{Y^{(m)}}{t}$ is chosen as described in Remark~\ref{rem:asym_reg}.
	In this proof, we will refer to the first, second, third and fourth term of the bound, respectively, meaning the respective addends in~\eqref{final_upper_bound}.
	For the first term of the bound, it is obvious that $kn^{-1/2}|\cO{\gamma}{ab}{k}| = \mathcal{O}(n^{-1/2})$, in both regimes.
	With respect to the second term of the bound we now show that it is of the order $\mathcal{O}(n^{-1/2})$, in both regimes, if $m$ is chosen appropriately.
	The expression $|\cP{\Sigma}{ab}{k} - \cP{\tilde\Sigma}{ab}{k}|$ can be controlled by employing Lemma~\ref{lem:bnd_var_mdep_est}.
	Therefore,
	\begin{align}
	\nonumber \frac{\sqrt{2}}{\sqrt{\pi\cO{\Sigma}{ab}{k}}}\left|\cO{\Sigma}{ab}{k} - \cO{\tilde{\Sigma}}{ab}{k}\right| & \leq 
	\frac{\sqrt{2}}{\sqrt{\pi\cO{\Sigma}{ab}{k}}}\Big| n \Var\big(\cPest{\gamma}{a b}{k}\big) - \cP{\Sigma}{a b}{k} \Big| \\
	\label{third_term_order_middle_step_b}
	& \quad + 2 \frac{\sqrt{2}}{\sqrt{\pi\cO{\Sigma}{ab}{k}}} \frac{(n-k)^2}{n} \Big(\max_{i=a,b} \tilde D_{i,m}^{(4)} \Big) \tilde F_m,
	\end{align}
	where $\tilde F_m$ is given in~\eqref{eq:Fm}. In Regime~1, $\tilde D_{j,m}^{(q)} = 0$ for $n \geq M$ and thus~\eqref{third_term_order_middle_step_b} vanishes for $n$ large enough. In Regime~2, since fourth moments are finite, from~\eqref{A:mdep_geom} and Jensen's inequality, we have that $$\min_{i=a,b} \tilde D_{i,m}^{(2)} \leq \max_{i=a,b} \tilde D_{i,m}^{(2)} \leq \max_{i=a,b} \tilde D_{i,m}^{(4)} = o(1),$$ as $m \rightarrow \infty$, which implies $\tilde F_m = \mathcal{O}(1)$.
	Now, since $m = C\log n$, with $C \geq -3/(2\log(\rho))$, we have that~\eqref{third_term_order_middle_step_b} is of the order $\mathcal{O}(n \rho^m) = \mathcal{O}(n^{-1/2})$, in Regime~2 and we had already seen that this is the case in Regime~1.
	It remains to assess the order of $|n \Var\big(\cPest{\gamma}{a b}{k}\big) - \cP{\Sigma}{a b}{k}|$, the difference of the finite sample variance and the asymptotic variance of the empirical cross-covariance.
	This quantity is of the order $\mathcal{O}(n^{-1})$ under condition~\eqref{eqn:sum_cum}, which we have argued holds in both regimes.
	This can, for example, be concluded from the arguments provided in Section~7.6 in \cite{Brillinger1975}.
	In conclusion, the second term of the bound is of the order $\mathcal{O}(n^{-1/2})$, in Regimes~1 and~2.
	
	With respect to the third term of the bound, $2(n-k) n^{-1/2} \tilde K_{m}$,
	note that in Regime 1 we have that $\tilde K_{m}$ vanishes eventually, as $\tilde D_{j,m}^{(q)} = 0$ for $n \geq M$ and in Regime~2 we have $\tilde K_{m} = \mathcal{O}(\rho^m)$, by~\eqref{A:mdep_geom}.
	Therefore, the third term in the bound vanishes for $n$ large enough in Regime~1 and is of order $\mathcal{O}(\sqrt{n}\rho^m) = \mathcal{O}(n^{-1/2})$ in Regime~2, if we choose $m = C\log n$, and it suffices that we have $C \geq -1/\log(\rho)$. In conclusion, the third term in the bound is of the order $\mathcal{O}(n^{-1/2})$, in both regimes.
	
	We now proceed to discuss the rate of the fourth term of the bound,  $2n^{-3/2} (\cO{\tilde{\Sigma}}{ab}{k})^{-3/2} \sum_{t=1}^{n-k} \tilde Q_t$. For this term, the regime makes a difference in the outcome of our analysis. From the discussion of the second term of the bound we have, $\cO{\tilde{\Sigma}}{ab}{k} \rightarrow \cO{\Sigma}{ab}{k} > 0$, as $n \rightarrow \infty$.
	By~\eqref{eqn:outl_bnd_tildeSmn_1}, Lemma~\ref{lem:bnd1_Smn} and condition~\eqref{A:mdep_geom} (which holds in both regimes according to Remark~\ref{rem:asym_reg}), it suffices to show that $\sup_{t=1,\ldots,n-k} Q_t$ is of the order $\mathcal{O}(1)$ in Regime~1 and of the order $\mathcal{O}(\log n)$ in Regime 2, respectively.
	
	Regime~1: Bounding $Q_t$ as in~\eqref{bnd:Qt_naive}, but with $\seq{\vO{Y^{(m)}}{t}}$ replaced by $\seq{\vO{X}{t}}$, we have $Q_t \leq 3 (4m + 4k + 1)^2 \|\vO{X}{0}\|_6^6 = \mathcal{O}(m^2) = \mathcal{O}(1)$, since $m = \min\{n,M\} \leq M$.
	
	Regime~2: Here we use a version of~\eqref{bnd:Smn_new_bound}, with $\seq{\vO{Y^{(m)}}{t}}$ replaced by $\seq{\vO{X}{t}}$:
	\begin{equation}\label{bnd:Qt_new_bound}
	\begin{split}
	Q_t & \leq \Var \Big( Z_t \sum_{j \in A_t} Z_j \Big)^{1/2} \Var\Big(\sum_{j \in B_t} Z_j \Big)^{1/2}\\
	&\quad + \frac{1}{2}  \Big[ \E \Big( Z_t \sum_{j \in A_t} Z_j \Big)^2 \Big]^{1/2} \Var \Big(\sum_{j \in A_t} Z_j \Big)^{1/2}
	\end{split}
	\end{equation}
	It suffices to show that the three quantities defined in~\eqref{bnd_parts}, with $\vO{Y^{(m)}}{t}$ replaced by $\vO{X}{t}$, are $\mathcal{O}(m)$. We now show that this is implied by condition~\eqref{eqn:sum_cum}, the summability of cumulants up to order~8. Following arguments from Section \ref{app:tech_details:Sec23}, we have
	\begin{equation}\label{bnd:Qt_new_bound2}
	\Var\Big(\sum_{j \in A_t} Z_j \Big)
	\leq |A_t| \sum_{u = -\infty}^{\infty} |C(u)|,
	\quad \text{and} \quad
	\E \Big( Z_t \sum_{j \in A_t} Z_j \Big)
	\leq \sum_{u = -\infty}^{\infty}  |C(u)|,
	\end{equation}
	where
	\begin{equation*}
	C(u) = \cum(\cO{X}{a}{k}, \cO{X}{b}{0}, \cO{X}{a}{u+k}, \cO{X}{b}{u}) + \cO{\gamma}{aa}{u} \cO{\gamma}{bb}{u} + \cO{\gamma}{ab}{k-u} \cO{\gamma}{ab}{k+u},
	\end{equation*}
	which is summable according to condition~\eqref{eqn:sum_cum}.
	\begin{table}[t]
		\caption{Table of indices and variables to be partitioned for the addends of the sum in~\eqref{cumEight_b}\label{tab2}}
		\centering
		\begin{subtable}[h]{0.35\textwidth}
			\centering
			\begin{tabular}{cc}
				\toprule
				(1,$a$) & (1,$b$) \\
				(2,$a$) & (2,$b$) \\
				(3,$a$) & (3,$b$) \\
				(4,$a$) & (4,$b$) \\
				\bottomrule
			\end{tabular}
			\caption{Table of indices to be partitioned}
			\label{tab:2a}
		\end{subtable}
		\hspace{1cm}
		\begin{subtable}[h]{0.5\textwidth}
			\centering
			\begin{tabular}{ll}
				\toprule
				$Y_{1,a} := \cO{X}{a}{t+k}$   & $Y_{1,b} := \cO{X}{b}{t}$ \\
				$Y_{2,a} := \cO{X}{a}{t+k}$   & $Y_{2,b} := \cO{X}{b}{t}$ \\
				$Y_{3,a} := \cO{X}{a}{j_1+k}$ & $Y_{3,b} := \cO{X}{b}{j_1}$ \\
				$Y_{4,a} := \cO{X}{a}{j_2+k}$ & $Y_{4,b} := \cO{X}{b}{j_2}$ \\
				\bottomrule
			\end{tabular}
			\caption{Variables of which the cumulants are considered}
			\label{tab2b}
		\end{subtable}
	\end{table}
	Finally, with notation from Table~\ref{tab2}(b), we have
	\begin{equation}\label{cumEight_b}
	\begin{split}
	& \sum_{j_1 \in A_t} \sum_{j_2 \in A_t} \cum ( Z_t, Z_t, Z_{j_1}, Z_{j_2} )
	\leq \sum_{\nu} S(\nu),\\
	& S(\nu) := \sum_{j_1 = -\infty}^{\infty} \sum_{j_2 = -\infty}^{\infty} \prod_{r=1}^R |\cum ( Y_{ij} : (ij) \in \nu_r )|,
	\end{split}
	\end{equation}
	where the sum extends over all indecomposable partitions $\nu := \{\nu_1, \ldots, \nu_R\}$ of the indices in Table~\ref{tab2}(a).
	To complete our argument it suffices to show that for every indecomposable partition $\nu$ of Table~\ref{tab2}(a), with $|\nu_r| \geq 2$,
	\begin{equation}\label{exist_C_nu}
	\exists C_{\nu}:|S(\nu)| \leq C_{\nu},
	\end{equation}
	with $C_{\nu}$ being independent of $t$. Then, since there are only finitely many $\nu$, we have shown that the left-hand side in~\eqref{cumEight_b} is $\mathcal{O}(1)$ uniformly in~$t$. The rigorous proof of~\eqref{exist_C_nu} is straightforward but tedious. It is given below after the remaining steps to the proof of Proposition~\ref{prop:rate}.
	
	In conclusion, by~\eqref{bnd:Qt_new_bound} in combination with the discussion of how to compute~\eqref{bnd:Smn_new_bound}, together with \eqref{bnd:Qt_new_bound2}, \eqref{cumEight_b}, \eqref{exist_C_nu}, $|A_t| \leq 4(m+k)+1$ and the choice of $m = C \log n$ for the first, second and third term of the bound, we have shown that in Regime~2 we have $\sup_{t=1,\ldots,n-k} Q_t = \mathcal{O}(\log n)$, which implies that the fourth term of the bound under Regime~2 is of the order $\mathcal{O}(n^{-1/2} \log n)$.
	
We now complete the proof by showing that for every $\nu$ as in~\eqref{exist_C_nu}, one of the two following assertions holds true:
\begin{itemize}
	\item[(a)] there exists one set $\nu_{\ell}$ in $\nu$ such that
	\[\sum_{j_1 = -\infty}^{\infty} \sum_{j_2 = -\infty}^{\infty} |\cum ( Y_{ij} : (ij) \in \nu_{\ell} )| \leq c_{\nu},\]
	
	\item[(b)] or there exist two distinct sets $\nu_{\ell_1}$ and $\nu_{\ell_2}$ in $\nu$ such that
	\[\sum_{j_1 = -\infty}^{\infty} \sum_{j_2 = -\infty}^{\infty} |\cum ( Y_{ij} : (ij) \in \nu_{\ell_1} )| |\cum ( Y_{ij} : (ij) \in \nu_{\ell_2} )| \leq c_{\nu},\]
\end{itemize}
with the constant $c_{\nu}$ not depending on $t$. The desired $S(\nu) \leq C_{\nu}$ then follows from applying
\begin{equation}\label{eqn:bnd_cumY}
|\cum(Y_{ij} : (ij) \in \nu_r)| \leq |\nu_r|^{|\nu_r|} (|\nu_r|-1)! \max\{1, \| \vO{X}{0} \|_{|\nu_r|} \}^{|\nu_r|}  =: B(\nu_r),
\end{equation}
to the remaining factors in the definition of $S(\nu)$.
Note that $B(\nu_r)$ does not depend on $t$.
The inequality in~\eqref{eqn:bnd_cumY} follows from definition~\eqref{def:cum}, the triangle inequality, the fact that the number of partitions of $\nu_r$ can be bounded by $|\nu_r|^{|\nu_r|}$, the fact that the maximum number of sets in a partition of $\nu_r$ is $|\nu_r|$, generalised H\"{o}lder inequality, stationarity and Jensen's inequality.
In the following, we refer to (3,1) and (3,2) from Table~\ref{tab2}(a) as the $j_1$-indices and we refer to (4,1) and (4,2) as the $j_2$-indices. By $j$-indices we will refer to (3,1), (3,2), (4,1) and (4,2), and by non-$j$-indices we will refer to the remaining (1,1), (1,2), (2,1) and (2,2). 
Now, assume that $\nu := \{\nu_1, \ldots, \nu_R\}$ is an indecomposable partition of Table~\ref{tab2}(a), where each $\nu_r$ contains at least two elements. To show that either (a) or (b) holds, it clearly suffices to focus on the possible arrangements of the $j$-indices. To see this, note that by~\eqref{eqn:bnd_cumY} we can ignore sets that don't contain at least one $j$-index.

Before we discuss all possible arrangements of $j$-indices in the indecomposable partitions $\nu$ with $|\nu_r| \geq 2$, we now explain how we show that (a) or (b) holds: To show that (a) holds, for a given partition $\nu := \{\nu_1, \ldots, \nu_R\}$, it suffices to \emph{identify a set $\nu_{\ell}$ in $\nu$ that contains at least one $j_1$-index, at least one $j_2$-index and at least one non-$j$-index.} We then have $|\nu_{\ell}| =: p \geq 3$. Say that the elements in $\nu_{\ell} := \{(u_1, v_1), \ldots, (u_p, v_p)\}$ are ordered such that a $j_1$-index is the first, a $j_2$-index is second and a non-$j$-index is $p$th. Then, by \eqref{eqn:bnd_cumY} and the stationarity of $\seq{\vO{X}{t}}$ we have
\begin{equation*}
\begin{split}
& S(\nu) \leq \Bigg( \prod_{\substack{r=1\\r\neq\ell}}^R B(\nu_r) \Bigg) \sum_{j_1 = -\infty}^{\infty} \sum_{j_2 = -\infty}^{\infty} |\cum ( \cO{X}{v_1}{j_1+h_1-s}, \cO{X}{v_2}{j_2+h_2-s}, \ldots, \cO{X}{v_p}{0} )| \\
& \leq \Bigg( \prod_{\substack{r=1\\r\neq\ell}}^R B(\nu_r) \Bigg) \sum_{i_1 = -\infty}^{\infty} \ldots\sum_{i_{p-1} = -\infty}^{\infty} |\cum ( \cO{X}{v_1}{i_1}, \ldots, \cO{X}{v_{p-1}}{i_{p-1}}, \cO{X}{v_p}{0} )| =: c_{\nu}, 
\end{split}
\end{equation*}
where $h_1, h_2 \in \{0,k\}$ (depending on $v_1$ and $v_2$), $s \in \{t, t+k\}$ (depending on $v_p$), and the $\ldots$ in the first line represents $|\nu_{\ell}|-3$ variables from Table~\ref{tab2}(a) with their time index shifted by $s$.

Now, to show that (b) holds, for a given partition $\nu := \{\nu_1, \ldots, \nu_R\}$, it suffices to \emph{identify sets $\nu_{\ell_1}$ and $\nu_{\ell_2}$ that satisfy the following conditions}:
\begin{itemize}
	\item[(I)] $\nu_{\ell_1} := \{(u_{1,1}, v_{1,1}), \ldots, (u_{1,p_1}, v_{1,p_1})\}$ contains at least one $j_1$-index;
	\item[(II)] $\nu_{\ell_2}  := \{(u_{2,1}, v_{2,1}), \ldots, (u_{2,p_2}, v_{2,p_2})\}$ contains at least one $j_2$-index;
	\item[(III)] either $\nu_{\ell_1}$ or $\nu_{\ell_2}$ contains at least one non-$j$-index;
	\item[(IV)] if $\nu_{\ell_1}$ contains both $j_1$-indices, then it must contain at least three elements and also if $\nu_{\ell_2}$ contains both $j_2$-indices, then it must contain at least three elements.
\end{itemize}
Note that for any indecomposable partition, (IV) is always satisfied, because a partition with a set of exactly two elements that are $j_1$-indices or exactly two elements that are $j_2$-indices is decomposable.

First, consider the case where the non-$j$-index which we know exists due to (III) is in~$\nu_{\ell_1}$.
The case where it is in $\nu_{\ell_2}$ can be treated analogously.
We can number the sets and indices such that $(u_{1,1}, v_{1,1})$ is a $j_1$-index, $(u_{2,1}, v_{2,1})$ is a $j_2$-index, and $(u_{1,p_1}, v_{1,p_1})$ is a non-$j$-index.
Further, due to (IV) and $|\nu_{\ell_2}| \geq 2$, $(u_{2,p_2}, v_{2,p_2})$ is not a $j_2$-index.
Thus, for $\nu_{\ell_1}$ and $\nu_{\ell_2}$ satisfying (I)--(IV) we have
\begin{equation*}
\begin{split}
& S(\nu) \leq \Bigg( \prod_{\substack{r=1\\r\notin\{\ell_1, \ell_2\}}}^R B(\nu_r) \Bigg) \sum_{j_1 = -\infty}^{\infty} \sum_{j_2 = -\infty}^{\infty} |\cum ( \cO{X}{v_{1,1}}{j_1+h_1-s_1}, \ldots, \cO{X}{v_{1,p_1}}{0} )| \\
& \hspace{5.7cm} \times |\cum ( \cO{X}{v_{2,1}}{j_2+h_2-s_2}, \ldots, \cO{X}{v_{2,p_2}}{0} )| \\
& \leq \Bigg( \prod_{\substack{r=1\\r\notin\{\ell_1, \ell_2\}}}^R B(\nu_r) \Bigg) \sum_{i_{1,1} = -\infty}^{\infty} \ldots\sum_{i_{1,p_1-1} = -\infty}^{\infty} |\cum ( \cO{X}{v_{1,1}}{i_{1,1}}, \ldots, \cO{X}{v_{1,p_1-1}}{i_{1,p_1-1}}, \cO{X}{v_{1,p_1}}{0} )| \\
& \hspace{0.3cm} \times  \sum_{i_{2,1} = -\infty}^{\infty} \ldots\sum_{i_{2,p_2-1} = -\infty}^{\infty} |\cum ( \cO{X}{v_{2,1}}{i_{2,1}}, \ldots, \cO{X}{v_{2,p_2-1}}{i_{2,p_2-1}}, \cO{X}{v_{2,p_2}}{0} )| =: c_{\nu}, 
\end{split}
\end{equation*}
where $h_1, h_2 \in \{0,k\}$ (depending on the values of $v_{1,1}$ and $v_{2,1}$), $s_1 \notin \{j_1, j_1+k, j_2, j_2 + k\}$, and $s_2 \notin \{j_2, j_2+k\}$.
For the second inequality we have substituted $j_1 + h_1-s_1$ by $i_{1,1}$, which corresponds to a shift in the index of the $j_1$-sum, since $(u_{1,p_1}, v_{1,p_1})$ is a non-$j$-index, and then bounded in two steps
\begin{multline}\label{bnd:j1_cum}
|\cum ( \cO{X}{v_{1,1}}{i_{1,1}}, \ldots, \cO{X}{v_{1,p_1}}{0} )| \\
\leq \sum_{i_{1,2} = -\infty}^{\infty} \ldots\sum_{i_{1,p_1-1} = -\infty}^{\infty} |\cum ( \cO{X}{v_{1,1}}{i_{1,1}}, \ldots, \cO{X}{v_{1,p_1-1}}{i_{1,p_1-1}}, \cO{X}{v_{1,p_1}}{0} )|,
\end{multline}
where the right-hand side in~\eqref{bnd:j1_cum} does not depend on $j_2$ (anymore).
Then, $S(\nu) \leq c_{\nu}$ follows from
\begin{multline*}
\sum_{j_2 = -\infty}^{\infty} |\cum ( \cO{X}{v_{2,1}}{j_2+h_2-s_2}, \ldots, \cO{X}{v_{2,p_2}}{0} )| \\
\leq \sum_{i_{2,1} = -\infty}^{\infty} \ldots\sum_{i_{2,p_2-1} = -\infty}^{\infty} |\cum ( \cO{X}{v_{2,1}}{i_{2,1}}, \ldots, \cO{X}{v_{2,p_2-1}}{i_{2,p_2-1}}, \cO{X}{v_{2,p_2}}{0} )|,
\end{multline*}
which we have since $(u_{2,p_2}, v_{2,p_2})$ is not a $j_2$-index.
To conclude the proof it remains to discuss all possible arrangements of the $j$-indices.
To cover all cases, we organise the following according to the number of sets in the partition that contain at least one $j$-index.
\begin{enumerate}
	\item There exists one $\nu_{\ell}$ in the partition that contains all $j$-indices. In this case $\nu_{\ell}$ has to also contain at least one more non-$j$-index, because otherwise the partition would be decomposable. Thus, from the above, we have that (a) holds.
	
	\item There exist two sets $\nu_{\ell_1}$ and $\nu_{\ell_2}$ such that each contains at least one $j$-index and their union contains all $j$-indices. The situation where $|\nu_{\ell_1}| = |\nu_{\ell_2}| = 2$ would imply a decomposable partition and is therefore not possible. Thus, we either have $|\nu_{\ell_1}| \geq 3$ and $|\nu_{\ell_2}| \geq 2$ or $|\nu_{\ell_1}| \geq 2$ and $|\nu_{\ell_2}| \geq 3$. Two sub-cases are possible: (i) one of the sets contains three $j$-indices and the other contains only one $j$-index, or (ii) $\nu_{\ell_1}$ and $\nu_{\ell_2}$ both contains exactly two $j$-indices.
	
	If (i), then the set containing only one $j$-index has to also contain at least one non-$j$-index (i.\,e., (III) is satisfied). Then (b) holds, since if the set with only one $j$-index contains, say, a $j_1$-index then the other set contains a $j_2$-index (or vice versa; i.\,e., (I) and (II) are satisfied). 
	
	If (ii), then the set with at least three elements contains a non-$j$-index (i.\,e., (III) is satisfied). Further, since each set contains exactly two $j$-indices and there are exactly two $j_1$-indices and two $j_2$-indices we have that if one set contains a $j_1$-index, the other has to contain a $j_2$-index (i.\,e., (I) and (II) are satisfied). Thus, (b) holds.
	
	\item There exist a set with exactly two $j$-indices and two sets with exactly one $j$-index each. Then, because of $|\nu_r| \geq 2$, the two sets with one $j$-index also contain at least one non-$j$-indices. Two sub-cases are possible: the $j$-indices in the sets with exactly one $j$-index are either (i) both $j_1$-indices or both $j_2$-indices, or (ii) we have an $j_1$-index in one and a $j_2$-index in the other set.
	
	If (i), then the $j$-indices in the set with exactly two $j$-indices are either both $j_2$-indices or both $j_1$-indices, respectively. The set with exactly two $j$-indices (either two $j_1$- or two $j_2$-indices) then has to have another non-$j$-index, as the partition would otherwise be decomposable. Taking $\nu_{\ell_1}$ as the set with two $j$-indices and $\nu_{\ell_2}$ as one of the two sets with exactly one $j$-index, we see that (b) holds: the indices are $j_1$- and $j_2$ indices in the two sets (i.\,e., (I) and (II) are satisfied) and each set has at least one non-$j$-index (i.\,e., (III) is satisfied).
	
	If (ii), then we take $\nu_{\ell_1}$ as one of the sets with exactly one $j_1$-index and $\nu_{\ell_2}$ as the set with exactly one $j_2$-index (i.\,e., (I) and (II) are satisfied). Then (b) is satisfied, as each of these sets also has one non-$j$-index (i.\,e., (III) is satisfied).
	
	\item There exist four sets with exactly one $j$-index each. We take $\nu_{\ell_1}$ as one of the sets with exactly one $j_1$-index and $\nu_{\ell_2}$ as the set with exactly one $j_2$-index (i.\,e., (I) and (II) are satisfied). Each of these sets also has one non-$j$-index (i.\,e., (III) is satisfied), thus (b) is satisfied.
\end{enumerate}
This finishes the proof of~\eqref{exist_C_nu} in the main paper and also clarified how to find $C_{\nu}$ in terms of the sum in~\eqref{eqn:sum_cum} and the quantities $B(\nu_r)$ defined in~\eqref{eqn:bnd_cumY}, which are all independent of~$t$.

\hfill$\square$

\section{Additional tables for the examples in the paper}\label{sec:adSimRes}
In this section we provide the values of the bound and of the true 1-Wasserstein distances considered in Theorem~\ref{general_theorem} for the case of the example in Section~\ref{sec:expl:AR1} in the following two scenarios:
$\varepsilon(t) \sim \mathcal{N}(0,1)$, and $\varepsilon(t) \sim \sqrt{12/14} t_{14}$.

\input{"tables/bound_N01.tex"}
\input{"tables/W1_N01.tex"}

\input{"tables/bound_t_14.tex"}
\input{"tables/W1_t_14.tex"}

\end{document}

%% file: tables/bound_N01.tex
% latex table generated in R 4.1.0 by xtable 1.8-4 package
% Mon Jun 14 18:07:25 2021
\begin{table}[H]
\caption{Value of the bound from Theorem~\ref{general_theorem} in combination with \eqref{bnd:Smn_new_bound}, with $m = m^*$ to minimise the bound as described in Section~\ref{sec:comp_bnd}, for empirical autocovariances, for a range of lags $k$ and sample sizes $n$. The data stems from an AR(1) process with $\varepsilon_t \sim \mathcal{N}(0,1)$ where $\alpha$ takes a range of values. \label{tab:bnd_N01}}
\centering
\begin{tabular}{ll |r@{\extracolsep{0.3cm}}r@{\extracolsep{0.3cm}}r@{\extracolsep{0.1cm}}r@{\extracolsep{0.3cm}}r@{\extracolsep{0.3cm}}r@{\extracolsep{0.3cm}}r@{\extracolsep{0.2cm}}r@{\extracolsep{0.1cm}}r@{\extracolsep{0.2cm}}r}
  \hline
 $k$ & $\alpha$ $\vert$ $n$ & \multicolumn{1}{l}{25} & \multicolumn{1}{l}{50} & \multicolumn{1}{l}{75} & \multicolumn{1}{l}{100} & \multicolumn{1}{l}{150} & \multicolumn{1}{l}{    200} & \multicolumn{1}{l}{250} & \multicolumn{1}{l}{500} & \multicolumn{1}{l}{1000} & \multicolumn{1}{l}{2000} \\ 
   \hline
0 & 0            &  0.912 &  0.645 &  0.527 &  0.456 &  0.372 &  0.322 &  0.288 &  0.204 &  0.144 &  0.102 \\ 
    & 0.1          &  4.438 &  3.784 &  3.219 &  2.894 &  2.533 &  2.340 &  2.222 &  1.661 &  1.239 &  0.967 \\ 
    & 0.3          &  7.005 &  5.701 &  5.033 &  4.673 &  3.984 &  3.585 &  3.324 &  2.564 &  1.986 &  1.488 \\ 
    & 0.5          &  9.842 &  8.080 &  7.164 &  6.484 &  5.676 &  5.110 &  4.715 &  3.650 &  2.798 &  2.128 \\ 
    & 0.7          & 13.981 & 11.712 & 10.379 &  9.485 &  8.289 &  7.511 &  6.931 &  5.375 &  4.124 &  3.137 \\ 
  1 & 0            &  2.564 &  1.818 &  1.485 &  1.286 &  1.050 &  0.909 &  0.813 &  0.574 &  0.406 &  0.287 \\ 
    & 0.1          &  6.778 &  5.045 &  4.259 &  3.801 &  3.279 &  2.988 &  2.752 &  1.996 &  1.477 &  1.135 \\ 
    & 0.3          &  9.065 &  6.998 &  6.105 &  5.498 &  4.658 &  4.170 &  3.849 &  2.907 &  2.226 &  1.648 \\ 
    & 0.5          & 11.388 &  9.075 &  7.926 &  7.126 &  6.179 &  5.522 &  5.084 &  3.886 &  2.947 &  2.221 \\ 
    & 0.7          & 14.561 & 12.279 & 10.808 &  9.820 &  8.531 &  7.699 &  7.091 &  5.452 &  4.152 &  3.138 \\ 
  2 & 0            &  4.088 &  2.916 &  2.385 &  2.067 &  1.688 &  1.462 &  1.308 &  0.924 &  0.653 &  0.462 \\ 
    & 0.1          &  8.272 &  6.135 &  5.159 &  4.582 &  3.917 &  3.540 &  3.249 &  2.348 &  1.726 &  1.312 \\ 
    & 0.3          & 10.061 &  7.882 &  6.848 &  6.167 &  5.216 &  4.658 &  4.288 &  3.197 &  2.429 &  1.792 \\ 
    & 0.5          & 12.843 & 10.204 &  8.804 &  7.902 &  6.803 &  6.068 &  5.575 &  4.234 &  3.196 &  2.400 \\ 
    & 0.7          & 15.174 & 13.119 & 11.527 & 10.442 &  9.040 &  8.137 &  7.485 &  5.726 &  4.342 &  3.270 \\ 
   \hline
\end{tabular} 
\end{table}

%% file: tables/W1_N01.tex
% latex table generated in R 4.1.0 by xtable 1.8-4 package
% Mon Jun 14 18:07:25 2021
\begin{table}[H]
\caption{Value of the true 1-Wasserstein distance considered in Theorem~\ref{general_theorem} for empirical autocovariances, for a range of lags $k$ and sample sizes $n$. The data stems from an AR(1) process with $\varepsilon_t \sim \mathcal{N}(0,1)$ where $\alpha$ takes a range of values. \label{tab:W1_N01}}
\centering
\begin{tabular}{ll |rrrrrrrrrr}
  \hline
 $k$ & $\alpha$ $\vert$ $n$  & \multicolumn{1}{l}{    25} & \multicolumn{1}{l}{    50} & \multicolumn{1}{l}{    75} & \multicolumn{1}{l}{   100} & \multicolumn{1}{l}{   150} & \multicolumn{1}{l}{   200} & \multicolumn{1}{l}{   250} & \multicolumn{1}{l}{   500} & \multicolumn{1}{l}{  1000} & \multicolumn{1}{l}{  2000} \\ 
   \hline
0 & 0            & 0.129 & 0.091 & 0.074 & 0.065 & 0.053 & 0.046 & 0.041 & 0.029 & 0.020 & 0.014 \\ 
    & 0.1          & 0.135 & 0.096 & 0.078 & 0.068 & 0.055 & 0.048 & 0.043 & 0.030 & 0.021 & 0.015 \\ 
    & 0.3          & 0.191 & 0.136 & 0.112 & 0.097 & 0.080 & 0.069 & 0.062 & 0.044 & 0.031 & 0.022 \\ 
    & 0.5          & 0.360 & 0.261 & 0.215 & 0.187 & 0.153 & 0.133 & 0.119 & 0.084 & 0.060 & 0.042 \\ 
    & 0.7          & 0.970 & 0.717 & 0.595 & 0.519 & 0.428 & 0.372 & 0.333 & 0.237 & 0.168 & 0.119 \\ 
  1 & 0            & 0.051 & 0.026 & 0.018 & 0.014 & 0.009 & 0.007 & 0.005 & 0.003 & 0.001 & 0.001 \\ 
    & 0.1          & 0.079 & 0.052 & 0.041 & 0.035 & 0.028 & 0.024 & 0.022 & 0.015 & 0.011 & 0.008 \\ 
    & 0.3          & 0.209 & 0.149 & 0.122 & 0.106 & 0.087 & 0.075 & 0.067 & 0.048 & 0.034 & 0.024 \\ 
    & 0.5          & 0.440 & 0.317 & 0.260 & 0.226 & 0.185 & 0.161 & 0.144 & 0.102 & 0.072 & 0.051 \\ 
    & 0.7          & 1.132 & 0.828 & 0.684 & 0.596 & 0.490 & 0.425 & 0.381 & 0.271 & 0.192 & 0.136 \\ 
  2 & 0            & 0.062 & 0.032 & 0.022 & 0.016 & 0.011 & 0.008 & 0.007 & 0.003 & 0.002 & 0.001 \\ 
    & 0.1          & 0.067 & 0.035 & 0.024 & 0.019 & 0.013 & 0.010 & 0.008 & 0.005 & 0.003 & 0.002 \\ 
    & 0.3          & 0.144 & 0.098 & 0.078 & 0.067 & 0.054 & 0.047 & 0.042 & 0.029 & 0.021 & 0.015 \\ 
    & 0.5          & 0.406 & 0.293 & 0.241 & 0.210 & 0.172 & 0.149 & 0.133 & 0.095 & 0.067 & 0.047 \\ 
    & 0.7          & 1.206 & 0.881 & 0.727 & 0.633 & 0.520 & 0.451 & 0.404 & 0.287 & 0.203 & 0.144 \\ 
   \hline
\end{tabular} 
\end{table}

%% file: tables/bound_t_14.tex
% latex table generated in R 4.1.0 by xtable 1.8-4 package
% Mon Jun 14 18:07:25 2021
\begin{table}[H]
\caption{Value of the bound from Theorem~\ref{general_theorem} in combination with \eqref{bnd:Smn_new_bound}, with $m = m^*$ to minimise the bound as described in Section~\ref{sec:comp_bnd}, for empirical autocovariances, for a range of lags $k$ and sample sizes $n$. The data stems from an AR(1) process with $\varepsilon_t \sim \sqrt{12/14} \ t_{14}$ where $\alpha$ takes a range of values. \label{tab:bnd_t14}}
\centering
\begin{tabular}{ll |r@{\extracolsep{0.3cm}}r@{\extracolsep{0.3cm}}r@{\extracolsep{0.1cm}}r@{\extracolsep{0.3cm}}r@{\extracolsep{0.3cm}}r@{\extracolsep{0.3cm}}r@{\extracolsep{0.2cm}}r@{\extracolsep{0.1cm}}r@{\extracolsep{0.2cm}}r}
	\hline
	$k$ & $\alpha$ $\vert$ $n$  & \multicolumn{1}{l}{25} & \multicolumn{1}{l}{50} & \multicolumn{1}{l}{75} & \multicolumn{1}{l}{100} & \multicolumn{1}{l}{150} & \multicolumn{1}{l}{    200} & \multicolumn{1}{l}{250} & \multicolumn{1}{l}{500} & \multicolumn{1}{l}{1000} & \multicolumn{1}{l}{2000} \\ 
   \hline
0 & 0            &  0.912 &  0.645 &  0.527 &  0.456 &  0.372 &  0.322 &  0.288 &  0.204 &  0.144 &  0.102 \\ 
    & 0.1          &  6.056 &  5.796 &  4.892 &  4.349 &  3.727 &  3.378 &  3.153 &  2.500 &  1.842 &  1.402 \\ 
    & 0.3          &  9.007 &  7.383 &  6.423 &  5.888 &  5.109 &  4.573 &  4.217 &  3.270 &  2.512 &  1.905 \\ 
    & 0.5          & 11.020 &  9.007 &  7.946 &  7.205 &  6.294 &  5.672 &  5.235 &  4.064 &  3.136 &  2.411 \\ 
    & 0.7          & 14.064 & 11.794 & 10.500 &  9.621 &  8.441 &  7.673 &  7.099 &  5.554 &  4.312 &  3.335 \\ 
  1 & 0            &  2.564 &  1.818 &  1.485 &  1.286 &  1.050 &  0.909 &  0.813 &  0.574 &  0.406 &  0.287 \\ 
    & 0.1          &  7.224 &  5.339 &  4.505 &  4.017 &  3.460 &  3.147 &  2.928 &  2.135 &  1.589 &  1.229 \\ 
    & 0.3          &  8.995 &  6.966 &  6.092 &  5.490 &  4.675 &  4.202 &  3.890 &  2.950 &  2.274 &  1.719 \\ 
    & 0.5          & 11.204 &  8.934 &  7.829 &  7.068 &  6.139 &  5.515 &  5.099 &  3.934 &  3.023 &  2.320 \\ 
    & 0.7          & 14.097 & 11.893 & 10.569 &  9.656 &  8.449 &  7.658 &  7.089 &  5.518 &  4.267 &  3.290 \\ 
  2 & 0            &  4.088 &  2.916 &  2.385 &  2.067 &  1.688 &  1.462 &  1.308 &  0.924 &  0.653 &  0.462 \\ 
    & 0.1          &  9.329 &  6.892 &  5.780 &  5.121 &  4.358 &  3.922 &  3.583 &  2.587 &  1.898 &  1.435 \\ 
    & 0.3          & 10.481 &  8.065 &  6.983 &  6.240 &  5.282 &  4.721 &  4.347 &  3.262 &  2.485 &  1.847 \\ 
    & 0.5          & 12.334 &  9.736 &  8.425 &  7.589 &  6.562 &  5.879 &  5.421 &  4.152 &  3.166 &  2.408 \\ 
    & 0.7          & 14.490 & 12.447 & 10.980 & 10.019 &  8.750 &  7.917 &  7.322 &  5.676 &  4.369 &  3.354 \\ 
   \hline
\end{tabular} 
\end{table}

%% file: tables/W1_t_14.tex
% latex table generated in R 4.1.0 by xtable 1.8-4 package
% Mon Jun 14 18:07:25 2021
\begin{table}[H]
\caption{Value of the true 1-Wasserstein distance considered in Theorem~\ref{general_theorem} for empirical autocovariances, for a range of lags $k$ and sample sizes $n$. The data stems from an AR(1) process with $\varepsilon_t \sim \sqrt{12/14} \ t_{14}$ where $\alpha$ takes a range of values. \label{tab:W1_t14}}
\centering
\begin{tabular}{ll |rrrrrrrrrr}
  \hline
 $k$ & $\alpha$ $\vert$ $n$  & \multicolumn{1}{l}{    25} & \multicolumn{1}{l}{    50} & \multicolumn{1}{l}{    75} & \multicolumn{1}{l}{   100} & \multicolumn{1}{l}{   150} & \multicolumn{1}{l}{   200} & \multicolumn{1}{l}{   250} & \multicolumn{1}{l}{   500} & \multicolumn{1}{l}{  1000} & \multicolumn{1}{l}{  2000} \\ 
   \hline
0 & 0            & 0.207 & 0.151 & 0.125 & 0.109 & 0.090 & 0.078 & 0.070 & 0.050 & 0.036 & 0.025 \\ 
    & 0.1          & 0.213 & 0.155 & 0.128 & 0.112 & 0.092 & 0.080 & 0.072 & 0.051 & 0.037 & 0.026 \\ 
    & 0.3          & 0.270 & 0.197 & 0.163 & 0.142 & 0.117 & 0.102 & 0.091 & 0.065 & 0.046 & 0.033 \\ 
    & 0.5          & 0.444 & 0.325 & 0.270 & 0.235 & 0.194 & 0.168 & 0.151 & 0.108 & 0.076 & 0.054 \\ 
    & 0.7          & 1.072 & 0.798 & 0.664 & 0.581 & 0.479 & 0.417 & 0.375 & 0.267 & 0.189 & 0.134 \\ 
  1 & 0            & 0.062 & 0.034 & 0.023 & 0.018 & 0.012 & 0.009 & 0.007 & 0.004 & 0.002 & 0.001 \\ 
    & 0.1          & 0.092 & 0.061 & 0.048 & 0.041 & 0.033 & 0.028 & 0.025 & 0.018 & 0.012 & 0.009 \\ 
    & 0.3          & 0.234 & 0.169 & 0.139 & 0.121 & 0.099 & 0.086 & 0.077 & 0.055 & 0.039 & 0.027 \\ 
    & 0.5          & 0.483 & 0.351 & 0.289 & 0.252 & 0.207 & 0.179 & 0.161 & 0.114 & 0.081 & 0.057 \\ 
    & 0.7          & 1.207 & 0.888 & 0.736 & 0.643 & 0.529 & 0.460 & 0.412 & 0.293 & 0.208 & 0.147 \\ 
  2 & 0            & 0.074 & 0.039 & 0.027 & 0.020 & 0.014 & 0.010 & 0.008 & 0.004 & 0.002 & 0.001 \\ 
    & 0.1          & 0.078 & 0.042 & 0.029 & 0.022 & 0.016 & 0.012 & 0.010 & 0.005 & 0.003 & 0.002 \\ 
    & 0.3          & 0.156 & 0.106 & 0.085 & 0.073 & 0.059 & 0.051 & 0.045 & 0.032 & 0.022 & 0.016 \\ 
    & 0.5          & 0.428 & 0.311 & 0.257 & 0.224 & 0.184 & 0.159 & 0.143 & 0.101 & 0.072 & 0.051 \\ 
    & 0.7          & 1.257 & 0.924 & 0.765 & 0.667 & 0.549 & 0.477 & 0.427 & 0.304 & 0.215 & 0.152 \\ 
   \hline
\end{tabular} 
\end{table}